\documentclass[a4paper,10pt]{article}

\usepackage{float}

\usepackage{pgffor}
\usepackage{tikz}
\usetikzlibrary{calc,3d,arrows,positioning,shapes.misc}

\usepackage[english]{babel}
\usepackage{amsthm}
\usepackage[intlimits]{amsmath}
\usepackage{amssymb}
\usepackage[latin1]{inputenc}
\usepackage{setspace}
\usepackage{enumerate}

\usepackage[colorlinks=true,linkcolor=black,citecolor=black,filecolor=black,urlcolor=black,menucolor=black]{hyperref}

\usepackage{dsfont}
\usepackage{array}
\usepackage{graphicx}
\usepackage[babel]{csquotes}
\usepackage{geometry}
\usepackage{bbm}
\usepackage{stmaryrd}
\usepackage[all]{xy}
\usepackage{mathrsfs}

\theoremstyle{plain}
\newtheorem{thm}{Theorem}[section]
\newtheorem{prop}[thm]{Proposition}
\newtheorem{cor}[thm]{Corollary}
\newtheorem{lem}[thm]{Lemma}
\newtheorem{algo}{Algorithm}
\theoremstyle{definition}
\newtheorem{defn}[thm]{Definition}
\newtheorem{expl}[thm]{Example}
\newtheorem{rem}[thm]{Remark}

\newcommand{\step}[1]{\par\medskip\par\noindent\textit{#1}}

\def \B{\mathscr{B}}
\def \E{\mathscr{E}}
\def \EE{\mathscr{E}^2}

\def \h{\sqrt{h}}
\def \H{\mathscr{H}}
\def \J{\mathscr{J}}
\def \L{\mathscr{L}}
\def \N{\mathbb{N}}
\def \R{\mathbb{R}}
\def \Z{\mathbb{Z}}
\newcommand {\sign} {\mathop \mathrm{sign}} 

\def \ae{\text{a.e.\ }}
\def \torus{{[0,\Lambda)^d}}
\newcommand {\lip} {\mathop \mathrm{Lip}}
\def \sphere{{{S}^{d-1}}}
\def \numphases{P}
\newcommand {\supp} {\mathop \textup{supp}}
\def \chara{\mathbf{1}}
\def \gap{\varepsilon^2} 
\def \ggap{\varepsilon^2} 

\def \kmax{K}
\def \lmax{L}
\def \nmax{N}
\def \kmean{\frac1\kmax \sum_{k=1}^\kmax}

\def \nmean{\frac1\nmax \sum_{n=1}^\nmax}

\title{Convergence of the thresholding scheme for \\multi-phase mean-curvature flow}
\author{Tim Laux\footnote{Max-Planck-Institut f\"ur Mathematik in den Naturwissenschaften, Inselstra{\ss}e 22, 04103 Leipzig, Germany. Please use {tim.laux@mis.mpg.de} for correspondence.  } 
\and Felix Otto\footnotemark[1]}
\date{\today}
\usepackage[backend=bibtex]{biblatex}
\bibliography{lit.bib}
\DefineBibliographyStrings{english}{%
  references = {Bibliography},
}

\begin{document}

  \maketitle
  \begin{abstract}
We consider the thresholding scheme, a time discretization for mean curvature flow
introduced by Merriman, Bence and Osher. 
We prove a convergence result in the multi-phase case.
The result establishes convergence towards a weak formulation of mean curvature
flow in the BV-framework of sets of finite perimeter.
The proof is based on the interpretation of the thresholding scheme
as a minimizing movements scheme by Esedo\u{g}lu et.\ al..
This interpretation means that the thresholding scheme preserves
the structure of (multi-phase) mean curvature flow as a gradient flow 
w.\ r.\ t.\ the total interfacial energy. 
More precisely, the thresholding scheme is a minimizing movements scheme for an
energy functional that $\Gamma$-converges to the total interfacial energy.
In this sense, our proof is similar to the convergence results of Almgren, Taylor and Wang and
Luckhaus and Sturzenhecker, which establish convergence of a more academic
minimizing movements scheme. Like the one of Luckhaus and Sturzenhecker, 
ours is a {\it conditional} convergence
result, which means that we have to assume that the
time-integrated energy of the approximation converges to
the time-integrated energy of the limit. This is a natural assumption, which however
is not ensured by the compactness coming from the basic estimates.

\medskip


\medskip

\end{abstract}

\section{Introduction}
\subsection{Context}
The thresholding scheme, a time discretization for mean curvature flow
introduced by Merriman, Bence and Osher \cite{MBO92}, has because of its conceptual 
and practical simplicity become a very popular scheme, see Algorithm \ref{MBO_arbitrary_surface_tensions}
for its definition in a more general context.
It has a natural extension from the two-phase case to the multi-phase case with 
triple junctions in local equilibrium, well-known in case for equal surface tensions
since some time \cite{MBO94}.
Multi-phase mean-curvature flow models the slow relaxation of grain boundaries
in polycrystals (called grain growth), where each grain corresponds to a phase.
Elsey, Esedo\u{g}lu and Smereka have shown that (a modification of) the thresholding scheme 
is practical in handling a large number of grains over time intervals sufficiently large
to extract significant statistics of the coarsening (also called aging)
of the grain configuration \cite{ elsey2009diffusion, elsey2011large, elsey2011large2}.
In grain growth, the surface tension (and the mobility)
of a grain boundary is both dependent on the misorientation between the crystal lattice
of the two adjacent grains and on the orientation of its normal. In other words,
the surface tension $\sigma_{ij}$ of an interface is indexed by the pair $(i,j)$ of phases it separates,
and is anisotropic. Esedo\u{g}lu and the second
author have shown in \cite{EseOtt14} the thresholding scheme can be extended to handle the first
extension in a very general way, including in particular the most popular Ansatz for
a misorientation-dependent grain boundary energy \cite{read1950dislocation}.
How to handle general anisotropies in the framework of the thresholding scheme, even in case of two phases, 
seems not yet to be completely settled, see however \cite{bonnetier2012consistency} and 
\cite{ishii1999threshold}.
--- Hence in this work, we will focus on the isotropic case, ignore mobilities, 
but make the attempt to be as general as \cite{EseOtt14} when it comes to the
dependence of $\sigma_{ij}$ on the pair $(i,j)$. 

\medskip

In the two-phase case, the convergence of the thresholding scheme is well-understood:
Two-phase mean curvature flow satisfies a geometric {\it comparison principle}, and it
is easy to see that the thresholding scheme preserves this structure. Partial differential
equations and geometric motions that allow for a comparison principle can typically be even
characterized by comparison with very simple solutions, which opens the way for a definition of a very robust notion
of weak solutions, namely what bears the somewhat misleading name of viscosity solutions. If one allows for 
what the experts know as fattening,
two-phase mean-curvature flow is well-posed in this framework \cite{evans1991motion}.
Barles and Georgelin \cite{barles1995simple} and Evans \cite{evans1993convergence} proved
independently that the thresholding scheme converges to mean-curvature flow in this sense.
--- Hence the main novelty of this work is a (conditional) convergence result in
the {\it multi-phase} case; where clearly a geometric comparison principle is absent. However,
the result has also some interest in the two-phase case, since it establishes convergence
even in situations where the viscosity solution features fattening. Together with Drew Swartz 
\cite{LauSwa15},
the first author uses similar arguments to treat another version of mean curvature flow
that does not even allow for a comparison principle in the two-phase case, namely volume-preserving mean-curvature
flow. They prove (conditioned) convergence of a scheme introduced 
by Ruuth and Wetton in \cite{RutWet03}.
We also draw the reader's attention to the recent work of Mugnai, Seis and Spadaro \cite{mugnai2015global},
where they prove a (conditional) convergence result as in \cite{LucStu95} 
of a modification of the scheme in \cite{ATW93, LucStu95} to volume-preserving
mean curvature flow.
Note that due to the only conditional convergence, our result does not provide a long-time existence
result for (weak solutions of) multi-phase mean curvature flow.
Short-time existence results of smooth solutions go back to the work of Bronsard and Reitich 
\cite{bronsard1993three}. 
Mantegazza et.\ al. \cite{mantegazza2003motion} and Schn\"urer et.\ al. \cite{schnurer2011evolution}
were able to construct long-time solutions close to a self-similar singularity.
\medskip

For the present work, the structural substitute for the comparison principle is
the {\it gradient flow} structure. Folklore says that mean curvature flow, also in its
multi-phase version, is the gradient flow of the total interfacial energy. It is
by now well-appreciated that the gradient flow structure also requires fixing a
Riemannian structure, that is, an inner product on the tangent space, which here
is given by the space of normal velocities. Mean curvature flow is then the gradient flow with
respect to the $L^2$-inner product, in case of grain growth weighted by grain-boundary-dependent
and anisotropic mobilities. Loosely speaking, Brakke's existence proof
in the framework of varifolds \cite{brakke1978motion} is based on this structure in the sense
that the solution monitors weighted versions of the interfacial energy. 
Recently, Kim and Tonegawa \cite{kim2015mean} improved this work by deriving the continuity of the volumes of the grains
in the case of grain growth with equal surface tensions which ensures that the solution is non-trivial.
Also Ilmanen's convergence proof of the Allen-Cahn equation, a diffuse interface approximation
of computational relevance in the world of phase-field models, to mean curvature flow makes use of the gradient flow structure 
\cite{ilmanen1993convergence}. It was only discovered recently that the thresholding algorithm
preserves also this gradient flow structure \cite{EseOtt14}, which in that paper
was taken as a guiding principle to extend the scheme to surface tensions $\sigma_{ij}$ and mobilities
that depend on the phase pair $(i,j)$. --- In this paper, we take the gradient flow structure, which we make
more precise in the following paragraphs, as a guiding principle for the convergence proof. 

\medskip

On the abstract level, every gradient flow has a natural
discretization in time, which comes in form of a sequence of variational problems:
The configuration $\Sigma^n$ at time step $n$ is obtained by minimizing 
$\frac{1}{2}{\rm dist}^2(\Sigma,\Sigma^{n-1})
+h E(\Sigma)$, where $\Sigma^{n-1}$ is the configuration at the preceding time step, $h$
is the time-step size and ${\rm dist}$ denotes the induced distance on the configuration
space endowed with the Riemannian structure. In the Euclidean case, the Euler-Lagrange equation
(i.\ e.\ the first variation) of this variational problem yields the implicit (or backwards)
Euler scheme. This variational scheme has been named ``minimizing movements'' by 
De Giorgi \cite{de1993new}, and has recently gained popularity because it allows
to interpret diffusion equations as gradient flows of an entropy functional w.\ r.\ t.\ the
Wasserstein metric (\cite{jordan1998variational}, see \cite{ambrosio2006gradient} for the abstract framework)
-- an otherwise unrelated problem. However, the formal Riemannian structure 
in case of mean curvature flow is completely degenerate: ${\rm dist}^2(\Sigma,\tilde\Sigma)$ 
as defined as the infimal energy of curves in configuration space that connect 
$\Sigma$ to $\tilde\Sigma$ vanishes identically, cf.\ \cite{michor2005vanishing}.

\medskip

Hence when formulating a minimizing movements scheme in case of mean curvature flow, 
one has to come up with a proxy for ${\rm dist}^2(\Sigma,\tilde\Sigma)$. This has been 
independently achieved by Almgren, Taylor and Wang \cite{ATW93} on the one side and Luckhaus and
Sturzenhecker \cite{LucStu95} on the other side of the Atlantic.
$\Sigma=\partial\Omega$ and $\tilde\Sigma=\partial\tilde\Omega$,
$2\int_{\Omega\triangle\tilde\Omega}d(x,\Sigma)dx$ is one possible substitute for ${\rm dist}^2(\Sigma,\tilde\Sigma)$
in the minimizing movements scheme, where $d(x,\Sigma)$ denotes the unsigned distance of the point $x$ to
the surface $\Sigma$ --- it is easy to see that to leading order in the proximity of
$\tilde\Omega$ to $\Omega$, 
this expression behaves as the metric tensor $\int_{\Sigma}V^2dx$, where $V$ is the normal velocity 
leading from $\Omega$ to
$\tilde\Omega$ in one unit time. Their work makes this point by proving that this minimizing movements scheme
converges to mean curvature flow. 
To be more precise, they consider a time-discrete solution $\{\Omega^n\}_n$ of the minimizing
movement scheme, interpolated as a piecewise constant function $\Omega^h$ in time and assume that
for a subsequence $h\downarrow0$, the time-dependent sets $\Omega^h$ 
converge to $\Omega$ in a stronger sense than the given compactness provides.
Almgren, Taylor and Wang assume that $\Sigma^h(t)$ converges to $\Sigma(t)$ in the Hausdorff distance
and show that $\Sigma$ solves the mean curvature flow equation in the above mentioned viscosity sense.
The argument was later substantially simplified by Chambolle and Novaga in \cite{chambolle2006convergence}.
Luckhaus and Sturzenhecker start from a weaker convergence assumption than the one in \cite{ATW93}:
They assume that for the finite time horizon $T$ under consideration,
$\int_0^T|\Sigma^h(t)|dt$ converges
to $\int_0^T|\Sigma(t)|dt$. Then they show that $\Omega$ evolves according to a weak formulation
of mean curvature flow, using a distributional formulation of mean curvature that is available
for sets of finite perimeter, see Definition \ref{def_motion_by_mean_curvature} for the 
multi-phase case of this formulation. 
Incidentally, weak-strong uniqueness of this formulation seems not to be understood -- even in the two-phase case.
Those are both only {\it conditional} convergence results: 
While the natural estimates
coming from the minimizing movements scheme, namely the uniform boundedness of $\sup_n|\Sigma^n|$
and $\sum_{n}2\int_{\Omega^n\triangle\Omega^{n+1}}d(x,\Sigma^n)dx$, are enough to ensure
$\int_0^T|\Omega^h(t)\triangle\Omega(t)|dt \to 0$ and $\int_0^T|\Sigma(t)|dt\le\liminf\int_0^T|\Sigma^h(t)|dt$,
they are not sufficient to yield $\limsup\int_0^T|\Sigma^h(t)|dt\le\int_0^T|\Sigma(t)|dt$ or even
the convergence of $\Sigma^h(t)$ to $\Sigma(t)$ in the Hausdorff distance.
--- Our result will be a conditional convergence result very much in the same sense as the one in 
\cite{LucStu95}
but it turns out that our convergence result for the thresholding scheme requires no regularity theory for
(almost) minimal surfaces, in contrast to the one of \cite{LucStu95} and is therefore not restricted to low spatial dimensions $d\leq 7$.
Although the time discretization scheme in \cite{ATW93, LucStu95} seems rather academic from a computational point of view,
it has been adapted for numerical simulations by Chambolle in \cite{chambolle2004algorithm}.
Nevertheless, even in that variant, in each step one has to compute a (signed) distance function and solve a convex optimization problem.

\medskip

Following \cite{EseOtt14}, we now explain in which sense the thresholding scheme may be considered
as a minimizing movements scheme for mean curvature flow.
Each step in Algorithm \ref{MBO_arbitrary_surface_tensions} is equivalent to minimizing a functional
of the form $E_h(\chi) - E_h(\chi-\chi^{n-1})$, where the functional $E_h$, defined below in
(\ref{def_Eh}), is an approximation to the total interfacial energy.
It is a little more subtle to see that the second term, $-E_h(\chi^n-\chi^{n-1})$, is comparable
to the metric tensor $\int_{\Sigma}V^2dx$.
The $\Gamma$-convergence of functionals of the type (\ref{def_Eh}) to the area functional has
a long history: For the two-phase case, cf.\ Alberti and Bellettini \cite{alberti1998non}
and Miranda et.\ al.\ \cite{miranda2007short}. The multi-phase case, also for arbitrary
surface tensions was investigated by Esedo\u{g}lu and the second author in \cite{EseOtt14}.
Incidentally, it is easy to see that $\Gamma$-convergence of the energy functionals is
not sufficient for the convergence of the corresponding gradient flows; Sandier and Serfaty
\cite{sandier2004gamma} have identified sufficient conditions on both the functional and the metric tensor
for this to be true.

\medskip

Identically, the approach of Saye and Sethian \cite{saye2011voronoi} for multi-phase evolutions can
also be seen as coming from the gradient flow structure when applied to mean-curvature flow with $\numphases$ phases.
More precisely, it can be understood as a time splitting of an $L^2$-gradient flow with an additional
phase whose volume is strongly penalized:
The first step is $(\numphases+1)$-phase gradient flow w.\ r.\ t.\ the total interfacial energy and the second step is
$(\numphases+1)$-gradient flow w.\ r.\ t.\ the volume penalization (so geometrical optics leading to the 
Voronoi construction).
\subsection{Informal summary of the proof}

We now give a summary of the main steps and ideas of the convergence proof.
In Section \ref{sec:comp}, we draw consequences from the basic estimate 
(\ref{energy_dissipation_estimate}) in a minimizing movements scheme,
like compactness, Proposition \ref{comp_prop}, coming from a uniform (integrated) modulus of continuity 
in space, Lemma \ref{comp_lem_z_shift}, and in time, Lemma \ref{comp_lem_tau_shift}.
We also draw the first consequence from the strengthened convergence (\ref{conv_ass})
in the case of equal surface tensions in Proposition \ref{lem_dtX<<DX}.
We strongly advise the reader to familiarize him- or herself with the argument for the
modulus of continuity in time, Lemma \ref{comp_lem_tau_shift},
since it is there that the mesoscopic time scale $\sqrt{h}$ appears
for the first time in a simple context before being used in Section \ref{sec:velocity}
in a more complex context.
In the same vein, the fudge factor $\alpha$ in the mesoscopic time scale $\alpha\sqrt{h}$, which
will be crucial in Section \ref{sec:velocity},
will first be introduced and used in the simple context when estimating
the normal velocity $V$ of the limit in Proposition \ref{lem_dtX<<DX}. 

\medskip

Starting from Section \ref{sec:curvature},
we also use the Euler-Lagrange equation (\ref{ELG}) of the minimizing
movement scheme. By Euler-Lagrange equation we understand the first variation w.\ r.\ t.\ the independent
variables, as generated by a test vector field $\xi$.
In Section \ref{sec:curvature}, we pass to the limit in the energetic part of the first variation,
recovering the mean curvature $H$ via the term 
$\int_{\Sigma}H\,\xi\cdot\nu=\int_{\Sigma}\nabla\cdot\xi - \nu\cdot\nabla\xi\,\nu$.
This amounts to show that under our assumption of strengthened convergence (\ref{conv_ass}), the
$\Gamma$-convergence of the {\it functionals} can be upgraded to a distributional convergence of
their {\it first variations}, cf.\ Proposition \ref{surf_prop_integrated_in_time}.
It is a classical result credited to Reshetnyak \cite{reshetnyak1968weak} that under the
strengthened convergence of sets of finite perimeter, the measure-theoretic normals 
and thus the distributional expression for mean curvature also converge. 
The fact that this convergence of the first variation
may also hold when combined with a diffuse interface approximation is known for instance in case of the
Ginzburg-Landau approximation of the area functional (also known by the names of Modica and Mortola,
who established this $\Gamma$-convergence \cite{modica1987gradient,modica1977esempio}), 
see \cite{luckhaus1989gibbs}.
In our case the convergence of the first variations relies on a localization
of the ingredients for the $\Gamma$-convergence worked out in \cite{EseOtt14}, 
like the consistency, i.\ e. pointwise convergence of these functionals.

\medskip

Section \ref{sec:velocity} constitutes the central and, as we believe, 
most innovative piece of the paper; we pass to the limit
in the dissipation/metric part of the first variation, recovering the normal velocity $V$ via
the term $\int_{\Sigma}V\,\xi\cdot\nu$.
In fact, we think of the test-field $\xi$
as localizing this expression in time and space, and recover the desired limiting expression
only up to an error that measures how well the limiting configuration can be approximated by
a configuration with only two phases and a flat interface in the space-time patch
under consideration; this is measured both in terms of
area (leading to a multi-phase excess in the language of the regularity theory of minimal surfaces)
and volume, see Proposition \ref{prop_velocity_good_balls}.
The main difficulty of recovering the metric term $\int_{\Sigma}V\,\xi\cdot\nu$
in comparison to recovering the distributional form $\int_{\Sigma}\nabla\cdot\xi - \nu\cdot\nabla\xi\,\nu$ 
of the energetic term is that one has to recover both the normal velocity $V$, which is 
distributionally characterized by $\partial_t\chi-V|\nabla\chi|=0$ on the level of
the characteristic function $\chi$, and the (spatial) normal $\nu$.
In short: one has to pass to the limit in a {\it product}.
More precisely,
the main difficulty is that there is no good bound on the discrete normal velocity $V$ at hand on 
the level of the 
{\it microscopic} time scale $h$; only on the level of the above-mentioned mesoscopic time scale $\sqrt{h}$, 
such an estimate is available. This comes from the fact that the basic estimate yields control
of the time derivative of the characteristic function $\chi$ only when mollified on the spatial
scale $\sqrt{h}$
in $u=G_h*\chi$. The main technical ingredient to overcome this lack of control in 
Proposition \ref{prop_velocity_good_balls} is presented in Lemma \ref{la_periodic} in the two-phase case
and in Lemma \ref{la_1d_multiphase} in the general setting:
If one of the two (spatial) functions $u,\tilde u$ is not too far from being strictly
monotone in a given direction (a consequence of the control of the tilt excess, see 
Lemma \ref{lem_local_estimates}), 
then the spatial $L^1$-difference between the level sets $\chi=\{u>\frac{1}{2}\}$ and 
$\tilde\chi=\{\tilde u>\frac{1}{2}\}$ is controlled by the squared $L^2$-difference between
$u$ and $\tilde u$.

\medskip

In Section \ref{sec:conv}, we combine the results of the previous two sections yielding the weak formulation
of $V=H$ on some space-time patch up to an error expressed in terms of the above mentioned (multi-phase) 
tilt excess of the limit on that patch.
Complete localization in time and partition of unity in space allows us to assemble this to obtain 
$V=H$ globally, up to an error expressed by the time integral of the sum of the tilt excess over the 
spatial patches of finite overlap.
De Giorgi's structure theorem for sets of finite perimeter (cf.\ Theorem 4.4 in \cite{giusti1984minimal}),
adapted to a multi-phase situation but just used for a fixed time slice, implies that the error expression 
can be made arbitrarily small by sending the length scale of the spatial patches to zero.

 \tableofcontents

\subsection{Notation}
We denote by
\begin{align*}
	G_h(z) := \frac{1}{(2\pi h)^{d/2}} \exp\left(-\frac{|z|^2}{2h}\right)
\end{align*}
the Gaussian kernel of variance $h$.
Note that $G_{2t}(z)$ is the fundamental solution to the heat equation and thus
\begin{align*}
 \partial_h G -\tfrac12 \Delta G & =0 
 \quad \text{in } (0,\infty)\times \R^d,\\
 G &  = \delta_0  \quad \text{for } h=0.
\end{align*}
We recall some basic properties, such as the normalization, non-negativity, boundedness and
the factorization property:
\begin{align*}
\int_{\R^d} G_h\,dz = 1 ,\quad 0 \leq G_h \leq C h^{-d/2},\quad \nabla G_h(z)= - \frac zh  G_h(z)  ,
\quad G(z) = G^1(z_1)\, G^{d-1}(z'),
\end{align*}
where $G^1$ denotes the 1-dimensional and $G^{d-1}$ the $(d-1)$-dimensional Gaussian kernel;
let us also mention the semi-group property
\begin{align*}
 G_{s+t} = G_s \ast G_t.
\end{align*}
Throughout the paper, we will work on the flat torus $\torus$.
The thresholding scheme for multiple phases, introduced in \cite{EseOtt14},
for arbitrary surface tensions $\sigma_{ij}$ and mobilities $\mu_{ij}=1/\sigma_{ij}$ is the following.
\begin{algo}\label{MBO_arbitrary_surface_tensions}
Given the partition $\Omega^{n-1}_1,\dots,\Omega^{n-1}_\numphases$ of $\torus$ at time $t=(n-1)h$, 
obtain the new partition $\Omega^n_1,\dots,\Omega^n_\numphases$ at time $t=nh$ by:
\begin{enumerate}
 \item Convolution step:
\begin{align}\label{convolve}
 \phi_i := G_h\ast \left( \sum_{j=1}^\numphases \sigma_{ij}\chara_{\Omega^{n-1}_j}\right).
\end{align}
\item Thresholding step:
\begin{align}\label{threshold}
 \Omega^n_i := \left\{ x\in \torus \colon \phi_i(x) < \phi_j(x) \text{ for all } 
j\neq i \right\}.
\end{align}
\end{enumerate}
\end{algo}
We will denote the characteristic functions of the phases $\Omega^n_i$ at the $n^\text{th}$ time step
by $\chi_i^n$ and interpolate these functions piecewise constantly in time, i.\ e.\
\begin{align*}
 \chi_i^h(t) := \chi_i^n = \chara_{\Omega^n_i}\quad \text{for } t\in [nh,(n+1)h).
\end{align*}
As in \cite{EseOtt14}, we define the \emph{approximate energies}
\begin{align}\label{def_Eh}
E_h(\chi) := \frac1{\sqrt{h}} \sum_{i,j} \sigma_{ij} \int \chi_i 
\, G_h\ast \chi_j \,dx
\end{align}
for \emph{admissible} measurable functions:
\begin{align}\label{admissible}
 \chi=(\chi_1,\dots,\chi_\numphases)\colon \torus \to \{0,1\}^\numphases
\quad \text{s.\ t.} \quad \sum_{i=1}^{\numphases} \chi_i =1 \quad \ae
\end{align}
Here and in the sequel $\int dx$ stands short for $\int_{\torus}dx$, whereas
$\int dz$ stands short for $\int_{\R^d} dz$.
The minimal assumption on the matrix of surface tensions $\{\sigma_{ij}\}$, next to the obvious
\begin{align*}
 \sigma_{ij}=\sigma_{ji}\geq \sigma_{\min}>0 \quad \text{if } i\neq j,\quad\sigma_{ii}=0,
\end{align*}
is the following triangle inequality
\begin{align*}
 \sigma_{ij}\leq \sigma_{ik} + \sigma_{kj}.
\end{align*}
It is known that (e.\ g. \cite{EseOtt14}), under the conditions above, these energies $\Gamma$-converge 
w.\ r.\ t.\ the $L^1$-topology
to the \emph{optimal partition energy} given by
\begin{align*}
E(\chi):= c_0 \sum_{i,j} \sigma_{ij} \frac12 \Big(\int |\nabla\chi_i| +\int |\nabla\chi_j|
 - \int  |\nabla(\chi_i+\chi_j)|\Big)
\end{align*}
for \emph{admissible} $\chi$:
\begin{align*}
 \chi=(\chi_1,\dots,\chi_\numphases)\colon \torus \to \{0,1\}^\numphases\in BV
\quad \text{s.\ t.} \quad \sum_{i=1}^{\numphases} \chi_i =1 \quad \ae
\end{align*}
The constant $c_0$ is given by 
\begin{align*}
 c_0 := \omega_{d-1} \int_0^\infty G(r) r^d dr = \frac 1{\sqrt{2\pi}}.
\end{align*}
For our purpose we ask the matrix of surface tensions $\sigma $ to satisfy a
\emph{strict triangle inequality}:
\begin{align*}
 \sigma_{ij} < \sigma_{ik} + \sigma_{kj} \quad \text{for pairwise different }i,\,j,\,k.
\end{align*}
We recall the minimizing movements interpretation from \cite{EseOtt14} which is easy to check. 
The combination of convolution and thesholding step 
in Algorithm \ref{MBO_arbitrary_surface_tensions} is equivalent to solving the following
minimization problem
\begin{align}\label{MMinterpretation}
 \chi^n = \arg\min_\chi \left\{E_h(\chi) - E_h(\chi-\chi^{n-1})\right\},
\end{align}
where $\chi$ runs over (\ref{admissible}).
The proof will mostly be based on the interpretation (\ref{MMinterpretation}) and only once uses
the original form (\ref{convolve}) and (\ref{threshold}) in Lemma \ref{la_periodic} and Lemma \ref{lem_local_estimates}, respectively.
Following \cite{EseOtt14}, we will additionally assume that 
$\sigma$ is \emph{conditionally negative-definite}, i.\ e.\
\begin{align*}
 \sigma \leq - \underline\sigma \quad \text{on } (1,\dots,1)^\perp,
\end{align*}
where $\underline \sigma >0$ is a constant. 
That means, that $\sigma$ is negative as a bilinear form on $(1,\dots,1)^\perp$.
This ensures that $-E_h(\chi-\chi^{n-1})$ in (\ref{MMinterpretation}) is non-negative and penalizes
the distance to the previous step.\\
In the following we write $A \lesssim B$ to express that $A\leq C B$ for
a (possibly large) generic constant $C<\infty$ that only depends on the dimension
$d$, the total number of phases $\numphases$ and on the matrix of surface tensions $\sigma$
through $\sigma_{\min}=\min_{i\neq j} \sigma_{ij}$, $\sigma_{\max} = \max \sigma_{ij}$, $\underline \sigma$ and 
$\min \{ \sigma_{ik} + \sigma_{kj} - \sigma_{ij}\colon i,\, j,\, k\, \text{pairwise different}\}$.
Furthermore, we say a statement holds for $A \ll B$ if the statement holds for $A \leq \frac1C B$ for some generic constant $C<\infty$ as above.
%
\subsection{Main result}
The definition of our weak notion of mean-curvature flow is a distributional formulation which is suited to
the framework of functions of bounded variation.
%
%
\begin{defn}[Motion by mean curvature]\label{def_motion_by_mean_curvature}
Fix some finite time horizon $T<\infty$, a matrix of surface tensions $\sigma$ as above and initial data 
$\chi^0\colon \torus \to \{0,1\}^\numphases$ with $E_0 := E(\chi^0) <\infty$.
We say that the network
\begin{align*}
\chi = \left( \chi_1,\dots,\chi_\numphases\right): (0,T)\times \torus 
\to \{0,1\}^\numphases
\end{align*}
with $\sum_i \chi_i = 1$ a.\ e. and
\begin{align*}
 \sup_t E(\chi(t)) < \infty
\end{align*}
\emph{moves by mean curvature} if there exist functions
$V_i\colon (0,T)\times\torus\to \R$ with 
\begin{align*}
\int_0^T \int V_i^2 \left|\nabla\chi_i\right|dt < \infty 
\end{align*}
which satisfy
\begin{align}\label{H=v}
\sum_{i,j} \sigma_{ij} \int_0^T \int \left(\nabla\cdot\xi - \nu_i \cdot \nabla \xi \,\nu_i
- 2\,\xi \cdot \nu_i \,V_i \right) 
\frac12\left(\left|\nabla\chi_i\right| + \left|\nabla\chi_j\right| - \left|\nabla(\chi_i+\chi_j)\right|\right)  dt
= 0
\end{align}
for all $\xi\in C^\infty_0((0,T)\times \torus, \R^d)$ and which are normal velocities in the sense that
for all $\zeta\in C^\infty([0,T]\times \torus)$ with $\zeta(T)=0$ and 
all $i\in \left\{1,\dots,\numphases\right\}$
\begin{align}\label{v=dtX}
\int_0^T \int \partial_t \zeta \, \chi_i\, dx\,dt + \int \zeta(0) \chi_i^0 \,dx
= - \int_0^T \int \zeta\, V_i\left| \nabla \chi_i\right|dt.
\end{align}
\end{defn}
Note that (\ref{v=dtX}) also encodes the initial conditions as well as (\ref{H=v}) encodes the Herring angle
condition.
Indeed, for a smooth evolution, since for any interface $\Sigma$ we have
\begin{equation*}
 \int_\Sigma \left( \nabla \cdot \xi - \nu \cdot \nabla \xi \, \nu\right) = \int_\Gamma b\cdot \xi + \int_\Sigma H \nu \cdot \xi,
\end{equation*}
where $\Gamma = \partial \Sigma$, $b$ denotes the conormal and $H$ the mean curvature of $\Sigma$,
we do not only obtain the equation
\begin{align*}
H_{ij} = 2 V_{ij} \quad \text{on } \Sigma_{ij} =\partial \Omega_i \cap \partial \Omega_j
\end{align*}
along the smooth parts of the interfaces
but also the Herring angle condition at triple junctions.
If three phases $\Omega_1$, $\Omega_2$ and $\Omega_3$ meet at a point $x$, then
we have
\begin{align*}
 \sigma_{12}\,\nu_{12}(x)+ \sigma_{23}\,\nu_{23}(x)+ \sigma_{31}\,\nu_{31}(x)=0.
\end{align*}
In terms of the opening angles $\theta_1$, $\theta_2$ and $\theta_3$ at the junction, 
this condition reads
\begin{align*}
 \frac{\sin \theta_1}{\sigma_{23}} =  \frac{\sin \theta_2}{\sigma_{13}} =  \frac{\sin \theta_3}{\sigma_{12}},
\end{align*}
so that the opening angles at triple junctions are determined by the surface tensions.
\begin{rem}
To prove the convergence of the scheme, we will 
need the following convergence assumption:
\begin{align}\label{conv_ass}
 \int_0^T E_h(\chi^h)\,dt \to \int_0^T E(\chi)\,dt.
\end{align}
This assumption makes sure that there is no loss of area in the limit $h\to0$ as in Figure \ref{fig_loss_of_area}.
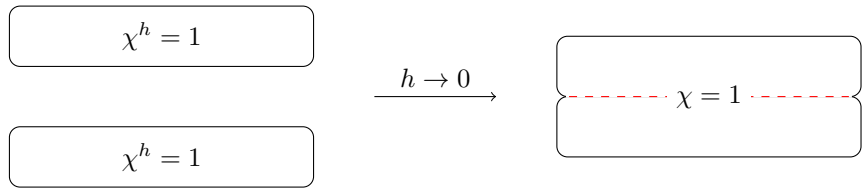
\begin{figure}[H]
\centering
\begin{tikzpicture}[scale=0.8]
    \draw[rounded corners] (0,-.5) rectangle (5,-1-.5) node[pos=.5] {$\chi^h=1$};
    \draw[rounded corners] (0,.5) rectangle (5,1+.5) node[pos=.5] {$\chi^h=1$};
    \draw[rounded corners]  (5+4,0) rectangle (10+4,1);
    \draw[rounded corners]  (5+4,0) rectangle (10+4,-1);
    \draw[->] (5+1,0)--(5+1+2,0) node[midway, above] {$h\to 0$};
    \draw[white,thick] (5+4+.15,0)--(10+4-.15,0);
    \draw[red,dashed] (5+4+.2,0)--(10+4-.2,0);
    \node[rectangle,fill=white]  at (7.5+4,0) {$\chi=1$};
\end{tikzpicture}
\caption{For fixed $t=t_0$ as $h\to 0$ there should be no loss of area. The ruled out case is
illustrated here. The dashed line is sometimes called hidden boundary.}%
\label{fig_loss_of_area}%
\end{figure}
\end{rem}

%
%
\begin{thm}\label{thm1}
Let $P\in \N$, let the matrix of surface tensions $\sigma$ 
satisfy the strict triangle inequality and be conditionally negative-definite, $T<\infty$ be a finite time
horizon and let $\chi^0$ be given with $E(\chi^0) < \infty$.
Then for any sequence there exists a subsequence $h\downarrow 0$ and a $\chi\colon (0,T)\times \torus \to \{0,1\}^\numphases$ with $E(\chi(t))\leq E_0$
such that the approximate solutions $\chi^h$ obtained by Algorithm \ref{MBO_arbitrary_surface_tensions}
converge to $\chi$.
Given (\ref{conv_ass}), $\chi$ moves by mean curvature in the sense of
Definition \ref{def_motion_by_mean_curvature} with initial data $\chi^0$.
\end{thm}

\begin{rem}
 An upcoming result of the authors will show that under the assumption (\ref{conv_ass}) the limit $\chi$ solves a localized energy inequality
 and is thus a weak solution in the sense of Brakke.
\end{rem}

\begin{rem}\label{timescales}
 Our proof uses the following three different time scales:
 \begin{enumerate}
  \item  The \emph{macroscopic time scale}, $T<\infty$, given by the finite time horizon,
  \item  the \emph{mesoscopic time scale},  $\tau=\alpha\h \sim \h>0$ and
  \item  the \emph{microscopic time scale}, $h>0$,  coming from the time discretization.
 \end{enumerate}
The mesoscopic time scale arises naturally from the scheme:
Due to the parabolic scaling, the microscopic time scale $h$ corresponds to the length scale
$\h$ as can be seen from the kernel $G_h$.
Since for a smooth evolution, the normal velocity $V$ is of order $1$, this prompts the
mesoscopic time scale $\h$.\\
The parameter $\alpha$ will be kept fixed most of the time until the very end, where we send $\alpha\to 0$.
Therefore, it is natural to think of $\alpha \sim 1$, but small.\\
These three time scales go hand in hand with the following numbers, which we will for simplicity assume
to be natural numbers throughout the proof:
\begin{enumerate}
 \item $N$ - the total number of microscopic time steps in a macroscopic time interval $(0,T)$,
 \item $\kmax$ - the number of microscopic time steps in a mesoscopic time interval $(0,\tau)$ and
 \item $\lmax$ - the number of mesoscopic time intervals in a macroscopic time interval.
\end{enumerate}
The following simple identities linking these different parameters will be used frequently:
\begin{align*}
T= Nh = L \tau , \quad \tau = K h,\quad  L = \frac NK = \frac T\tau.
\end{align*}
\end{rem}
\begin{figure}
\centering
\begin{tikzpicture}[scale=1]
\draw[->] (0,0) -- (11,0);
\node[below=.37,anchor=base] at (11,0) {$t$};
\node[above] at (11,0) {$\#$steps};

\node [below=.37,anchor=base] at (0,0) {$0$};
\node [below=.37,anchor=base] at (10,0) {$T$};
\node [above] at (0,0) {$0$};
\node [above] at (10,0) {$N$};

\foreach \index in {0, ..., 24}
{%
\draw (\index*10/24,-.04) -- (\index*10/24,.04) ;
}
\node [below=.37,anchor=base] at (10/24,0) {$h$};
\node [above] at (10/24,0) {$1$};
\node [below=.37,anchor=base] at (2*10/24,0) {$2h$};
\node [above] at (2*10/24,0) {$2$};

\foreach \index in {0, ..., 4}
{%
\draw (\index*10/4,-2*.04) -- (\index*10/4,2*.04) ;
}
\draw  (0,-2*.04) -- (0,2*.04) ;
\draw  (10,-2*.04) -- (10,2*.04) ;
\node [below=.37,anchor=base] at (10/4,0) {$\tau$};
\node [above] at (10/4,0) {$K$};
\node [below, anchor=base] at (2*10/4,-10*.04) {$2\tau$};
\node [above] at (2*10/4,0) {$2K$};
\end{tikzpicture}

\caption{The micro-, meso-, and macroscopic time scales $h$, $\tau$ and $T$.}%
\label{fig_timescales}%
\end{figure}
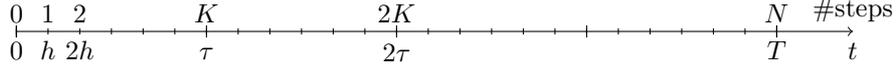

\section{Compactness}\label{sec:comp}
In this section we prove the compactness of the approximate solutions,
construct the normal velocities and derive bounds on these velocities.
In the first subsection we present all results of this section;
the proofs can be found in the subsequent subsection.

\subsection{Results}\label{sub:comp results}
The first main result of this section is the following compactness statement.
%
%
\begin{prop}[Compactness]\label{comp_prop}
 There exists a sequence $h\downarrow 0$ and a limit
$\chi \colon(0,T)\times\torus \to \{0,1\}^\numphases$ such that
\begin{align}\label{comp_conv_2}
 \chi^h\longrightarrow \chi \quad \text{a.\ e. in } (0,T)\times \torus
\end{align}
and the limit satisfies 
$
 E(\chi(t)) \leq E_0
$ and $\chi(t)$ is admissible in the sense of (\ref{admissible}) for a.\ e.\ $t\in (0,T)$.
\end{prop}

The second main result of this section is the following construction of the
normal velocities and the square-integrability under the convergence assumption (\ref{conv_ass}).
\begin{prop}\label{lem_dtX<<DX}
If the convergence assumption (\ref{conv_ass}) holds, 
the limit $\chi=\lim_{h\to0} \chi^h$ has the following properties.
\begin{enumerate}[(i)]
 \item $\partial_t \chi$ is a Radon measure with
  \begin{align*}
   \iint \left| \partial_t \chi_i\right|
 \lesssim (1+T)E_0
  \end{align*}
 for each $i\in\{1,\dots,\numphases\}.$
 \item For each $i\in\{1,\dots,\numphases\}$, $\partial_t \chi_i$ is absolutely continuous w.\ r.\ t.
$\left|\nabla \chi_i \right|dt$. 
    In particular, there exists a density
    $V_i\in L^1(\left|\nabla \chi_i \right|dt)$ such that
\begin{align*}
 -\int_0^T\int \partial_t \zeta \,\chi_i\, dx\,dt 
= \int_0^T\int \zeta \, V_i \left|\nabla \chi_i \right|dt
\end{align*}
for all $\zeta\in C_0^\infty ((0,T)\times\torus)$.
 \item  We have a strong $L^2$-bound: For each $i\in\{1,\dots,\numphases\}$
 \begin{align*}
  \int_0^T \int V_i^2 \left| \nabla \chi_i \right| dt \lesssim \left(1+T\right)E_0.
 \end{align*}
\end{enumerate}
\end{prop}

Both results essentially stem from the following 
basic estimate, a direct consequence of the minimizing movements interpretation
(\ref{MMinterpretation}).
%
%
%
\begin{lem}[Energy-dissipation estimate]\label{lem_energy_dissipation_estimate}
 The approximate solutions satisfy
\begin{align}\label{energy_dissipation_estimate}
 E_h(\chi^N) - \sum_{n=1}^N E_h(\chi^n-\chi^{n-1}) \leq E_0.
\end{align}
$\sqrt{-E_h}$ defines a norm on the \emph{process space}
 $\{ \omega \colon \torus \to \R^\numphases | \sum_i \omega_i = 0 \}.
 $
In particular, the algorithm dissipates energy.
\end{lem}

%
%
In order to prove Proposition \ref{comp_prop} we derive estimates on time-
and space-variations of the approximations only using the basic estimate 
(\ref{energy_dissipation_estimate}).

The estimate (\ref{energy_dissipation_estimate}) bounds the
(approximate) energies $E_h(\chi^h)$, which in turn control
$\int \left| \nabla G_h\ast \chi^h\right| dx$ and thus variations of $G_h\ast \chi^h$ in space.
On length scales greater than $\h$, this estimate also survives for the approximations $\chi^h$.

\begin{lem}[Almost BV in space]\label{comp_lem_z_shift}
The approximate solutions satisfy
\begin{align}\label{comp_z_shift}
 \int_0^{T} \int \left| \chi^h(x+\delta e,t)-\chi^h(x,t) \right| dx\,dt 
\lesssim (1+T)E_0\left(\delta +\h\right)
\end{align}
for any $\delta >0$ and $e\in \sphere$.
\end{lem}
%

%
%
Variations in time are controlled by the following lemma coming from 
interpolating the (unbalanced) estimate (\ref{energy_dissipation_estimate})
on time scales of order $\h$.

\begin{lem}[Almost BV in time]\label{comp_lem_tau_shift}
The approximate solutions satisfy
\begin{align}\label{comp_tau_shift}
 \int_\tau^{T} \int \left| \chi^h(t)-\chi^h(t-\tau) \right| dx\,dt 
\lesssim (1+T)E_0 \left( \tau + \h\right)
\end{align}
for any $\tau>0$.
\end{lem}
Let us also mention that with the same methods we can prove 
$C^{1/2}$-H\"older-regularity of the volumes, i.\ e.\
$\left|\Omega(s)\Delta \Omega(t)\right| \lesssim \left| s-t\right|^{\frac12}$. 
For the approximations this estimate of course only holds on time scales larger than the time-step size $h$.
%
%
\begin{lem}[$C^{1/2}$-Bounds]\label{comp_lem_hoelder}
We have uniform H\"older-type bounds for the approximate solutions:
I.\ e.\ for any pair $s,t\in[0, T]$ with $|s-t|\geq h$ we have
\begin{align}\label{Hoelder_discrete}
 \int \left| \chi^h(s) -\chi^h(t)\right|\,dx \lesssim E_0 \left|s-t\right|^{\frac12}.
\end{align}
In particular, $\chi\in C^{1/2}([0,T],L^1(\torus))$: For almost every $s,t\in(0, T)$, we have
\begin{align}\label{Hoelder_continuous}
 \int \left| \chi(s) -\chi(t)\right|\,dx \lesssim E_0 \left|s-t\right|^{\frac12}.
\end{align}
\end{lem}
%

%
%
For the proof of the second main result of this section, Proposition \ref{lem_dtX<<DX},
and also for later use in Section \ref{sec:velocity} 
it is useful to define certain measures which are induced by the metric term.
These measures allow us to localize the result of Lemma \ref{comp_lem_tau_shift}.
In the two-phase case this is enough to
prove that the measure $\partial_t \chi$ is absolutely continuous
w.\ r.\ t.\ the perimeter and
the existence and integrability of the normal velocity, cf. (i) and (ii) of
Proposition \ref{lem_dtX<<DX}.
The square-integrability follows then from a refinement of these estimates
by localizing the fudge factor $\alpha$ (cf. Remark \ref{timescales}) after passage to the limit
$h\to0$.

\begin{defn}[Dissipation measure]\label{def_diss_meas}
 For $h>0$, we define the \emph{approximate dissipation measures} (associated to the approximate solution
 $\chi^h$) $\mu_h$ on $[0,T]\times \torus$ by
\begin{align}\label{diss meas}
 \iint \zeta \,d\mu_h := \sum_{n=1}^{N} 
 \frac1\h \int \overline \zeta^{n} \left(\left|G_{h/2}\ast \left( \chi^n - \chi^{n-1}\right) \right|^2
 + \left|G_{h}\ast \left( \chi^n - \chi^{n-1}\right) \right|^2\right) dx,
\end{align}
where $\zeta \in C^\infty([0,T]\times \torus)$ and $\overline \zeta^n$ is the time average of $\zeta$
on the interval $[nh,(n+1)h)$.
By the monotonicity of $h\mapsto \|G_h\ast u\|_{L^2}$ and the 
energy-dissipation estimate (\ref{energy_dissipation_estimate}), we have
\begin{align}\label{mu finite}
 \mu_h([0,T]\times\torus) \lesssim E_0
\end{align}
and $\mu_h \rightharpoonup \mu$ after passage to a further subsequence for some finite,
non-negative measure $\mu$ on $[0,T]\times\torus$ with $\mu([0,T]\times\torus) \lesssim E_0$.
We call $\mu$ the \emph{dissipation measure}.
\end{defn}
%
%
In order to prove Proposition \ref{lem_dtX<<DX} also in the multi-phase case we have to ensure that the convergence assumption
implies the convergence of the individual interfacial areas $\frac12 \int \left(\left| \nabla \chi_i \right| + \left| \nabla \chi_j \right|
  - \left| \nabla (\chi_i+\chi_j) \right|\right)$.
\begin{lem}[Implications of convergence assumption]\label{la_impl_conv_ass}
 The convergence assumption (\ref{conv_ass}) ensures that for any pair $i\neq j$ and
 any $\zeta \in C^\infty([0,T]\times\torus)$,
 \begin{align}\label{impl_of_conv_ass}
  \int_0^T \frac 1\h \int \zeta \left(\chi^h_i \, G_h\ast \chi^h_j +  \chi^h_j \, G_h\ast \chi^h_i \right) dx \,dt
  \to c_0 \int_0^T \int \zeta \left(\left| \nabla \chi_i \right| + \left| \nabla \chi_j \right|
  - \left| \nabla (\chi_i+\chi_j) \right|\right)dt,
 \end{align}
 as $h\to0$.
\end{lem}
The proof of Lemma \ref{la_impl_conv_ass} heavily relies on the fact that $\sigma$ satisfies the \emph{strict} triangle inequality so that
we can preserve the triangle inequality after perturbing the energy functional.
The following example shows that this is not a technical assumption but is a necessary condition for the lemma
to hold and thus plays a crucial role in identifying the normal velocities $V_i$.
\begin{expl}
 To fix ideas let  us consider three sets $\Omega_1,\,\Omega_2$ and $\Omega_3$ in dimension $d=2$ with surface tensions $\sigma_{12}= \sigma_{23}=1$, $\sigma_{13}=2$
 as illustrated in Figure \ref{fig_interfaces_merging}.
   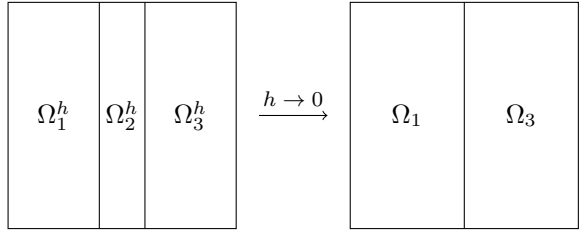
\begin{figure}
    \centering
      \begin{tikzpicture}[scale=1]
    \draw (0,0) rectangle (3,3);
    \draw (3/2-.3,0) -- (3/2-.3,3);
    \draw (3/2+.3,0) -- (3/2+.3,3);
    \node at (3/4-.3/2,3/2) {$\Omega_1^h$};
    \node at (3*3/4+.3/2,3/2) {$\Omega_3^h$};
    \node at (3/2,3/2) {$\Omega_2^h$};
    
    \draw[->] (3+.3,3/2) -- node[above] {\footnotesize$h\to 0$} (1.5*3-.3,3/2);
\begin{scope}[shift={(1.5*3,0)}]
    \draw (0,0) rectangle (3,3);
    \draw (3/2,0) -- (3/2,3);
    \node at (3/4,3/2) {$\Omega_1$};
    \node at (3*3/4,3/2) {$\Omega_3$};
\end{scope}
\end{tikzpicture}
    \caption{As $h\to0$, the two interfaces $\Sigma^h_{12}$ and $\Sigma^h_{23}$ merge into one interface, $\Sigma_{13}$, between Phases 1 and 3. Therefore the
    measure of $|\Sigma^h_{13}|$ jumps up in the limit $h\to0$ although the total interfacial energy converges due to the choice of surface tensions.}
    \label{fig_interfaces_merging}
   \end{figure}
Then, the total energy is constant in $h$ and due to the choice of the surface tensions
the convergence assumption is fulfilled.
Nevertheless, we clearly have
\begin{align*}
  |\Sigma_{12}^h| = \text{const.} > 0 
  = |\Sigma_{12}| \quad \text{and} \quad
  |\Sigma_{13}^h| = 0 < \text{const.} 
  = |\Sigma_{13}|.
\end{align*}
This example also illustrates that although the energy functional $E$ is lower semi-continuous,
the individual interfacial energies $\frac12\int \left( |\nabla\chi_i| +|\nabla\chi_j| -
|\nabla(\chi_i+\chi_j)| \right)$ are not.
\end{expl}

\subsection{Proofs}\label{sub:comp proofs}
Before proving the statements of this section
we cite two results of \cite{EseOtt14} which will be used frequently in the proofs.

The following monotonicity statement is a key tool for the $\Gamma$-convergence in \cite{EseOtt14}.
We will use it throughout our proofs but we seem not to rely heavily on it.
\begin{lem}[Approximate monotonicity]\label{lem_AM_EO}
 For all $0<h\leq h_0$ and any admissible $\chi$, we have 
 \begin{align}\label{approximate monotonicity}
  E_h(\chi) \geq \left( \frac{\sqrt{h_0}}{\h+\sqrt{h_0}}\right)^{d+1}E_{h_0}(\chi).
 \end{align}
\end{lem}

Another important tool for the $\Gamma$-convergence in \cite{EseOtt14}
is the following consistency, or pointwise convergence of the functionals $E_h$ to $E$, which we will
refine in Section \ref{sec:curvature}.

\begin{lem}[Consistency]\label{lem_CONS_EO}
 For any admissible $\chi\in BV$, we have 
 \begin{align}\label{consistency}
  \lim_{h\to0}E_h(\chi) = E(\chi).
 \end{align}
\end{lem}
Taking the limit $h\to 0$ in (\ref{approximate monotonicity}) with $\chi=\chi^0$
and using (\ref{consistency}), we see that that the interfacial energy $E_0$ 
of the initial data $\chi(0)\equiv \chi^0$ bounds the approximate energy of the initial data:
\begin{align*}
 E_0 := E(\chi(0))\geq E_h(\chi^0).
\end{align*}

We first prove Proposition \ref{comp_prop} which follows directly from the estimates in Lemmas \ref{comp_lem_z_shift} and \ref{comp_lem_tau_shift}.
Then we give the proofs of the Lemmas used for Proposition \ref{comp_prop}.
We present the proof of Proposition \ref{lem_dtX<<DX} at the end of this section since the proof heavily relies on the techniques developed in the proofs of the lemmas,
especially in Lemma \ref{comp_lem_tau_shift}.

%
%
\begin{proof}[Proof of Proposition \ref{comp_prop}] 
The proof is an adaptation of the Riesz-Kolmogorov $L^p$-compactness theorem.
By Lemma \ref{comp_lem_z_shift} and Lemma \ref{comp_lem_tau_shift}, we have
\begin{align}\label{comp_tau_z_shift}
 \int_0^{T} \int \left| \chi^h(x+\delta e,t+\tau)-\chi^h(t) \right| dx\,dt 
\lesssim \left(1+T\right)E_0 \left( \delta + \tau + \h\right)
\end{align}
for any $\delta,\tau>0$ and $e\in \sphere$.
For $\delta>0$ consider the mollifier $\varphi_\delta$ given by the scaling
 $\varphi_\delta(x) := \frac1{\delta^{d+1}}\varphi(\frac x\delta,\frac t\delta)$ and
 $\varphi\in C_0^\infty( (-1,0)\times B_1)$ such that $0\leq \varphi\leq 1$ and 
$\int_{-1}^0 \int_{B_1} \varphi =1$. 
We have the estimates
\begin{align*}
 \left|\varphi_\delta\ast \chi^h\right| \leq 1\quad \text{and} \quad
 \left| \nabla (\varphi_\delta \ast \chi^h)\right| \lesssim \frac1\delta.
\end{align*}
Hence, on the one hand, 
the mollified functions are equicontinuous and by Arzel\`a-Ascoli precompact in $C^0([0,T]\times\torus)$:
For given $\epsilon,\delta>0$ there exist functions 
$u_i\in C^0([0,T]\times\torus)$, $i=1,\dots,n(\epsilon,\delta)$ such that
\begin{align*}
\left\{ \varphi_\delta \ast \chi^h \colon h>0\right\} 
\subset \bigcup_{i=1}^{n(\epsilon,\delta)} B_\epsilon(u_i),
\end{align*}
where the balls $B_\epsilon(u_i)$ are given w.\ r.\ t.\ the $C^0$-norm. 
On the other hand, for any function $\chi$ we have
\begin{align*}
 \int_0^T\int \left|\varphi_\delta \ast \chi -\chi \right|dx\,dt
\leq & \int \varphi_\delta(z,s) \int \left| \chi(x-z,t-s)- \chi(x,t) \right|d(x,t)\,d(z,s)\\
\leq & \sup_{(z,s)\in \supp \varphi_\delta}
 \int_0^T \int \left| \chi(x-z,t-s)- \chi(x,t) \right|dx\,dt.
\end{align*}
Using this for $\chi^h$ and plugging in (\ref{comp_tau_z_shift}) yields
\begin{align*}
 \int_0^T\int \left|\varphi_\delta \ast \chi^h -\chi^h \right|dx\,dt 
\lesssim \left(1+T\right)E_0 \left( \delta + \h \right).
\end{align*}
Given $\rho>0$, fix $\delta,h_0>0$ such that
\begin{align*}
 \int_0^T\int \left|\varphi_\delta \ast \chi^h -\chi^h \right|dx\,dt \leq \frac\rho2
\quad \text{for all }h\in(0,h_0). 
\end{align*}
Then set $\epsilon:=\frac\rho{T \Lambda ^d}$ and find $u_1,\dots,u_n$ from above.
Note that only finitely many of the elements in the sequence $\{h\}$ are greater than $h_0$. 
Therefore,
\begin{align*}
 \{\chi^h \}_h \subset \bigcup_{i=1}^{n} B_\rho(u_i) \cup \{\chi^h\}_{h>h_0}
\subset  \bigcup_{i=1}^{n} B_\rho(u_i) \cup \bigcup _{h>h_0} B_\rho(\chi^h)
\end{align*}
is a finite covering of balls (w.\ r.\ t.\ $L^1$-norm) of given radius $\rho>0$. 
Therefore, $\{\chi^h\}_h$ is precompact and hence relatively compact in $L^1$.
Hence we can extract a converging subsequence. 
After passing to another subsequence, 
we can w.\ l.\ o.\ g.\ assume that we also have pointwise convergence almost everywhere 
in $ (0,T)\times \torus$.
\end{proof}

\begin{proof}[Proof of Lemma \ref{lem_energy_dissipation_estimate}]
By the minimality condition (\ref{MMinterpretation}), we have in particular 
\begin{align*}
E_h(\chi^n) - E_h(\chi^n-\chi^{n-1}) \leq E_h(\chi^{n-1})
\end{align*}
for each $n=1,\dots,N$. 
Iterating this estimate yields (\ref{energy_dissipation_estimate}) 
with $E_h(\chi^0)$ instead of $E_0= E(\chi^0)$.
Then (\ref{energy_dissipation_estimate}) follows from the short argument after Lemma \ref{lem_CONS_EO}.\\
We claim that the pairing
$
 - \frac1\h \int \omega \cdot \sigma \left( G_h\ast \tilde \omega\right)dx
$
defines a scalar product on the process space. 
It is bilinear and symmetric thanks to the symmetry of $\sigma$ and $G_h$.
Since $\sigma$ is conditionally negative-definite,
\begin{align*}
 - \frac1\h \int \omega \cdot \sigma \left( G_h\ast \omega\right)dx
 =-\frac1\h\int \left( G_{h/2}\ast \omega \right)
\cdot \sigma \left( G_{h/2}\ast \omega \right)dx
\geq \frac{\underline\sigma}\h \|G_{h/2}\ast \omega \|_{L^2}^2 \geq 0.
\end{align*}
Furthermore, we have equality only if $\omega\equiv 0$. 
Thus, $\sqrt{-E_h}$ is the induced norm on the process space.
\end{proof}

\begin{proof}[Proof of Lemma \ref{comp_lem_z_shift}]
\step{Step 1:} We claim that
\begin{align}\label{comp_bound_on_gradient}
 \int_0^T \int \left| \nabla  G_h\ast \chi^h\right| \,dx\,dt \lesssim (1+T)E_0.
\end{align}
Indeed, for any characteristic function $\chi: \torus \to \{0,1\}$ we have
\begin{align*}
 \nabla (G_h\ast \chi)(x) 
= - \int \nabla G_h(z) \left( \chi(x+z)-\chi(x) \right) dz .
\end{align*}
Therefore, since $\left| \nabla  G_h(z)\right| \lesssim \frac1\h \left| G_{2h}(z)\right|$,
\begin{align*}
 \int \left| \nabla  G_h\ast \chi \right| \,dx \lesssim \frac1\h \int G_{2h}(z) \int \left|\chi(x+z)-\chi(x) \right|\,dx\,dz.
\end{align*}
By $\chi\in \{0,1\}$, we have $\left|\chi(x+z)-\chi(x) \right| = \chi(x)\left(1-\chi\right)(x+z) + \left(1-\chi\right)(x)\chi(x+z)$
 and thus by symmetry of $G_{2h}$:
\begin{align*}
  \int \left| \nabla  G_h\ast \chi\right| \,dx 
\lesssim \frac1\h \int \left(1-\chi\right)\, G_{2h} \ast \chi \,dx.
\end{align*}
Applying this on $\chi^h_i$, summing over $i=1,\dots,\numphases$, using $\chi^h_i = 1-\sum_{j\neq i} \chi^h_j$ 
and $\sigma_{ij}\geq\sigma_{\min}>0$ for $i\neq j$ we obtain 
\begin{align*}
 \int \left| \nabla G_h\ast \chi^h(t)\right| \,dx \lesssim E_{2h}(\chi^h) 
\lesssim E_h(\chi^h),
\end{align*}
where we used the approximate monotonicity of $E_h$, cf.\ Lemma \ref{lem_AM_EO}.
Using the energy-dissipation estimate 
(\ref{energy_dissipation_estimate}), we have 
\begin{align*}
 \int \left| \nabla G_h\ast \chi^h(t)\right| dx \lesssim E_0
\end{align*}
and integration in time yields (\ref{comp_bound_on_gradient}).
\step{Step 2:}
By (\ref{comp_bound_on_gradient}) and Hadamard's trick,
we have on the one hand
\begin{align*}
 \int_0^{T}
\int \left|G_h\ast \chi^h(x+\delta e,t)- G_h\ast\chi^h(x,t) \right| dx\,dt
 \lesssim (1+T)E_0 \delta.
\end{align*}
Since $\chi\in \{0,1\}$, we have on the other hand
\begin{align*}
 \left(\chi-G_h\ast \chi\right)_+ = \chi \,G_h\ast \left(1-\chi\right)\quad \text{and}\quad
  \left(\chi-G_h\ast \chi\right)_- = \left(1-\chi\right) G_h\ast \chi,
\end{align*}
which yields
\begin{align}\label{comp_x-Ghx}
|\chi-G_h\ast \chi| =  \left(1-\chi\right) G_h\ast \chi + \chi \,G_h\ast \left(1-\chi\right).
\end{align}
Using the translation invariance and (\ref{comp_x-Ghx}) for
the components of $\chi^h$, we have
\begin{align*}
\int_0^T \int \left| \chi^h(x+\delta e,t)- \chi^h(x,t)\right| dx\,dt
\leq& 2 \int_0^T \int \left| G_h\ast \chi^h 
- \chi^h \right| dx\,dt\\
&+ \int_0^T \int \left|G_h\ast \chi^h(x+\delta e,t)
- G_h\ast\chi^h(x,t) \right| dx\,dt\\
\lesssim& \left(1+T\right) E_0 \left(\h + \delta\right).
\end{align*}
\end{proof}

\begin{proof}[Proof of Lemma \ref{comp_lem_tau_shift}]
In this proof, we make use of the mesoscopic time scale $\tau=\alpha\h$, see Remark \ref{timescales} for the notation.
First we argue that it is enough to prove
\begin{align}\label{comp_tau_shift_bounded_tau}
 \int_\tau^{T} \int \left| \chi^h(t)-\chi^h(t-\tau) \right| dx\,dt 
\lesssim (1+T)E_0  \tau 
\end{align}
for $\alpha\in[1,2]$.
If $\alpha\in(0,1)$, we can apply (\ref{comp_tau_shift_bounded_tau}) twice, 
once for $\tau= \h$ and once for $\tau = (1+\alpha)\h$ and obtain (\ref{comp_tau_shift}).
If $\alpha>2$, we can iterate (\ref{comp_tau_shift_bounded_tau}).
Thus we may assume that $\alpha\in[1,2]$.
We have
\begin{align*}
 \int_\tau^{T} \int \left| \chi^h(t)-\chi^h(t-\tau) \right| dx\,dt
 =& h \sum_{k=0}^{\kmax-1} \sum_{l=1}^\lmax \int 
 \big| \chi^{\kmax l + k}-\chi^{\kmax(l-1)+k} \big| \,dx\\
=& \frac1\kmax \sum_{k=0}^{\kmax-1} \tau \sum_{l=1}^\lmax \int  
 \big| \chi^{\kmax l + k}-\chi^{\kmax(l-1)+k} \big| \,dx.
\end{align*}
Thus, it is enough to prove
\begin{align*}
 \sum_{l=1}^\lmax \int  
 \big| \chi^{\kmax l + k}-\chi^{\kmax(l-1)+k} \big| \,dx \lesssim  (1+T)E_0
\end{align*}
for any $k=0,\dots,\kmax-1$.
By the energy-dissipation estimate (\ref{energy_dissipation_estimate}), we have
$E_h(\chi^k) \leq E_0$ for all these $k$'s.
Hence we may assume w.\ l.\ o.\ g.\ that $k=0$ and prove only
\begin{align}\label{comp_Nsteps}
 \sum_{l=1}^\lmax \int  
 \big| \chi^{\kmax l}-\chi^{\kmax(l-1)} \big|\, dx \lesssim (1+T)E_0.
\end{align}
Note that for any two characteristic functions $\chi,\tilde \chi$ we have
\begin{align}\label{comp_x-xtilde}
 |\chi-\tilde \chi| 
=& (\chi-\tilde\chi )\,G_h\ast (\chi-\tilde\chi ) 
+ (\chi-\tilde\chi ) (\chi-\tilde\chi -G_h\ast (\chi-\tilde\chi ))\notag \\
\leq& (\chi-\tilde\chi )\,G_h\ast (\chi-\tilde\chi ) 
+ \left| \chi-G_h\ast \chi \right| + \left| \tilde \chi-G_h\ast \tilde \chi \right|.
\end{align}
Now we post-process the energy-dissipation estimate (\ref{energy_dissipation_estimate}). 
Using the triangle inequality for the norm $\sqrt{-E_h}$ on the process space and 
Jensen's inequality, we have
\begin{align}\label{triangle and jensen}
  -E_h\left(\chi^{\kmax l}-\chi^{\kmax(l-1)}\right)
\leq& \left( \sum_{n=\kmax(l-1)+1}^{\kmax l} 
\!\!\!\Big(  -E_h\left(\chi^{n}-\chi^{n-1}\right) \Big)^\frac12 \right)^2
\leq   \kmax\!\!\! \sum_{n=\kmax(l-1)+1}^{\kmax l}  \!\!\! -E_h\left(\chi^{n}-\chi^{n-1}\right).
\end{align}
Using (\ref{comp_x-xtilde}) for $\chi^{\kmax l}_i$ and $\chi^{\kmax (l-1)}_i$ with
(\ref{comp_x-Ghx}) for the second and the third right-hand side term and the conditional negativity of $\sigma$ and the above inequality for the first right-hand side term we obtain
\begin{align*}
   \sum_{l=1}^\lmax \int  
\big| \chi_i^{\kmax l }-\chi_i^{\kmax(l-1)} \big|\, dx 
\lesssim  \h\kmax
\sum_{n=1}^{N}  -E_h\left(\chi^{n}-\chi^{n-1}\right) +  \lmax   
\max_{n} \int \left(1-\chi_i^n\right)\,G_h\ast \chi_i^n\,dx.
\end{align*}
Since $(1-\chi_i^n) = \sum_{j\neq i} \chi_j^n$ a.\ e.\ and $\sigma_{ij}\geq \sigma_{\min}>0$ 
for all $i\neq j$, the energy-dissipation estimate (\ref{energy_dissipation_estimate}) yields
\begin{align*}
  \sum_{l=1}^{\lmax} \int  \big| \chi^{\kmax l}-\chi^{\kmax (l-1)} \big|\, dx 
\lesssim  \alpha E_0 + \frac1\alpha T   E_0 \lesssim (1+T) E_0,
\end{align*}
which establishes (\ref{comp_Nsteps}) and thus concludes the proof.
\end{proof}

\begin{proof}[Proof of Lemma \ref{comp_lem_hoelder}]
 First note that (\ref{Hoelder_continuous}) follows directly from (\ref{Hoelder_discrete})
 since we also have $\chi^h(t) \to \chi(t)$ in $L^1$ for almost every $t$.
The argument for (\ref{Hoelder_discrete}) comes in two steps. 
Let $s>t$, $\tau:= s-t$ and $t\in [nh,(n+1)h)$.
\step{Step 1:} Let $\tau$ be a multiple of $h$.
We may assume w.\ l.\ o.\ g.\ that $\tau = m^2h$ for some $m\in \N$.
As in the proof of Lemma \ref{comp_lem_tau_shift}, using (\ref{comp_x-xtilde}) and (\ref{triangle and jensen}) we derive
\begin{align*}
 \int \left| \chi^{n+m}-\chi^n\right| \,dx\lesssim m\h \sum_{k=1}^m 
 -E_h(\chi^{n+k}-\chi^{n+k-1}) +  \h \max_{t} E_h(\chi^h(t)). 
\end{align*}
As before, we sum these estimates:
\begin{align*}
  \int \big| \chi^{n+m^2}-\chi^n\big| \,dx
  \leq & \sum_{l=0}^{m-1}\int \big| \chi^{n+m(l+1)}-\chi^{n+ml}\big| \,dx\\
 \lesssim& m \h \sum_{n'=n}^{n+m^2} -E_h(\chi^{n'}-\chi^{n'-1})
 +  m\h \max_{t} E_h(\chi^h(t))\\
\lesssim& m\h E_0 = E_0 \sqrt{\tau}.
\end{align*}
\step{Step 2:} Let $\tau\geq h$ be arbitrary. Take $m\in \N$ such that $s\in [(m+n)h, (m+n+1)h)$. 
From Step 2 we obtain the bound in terms of $mh$ instead of $\tau$.
If $\tau\geq mh$, we are done. If $h\leq \tau<  mh$, 
then $m\geq2$ and thus $mh\leq \frac{m}{m-1}\tau \lesssim \tau$.
\end{proof}

\begin{proof}[Proof of Lemma \ref{la_impl_conv_ass}]
W.\ l.\ o.\ g. let $i=1$, $j=2$.
We prove the statement in three steps. In the first step we reduce the statement to a time-independent one.
In the second step, we show that due to the strict triangle inequality, the convergence of the energies implies the
convergence of the individual perimeters.
In the third step, we conclude by showing that this convergence still holds true if we localize with a test function $\zeta$, which
proves the time-independent statement formulated in the first step.
\step{Step 1: Reduction to a time-independent problem.}
It is enough to prove that $\chi^h\to \chi $ in $L^1(\torus,\R^\numphases)$ and
$ E_h(\chi^h)\to E(\chi)$
imply 
\begin{align}\label{pf_impl_of_conv_ass}
  \frac 1\h \int \zeta \left(\chi^h_1 \, G_h\ast \chi^h_2 +  \chi^h_2 \, G_h\ast \chi^h_1 \right) \,dx
  \to c_0  \int \zeta \left(\left| \nabla \chi_1 \right| + \left| \nabla \chi_2 \right|
  - \left| \nabla (\chi_1+\chi_2) \right|\right)
\end{align}
for any $\zeta \in C^\infty(\torus)$.\\
Given $\chi^h\to \chi$ in $L^1((0,T)\times \torus)$, for a subsequence we clearly have
$\chi^h(t)\to \chi(t)$ in $L^1(\torus)$ for a.\ e. t.
We further claim that for a subsequence
\begin{align}\label{EhtoEpointwise}
 E_h(\chi^h) \to E(\chi)\quad  \text{for a.\ e. } t.
\end{align}
Writing
$
 \big|E_h(\chi^h) - E(\chi)\big| = 2 \big(  E(\chi)- E_h(\chi^h)\big)_+ + E_h(\chi^h) - E(\chi)
$
and using the $\liminf$-inequality of the $\Gamma$-convergence of $E_h$ to $E$, we have
\begin{align*}
 \lim_{h\to 0}\left(  E(\chi)- E_h(\chi^h)\right)_+ = 0\quad \text{for a.\ e. } t.
\end{align*}
Then Lebesgue's dominated convergence, cf. (\ref{energy_dissipation_estimate}), and the convergence assumption (\ref{conv_ass}) yield
\begin{align*}
 \lim_{h\to0} \int_0^T \left|E_h(\chi^h) - E(\chi)\right| dt = 0
\end{align*}
and thus (\ref{EhtoEpointwise}) after passage to a subsequence.
Therefore, we can apply (\ref{pf_impl_of_conv_ass}) for a.\ e. $t$ and 
the time-dependent version follows from the time-independent one by Lebesgue's dominated convergence theorem and (\ref{energy_dissipation_estimate}).
\step{Step 2: Convergence of perimeters.} 
We claim that given $\chi^h\to \chi $ in $L^1(\torus,\R^\numphases)$ and
$ E_h(\chi^h)\to E(\chi)$, the individual perimeters converge in the following sense:
We have
 \begin{align*}
 F_h(\chi^h_1) \to F(\chi_1),\quad F_h(\chi^h_2) \to F(\chi_2)\quad \text{and} \quad  
 F_h(\chi^h_1+\chi^h_2) \to F(\chi_1+\chi_2).
 \end{align*}
 where $F_h$ and $F$ are the two-phase analogues of the (approximate) energies:
 \begin{align*}
 F_h(\tilde\chi) := \frac2\h \int \left(1-\tilde \chi\right) G_h\ast \tilde \chi \, dx \quad\text{and}\quad
 F(\tilde\chi) :=2 c_0 \int \left| \nabla \tilde \chi \right|.
\end{align*}
We will prove this claim by perturbing the functional $E_h$. We recall that the functionals $F_h$ $\Gamma$-converge to $F$ (see e.\ g. \cite{miranda2007short} or \cite{EseOtt14}).
Since the argument for the three cases work in the same way, we restrict ourself to the first case,
$F_h(\chi_1^h) \to F(\chi_1)$.
Since the matrix of surface tensions $\sigma$ satisfies the strict triangle inequality, we can perturb
the functionals $E_h$ in the following way:
For sufficiently small $\epsilon>0$, the associated surface tensions for the functional
$
 \chi \mapsto  E_h(\chi) - \epsilon F_h(\chi_1)
$
satisfy the triangle inequality so that approximate monotonicity, Lemma \ref{lem_AM_EO},
and consistency, Lemma \ref{lem_CONS_EO}, still apply.
Therefore, by Lemma \ref{lem_AM_EO}, we have for any $h_0\geq h$
\begin{align*}
 E_h(\chi^h) = & E_h(\chi^h) - \epsilon F_h( \chi^h_1) + \epsilon F_h( \chi^h_1) 
 \geq \left( \frac{\sqrt{h_0}}{\h + \sqrt{h_0}}\right)^{d+1}
 \left( E_{h_0}(\chi^h) - \epsilon F_{h_0}( \chi^h_1)\right) +  \epsilon F_h( \chi^h_1).
\end{align*}
By assumption, the left-hand side converges to $E(\chi)$.
Since for fixed $h_0$, $\chi\mapsto E_{h_0}(\chi)-\epsilon F_{h_0}(\chi_1)$
is clearly a continuous functional on $L^2$, the first right-hand side term converges as $h\to 0$.
Thus, for any $h_0>0$,
\begin{align*}
 \limsup_{h\to0} \epsilon F_h(\chi^h_1) \leq E(\chi) 
 - \left( E_{h_0}(\chi) - \epsilon F_{h_0}( \chi_1)\right).
\end{align*}
As $h_0 \to 0$, Lemma \ref{lem_CONS_EO} yields
\begin{align*}
  \limsup_{h\to0} F_h(\chi^h_1) \leq F(\chi_1).
\end{align*}
By the $\Gamma$-convergence we also have
\begin{align*}
  \liminf_{h\to0} F_h(\chi^h_1) \geq F(\chi_1)
\end{align*}
and thus the convergence
$
 F_h(\chi^h_1) \to F(\chi_1).
$
\step{Step 3: Conclusion.} 
We claim that given $\chi^h\to \chi $ in $L^1(\torus,\R^\numphases)$ and
$ E_h(\chi^h)\to E(\chi)$, for any $\zeta \in C^\infty(\torus)$
we have (\ref{pf_impl_of_conv_ass}).\\
We will not prove (\ref{pf_impl_of_conv_ass}) directly but prove that for any 
$ \zeta\in C^\infty(\torus)$ 
\begin{align}\label{pf_impl_of_conv_ass2}
F_h(\chi^h_1,\zeta) \to F(\chi_1,\zeta),\quad
F_h(\chi^h_2,\zeta) \to F(\chi_2,\zeta) \quad \text{and}\quad
F_h(\chi^h_1+\chi^h_2,\zeta) &\to F(\chi_1+\chi_2,\zeta)
\end{align}
for the localized functionals
\begin{align}\label{def F_h}
 F_h(\tilde\chi,\zeta) := \frac1\h \int \zeta \left[ \left(1-\tilde \chi\right) G_h\ast \tilde \chi +\tilde \chi\, G_h\ast \left(1-\tilde \chi\right) \right]  dx 
 \quad\text{and}\quad
 F(\tilde\chi,\zeta) := 2 c_0 \int \zeta \left| \nabla \tilde \chi \right|
\end{align}
instead.
This is indeed sufficient since for any $\chi_1,\chi_2$, we clearly have
\begin{align*}
 \chi_1 \,G_h\ast \chi_2 + \chi_2 \,G_h\ast \chi_1
= \left(1- \chi_1\right) G_h\ast \chi_1 
+ \left(1- \chi_2\right) G_h\ast \chi_2
- \left(1- (\chi_1+\chi_2)\right) G_h\ast (\chi_1+\chi_2)
\end{align*}
and (\ref{pf_impl_of_conv_ass2}) therefore implies (\ref{pf_impl_of_conv_ass}).

Now we give the argument for (\ref{pf_impl_of_conv_ass2}).
As before, we only prove one of the statements, namely $F_h(\chi^h_1,\zeta) \to F(\chi_1,\zeta)$.
For this we use two lemmas that we will prove in Section \ref{sec:curvature}.
First, by applying Lemma \ref{surf_la_cons_G_and_k}, which is the localized version of Lemma \ref{lem_CONS_EO},
we have for the functional $F_h$ instead of $E_h$ we have
$
  F_h(\chi_1) \to  F(\chi_1).
$
Then, by Lemma \ref{surf_la_err} we can estimate
$
 \left|F_h(\chi_1) - F_h(\chi^h_1)\right| \to 0
$
and thus conclude the proof.

Let us mention that one can also follow a different line of proof by localizing the monotonicity statement of Lemma \ref{lem_AM_EO} with a test function $\zeta$.
Since Lemma \ref{surf_la_err} seems more robust, we only prove the statement in this fashion.
\end{proof}

\begin{proof}[Proof of Proposition \ref{lem_dtX<<DX}]
We make use of the mesoscopic time scale $\tau$, see Remark \ref{timescales} for the notation.
\step{Argument for (i):} Let $\zeta\in C_0^\infty ((0,T)\times\torus)$.
We have to show that
\begin{align*}
 - \int_0^T \int \partial_t \zeta \, \chi_i \, dx\,dt 
\lesssim (1+T) E_0 \left\|\zeta\right\|_{\infty}.
\end{align*}
In this part we choose $\alpha = 1$.
Using the notation $\partial^{\tau}\zeta $ for the
 discrete time derivative $ \frac1\tau\left( \zeta(t+\tau) - \zeta(t)\right)$, by the smoothness of $\zeta$,
\begin{align*}
 \partial^\tau \zeta \to \partial_t \zeta \quad \text{uniformly in }
 (0,T)\times\torus \text{ as } h\to 0.
\end{align*}
Since $\chi^h \to \chi$ in $L^1((0,T)\times\torus) $, the product converges:
\begin{align*}
 \int_0^T \int \partial_t \zeta \, \chi_i \,dx \,dt 
= \lim_{h\to 0} \int_0^T \int \partial^{ \tau} \zeta \, \chi^h_i \,dx \,dt.
\end{align*}
Since $\supp \zeta$ is compact, by Lemma \ref{comp_lem_tau_shift} we have
\begin{align*}
 -\int_0^T \int \partial^{\tau} \zeta \, \chi^h_i \,dx \,dt
=  \int_0^T \int  \zeta \,\partial^{-\tau} \chi^h_i \,dx \,dt
\leq  \|\zeta\|_{\infty}
\int_\tau^T \int  \left|\partial^{- \tau} \chi^h_i\right|dx \,dt
\lesssim  \left( 1+T\right)E_0 \, \|\zeta\|_{\infty}
\end{align*}
for sufficiently small $h$.
\step{Argument for (ii):}
First we prove
\begin{align}\label{pf_dtX<<DX}
  -\int_0^T \int \partial_t \zeta \, \chi_i \,dx\,dt
 \lesssim \frac1\alpha \int_0^T \int \left|\zeta \right| \left| \nabla \chi_i\right|dt 
 + \alpha \iint \left|\zeta\right| d\mu
\end{align}
for any $\alpha>0$ and any $\zeta\in C_0^\infty((0,T)\times\torus)$.
We fix $\zeta$ and by linearity we may assume that $\zeta\geq0$ if we prove the inequality
with absolute values on the left-hand side.
We use the identity from above
\begin{align*}
 -\int_0^T \int \partial_t\zeta \, \chi_i\, dx\,dt
= \lim_{h\to0} \int_0^T \int \zeta \, \partial^{-\tau}\chi_i^h  dx\,dt.
\end{align*}
Setting 
\begin{align*}
\zeta^n:= \frac1h\int_{nh}^{(n+1)h}\zeta(t)\,dt
\end{align*}
to be the time average over a microscopic time interval, we have
\begin{align*}
 \left|\int_0^T \int \zeta \, \partial^{-\tau}\chi^h_i  dx\,dt\right|
\leq & \kmean \sum_{l=1}^\lmax \int \zeta^{\kmax l + k}
\big| \chi_i^{\kmax l + k} - \chi_i^{\kmax (l-1) + k}\big|\, dx.
\end{align*}
Now fix $k\in\{1,\dots, \kmax\}$. 
For simplicity, we will ignore $k$ at first.
We can argue as in the proof of Lemma \ref{comp_lem_tau_shift}, here with the localization $\zeta$:
By (\ref{comp_x-Ghx}) we have for any $\chi\in\{0,1\}$ 
\begin{align*}
 \frac1\h \int \zeta \left| G_h \ast \chi - \chi \right|\,dx 
 = \frac1\h\int \zeta \left[ \left(1-\chi\right) G_h\ast \chi
 + \chi\, G_h \ast \left(1-\chi\right)\right]dx = F_h(\chi,\zeta)
\end{align*}
with $F_h$ as in (\ref{def F_h}) and furthermore
\begin{align*}
 \left| \int \big(\zeta^{\kmax (l+1)} - \zeta^{\kmax l}\big) \left(1-\chi\right) G_h\ast \chi \,dx\right|
 \leq \|\partial_t \zeta  \|_{\infty}\alpha \h \int \left(1-\chi\right) G_h\ast \chi \,dx.
\end{align*}
Therefore, using (\ref{comp_x-xtilde}) we obtain
\begin{align*}
   \sum_{l=1}^{\lmax} \int \zeta^{\kmax l}\,
\big| \chi_i^{\kmax l} - \chi_i^{\kmax (l-1)}\big|\, dx
\lesssim &\sum_{l=1}^{\lmax} \int \zeta^{\kmax l} \big( \chi_i^{\kmax l} - \chi_i^{\kmax (l-1)} \big) G_h \ast \big( \chi_i^{\kmax l} - \chi_i^{\kmax (l-1)} \big) dx\\
&+ \frac \tau \alpha \sum_{l=1}^{\lmax} F_h(\chi_i^{\kmax l},\zeta^{\kmax l}) +  \h \left\|\partial_t\zeta\right\|_{\infty} \tau \sum_{l=1}^{\lmax} E_h(\chi^{\kmax l}),
\end{align*}
where the last right-hand side term vanishes as $h\downarrow 0$ by (\ref{energy_dissipation_estimate}).
For the first right-hand side term we note that for any $\zeta\in C^\infty(\torus)$ and any $\chi,\,\tilde\chi\in \{0,1\}$ we have
\begin{align*}
 &\left|\int \zeta \left[ G_{h/2}\ast\left( \chi-\tilde\chi\right)\right]^2dx
- \int \zeta \left( \chi-\tilde\chi\right) G_{h}\ast\left( \chi-\tilde\chi\right) dx\right|\\
&= \left|\int \left( \zeta \, G_{h/2}\ast \left( \chi-\tilde\chi\right)
-G_{h/2}\ast \left[ \zeta \left( \chi-\tilde\chi\right)\right]\right)
G_{h/2}\ast\left( \chi-\tilde\chi\right) dx\right|\\
&\leq \int G_{h/2}(z) \int \left|\zeta(x+z) -\zeta(x) \right| \left| \chi-\tilde\chi\right|(x+z)
\left|G_{h/2}\ast\left( \chi-\tilde\chi\right)\right|(x) \,dx\,dz\\
&\lesssim \left\|\nabla \zeta\right\|_\infty \h \int \frac{|z|}{\h} G_{h/2}(z)\,dz\,
\int \left| \chi-\tilde\chi\right|dx\\
&\lesssim  \left\|\nabla \zeta\right\|_\infty \h \int \left| \chi-\tilde\chi\right|dx,
\end{align*}
so that we can replace the first right-hand side term by 
\begin{align*}
  \sum_{l=1}^{\lmax} \int \zeta^{\kmax l} \left( G_{h/2} \ast \big( \chi_i^{\kmax l} - \chi_i^{\kmax (l-1)} \big)\right)^2 dx,
\end{align*}
up to an error that vanishes as $h\downarrow 0$, due to the above calculation and e.\ g.\ Lemma \ref{comp_lem_hoelder}.
As in (\ref{triangle and jensen}) for $-E_h$, now for this localized version, we can use the triangle inquality and Jensen's inequality to bound this term by
\begin{align*}
  \sum_{l=1}^{\lmax} \kmax \sum_{n=\kmax(l-1)+1}^{\kmax l} \int \zeta^{\kmax l} \left( G_{h/2} \ast \big( \chi_i^{n} - \chi_i^{n-1} \big)\right)^2 dx
  \leq  \alpha \iint \zeta \, d\mu_h +o(1),
\end{align*}
as $h\downarrow 0$, where $\mu_h$ is the (approximate) dissipation measure defined in (\ref{diss meas}).
Therefore we have
\begin{align*}
  \sum_{l=1}^{\lmax}  \int \zeta^{\kmax l}\,
\big| \chi_i^{\kmax l} - \chi_i^{\kmax (l-1)}\big|\, dx
\lesssim &  \frac \tau\alpha \sum_{l=1}^\lmax F_h(\chi_i^{\kmax l},\zeta^{\kmax l})
+ \alpha \iint \zeta \,d\mu_h 
+ o(1),
\end{align*}
as $h\downarrow 0$.
Taking the mean over the $k$'s we obtain
\begin{align*}
\left| \int_0^T \int \zeta \, \partial^{-\tau}\chi_i^h  dx\,dt\right|
\lesssim 
\frac1\alpha \int_0^T F_h(\chi_i^h,\zeta)\,dt
+ \alpha \iint \zeta \,d\mu_h
+ o(1).
\end{align*}
Passing to the limit $h\to0$, (\ref{impl_of_conv_ass}), which is guaranteed by the 
convergence assumption (\ref{conv_ass}), implies (\ref{pf_dtX<<DX}).\\
Now let $U \subset (0,T)\times \torus$ be open such that
\begin{align*}
 \iint_U \left| \nabla \chi_i \right|dt=0.
\end{align*}
If we take $\zeta\in C_0^\infty(U)$, the first term on the right-hand side of (\ref{pf_dtX<<DX}) 
vanishes and therefore
\begin{align*}
-\int_0^T \int \partial_t \zeta \,  \chi_i \,dx\,dt
\lesssim \alpha \iint \left|\zeta \right| d\mu.
\end{align*}
Since the left-hand side does not depend on $\alpha$, we have
\begin{align*}
 -\int_0^T \int \partial_t \zeta \,  \chi_i\,dx\,dt \leq 0.
\end{align*}
Taking the supremum over all $\zeta\in C_0^\infty(U)$ yields
\begin{align*}
 \iint_U \left| \partial_t \chi_i \right| = 0.
\end{align*}
Thus, $\partial_t \chi_i$ is absolutely continuous w.\ r.\ t.\ 
$\left|\nabla \chi_i\right|dt$ and the Radon-Nikodym theorem completes the proof.
\step{Argument for (iii):}
We refine the estimate in the argument for (ii).
Instead of estimating the right-hand side of 
(\ref{pf_dtX<<DX}) and optimizing afterwards, which leads to a weak $L^2$-bounds, we localize.
Starting from (\ref{pf_dtX<<DX}), we notice that we can localize with the test function 
$\zeta$. Thus, we can post-process the estimate and obtain
\begin{align*}
 \left| \int_0^T\int V_i \,\zeta \left| \nabla \chi_i\right| dt \right| 
 \leq C  \int_0^T\int \frac1{\alpha} \left|\zeta\right|
 \left| \nabla \chi_i\right| dt
 +   C\iint \alpha\left|\zeta\right| d\mu
\end{align*}
for any integrable $\zeta:(0,T)\times \torus \to \R$, any measurable
$\alpha \colon (0,T)\times \torus \to (0,\infty)$ and some constant $C<\infty$ which depends only
on the dimension $d$, the number of phases $\numphases$ and the matrix of surface tensions $\sigma$.
Now choose 
\begin{align*}
 \zeta = V_i \quad \text{and}\quad \alpha = \frac{2C}{| V_i|},
\end{align*}
where we set $\alpha:= 1$ if $V_i=0$, in which case all other integrands vanish.
Then, the first term on the right-hand side can be absorbed in the left-hand side and we obtain
\begin{align*}
  \int_0^T\int V_i^2 \left| \nabla \chi_i\right| dt \lesssim \mu([0,T]\times\torus) \lesssim E_0.
\end{align*}
\end{proof}

%
%
%
%
%
%
%
%
%
%
%
\section{Energy Functional and Curvature}\label{sec:curvature}
It is a classical result by Reshetnyak \cite{reshetnyak1968weak} that the convergence
$\chi^h\to \chi$ in $L^1$ and 
\begin{align*}
 \int | \nabla\chi^h|\to \int \left| \nabla\chi\right|=:E(\chi)
\end{align*}
imply convergence of the first variation
\begin{align*}
 \delta E(\chi,\xi) = \int \left( \nabla \cdot \xi - \nu\cdot \nabla \xi \,\nu \right)
 \left| \nabla \chi \right|.
\end{align*}
A result by Luckhaus and Modica \cite{luckhaus1989gibbs} shows that this may extend to a $\Gamma$-convergence
situation, namely in case of the Ginzburg-Landau functional
\begin{align*}
 E_h(u) := \int h \left| \nabla u\right|^2 + \frac1h \left( 1-u^2\right)^2 dx.
\end{align*}
We show that this also extends to our $\Gamma$-converging functionals $E_h$.
Let us first address why the first variation of the approximate energies is of interest in view of our
minimizing movements scheme.
We recall (\ref{MMinterpretation}): the approximate solution $\chi^n$ at time $nh$
minimizes $E_h(\chi) - E_h(\chi-\chi^{n-1})$ among all $\chi$.
The natural variations of such a minimization problem are \emph{inner variations},
i.\ e.\ variations of the independent variable.
Given a vector field $\xi\in C^\infty(\torus,\R^d)$ and an admissible $\chi$, we define the deformation
$\chi_s$ of $\chi$ along $\xi$ by the distributional equation
\begin{align*}
  \frac{\partial}{\partial s} \chi_{i,s} + \nabla \chi_{i,s} \cdot \xi =0,\quad 
  \chi_{i,s}\big|_{s=0} = \chi_i,
\end{align*}
which means that the phases are deformed by the flow generated through $\xi$.
The inner variation $\delta E_h$ of the energy $E_h$ at $\chi$ along the vector field
$\xi$ is then given by
\begin{align}\label{variation_E}
 \delta E_h (\chi,\xi):= \frac{d}{ds} E_h(\chi_s)\big|_{s=0}
=  \frac2{\sqrt{h}} \sum_{i,j} \sigma_{ij} \int \chi_i \, G_h\ast\left( -\nabla \chi_j\cdot \xi\right) dx.
\end{align}
For an admissible $\tilde \chi$ the inner variation
of the metric term $-E_h(\chi-\tilde \chi)$ is given by
\begin{align}\label{variation_D}
  -\delta E_h(\,\cdot\,-\tilde \chi)(\chi,\xi):= \frac{d}{ds} \big( -E_h(\chi_s-\tilde \chi) \big)\big|_{s=0}
  =  \frac2{\sqrt{h}} \sum_{i,j} \sigma_{ij} \int \left( \chi_i -\tilde \chi_i\right)  G_h\ast\left( \nabla \chi_j\cdot \xi\right) dx.
\end{align}
The (chosen and not necessarily unique) minimizer $\chi^n$ 
in Algorithm \ref{MBO_arbitrary_surface_tensions} therefore
satisfies the Euler-Lagrange equation
\begin{align}\label{ELG}
 \delta E_h(\chi^n,\xi) - \delta E_h(\,\cdot\,- \chi^{n-1})(\chi^n,\xi) = 0
\end{align}
for any vector field $\xi\in C^\infty(\torus,\R^d)$.
\subsection{Results}\label{sub:surf results}
%
%
The goal of this section is to prove the following statement about the convergence of the first term
in the Euler-Lagrange equation.
\begin{prop}\label{surf_prop_integrated_in_time}
Let  $\chi^h,\, \chi \colon (0,T)\times \torus \to \{0,1\}^\numphases$ be such that
$\chi^h(t),\, \chi(t)$ are admissible in the sense of (\ref{admissible}) and $E(\chi(t)) < \infty$ for a.\ e. $t$. Let
\begin{align}
 \chi^h & \longrightarrow \chi \quad \text{a.\ e. in } (0,T)\times \torus,\label{surf_conv_ass1_spacetime}
\end{align}
and furthermore assume that
\begin{align}
 \int_0^T E_h(\chi^h)\,dt &\longrightarrow \int_0^T E(\chi)\,dt.\label{surf_conv_ass3_spacetime}
\end{align}
Then, for any $\xi\in C_0^\infty((0,T)\times\torus,\R^d)$, we have
\begin{align*}
\lim_{h\to 0} \int_0^T\delta E_h (\chi^h,\xi) \,dt=
c_0 \sum_{i,j} \sigma_{ij} \int_0^T \int \left( \nabla\cdot \xi - \nu_i\cdot \nabla \xi\,\nu_i\right)
\frac12\left( \left|\nabla\chi_i\right| +\left|\nabla\chi_j\right|
-\left|\nabla(\chi_i+\chi_j)\right|\right) dt.
\end{align*}
\end{prop}
%

%
%
It is easy to reduce the statement to the following time-independent statement.
\begin{prop}\label{surf_prop}
Let $\chi^h,\, \chi \colon \torus \to \{0,1\}^\numphases$ 
 be admissible in the sense of (\ref{admissible}) with $E(\chi) <\infty$ such that
\begin{align}
 \chi^h & \longrightarrow \chi \quad \text{a.\ e.},\label{surf_conv_ass1}
 \end{align}
 and furthermore assume that
 \begin{align}
E_h(\chi^h)&\longrightarrow E(\chi).\label{surf_conv_ass3}
\end{align}
Then, for any $\xi\in C^\infty(\torus,\R^d)$, we have
\begin{align*}
\lim_{h\to 0} \delta E_h (\chi^h,\xi) =
c_0 \sum_{i,j} \sigma_{ij} \int \left( \nabla\cdot \xi - \nu_i\cdot \nabla \xi\,\nu_i\right)
 \frac12\left( \left|\nabla\chi_i\right| +\left|\nabla\chi_j\right| -\left|\nabla(\chi_i+\chi_j)\right|\right).
\end{align*}
\end{prop}

\begin{rem}
 Proposition \ref{surf_prop} and all other statements in this section hold also in a more general context.
 We do not need the approximations $\chi^h$ to be characteristic functions. In fact the statements hold for any sequence
 $u^h \colon \torus \to [0,1]^\numphases$ with $\sum_i u_i^h =1$ a.\ e.\ converging to some $\chi\colon \torus \to \{0,1\}^\numphases$ with
 $E(\chi) <\infty$ in the sense of (\ref{surf_conv_ass1})--(\ref{surf_conv_ass3}).
\end{rem}

%
%
The following first lemma brings the first variation $\delta E_h$ of $E_h$ into a more convenient form,
up to an error vanishing as $h\to0$ because of the smoothness of $\xi$.
Already at this stage one can see the structure 
\begin{align*}
 \nabla \cdot \xi - \nu \cdot \nabla \xi \, \nu=\nabla \xi\colon (Id-\nu\otimes \nu)
\end{align*}
in the first variation of $E$ in the form of
$\nabla\xi \colon ( G_h Id-h\nabla^2 G_h)$ on the level of the approximation.
\begin{lem}\label{surf_la_ELG} 
Let $\chi$ be admissible and $\xi\in C^\infty(\torus,\R^d)$ then
\begin{align}\label{surf_eq_ELG}
 \delta E_h (\chi,\xi)
= \frac{1}{\sqrt{h}} \sum_{i,j} \sigma_{ij} \int \chi_i\, 
\nabla \xi : \left(  G_hId - h\nabla^2 G_h\right) \ast
 \chi_j \,dx
 + O\left(\|\nabla^2 \xi\|_\infty E_h(\chi)\h\right).
\end{align}
\end{lem}
We have already seen in Lemma \ref{la_impl_conv_ass} that we can pass to the limit in the term involving only the kernel $G_h Id$:
\begin{align*}
 \lim_{h\to0}  \frac{1}{\sqrt{h}} \sum_{i,j} \sigma_{ij} \!
\int\!\, \zeta \, \chi^h_i \,G_h\ast \chi^h_j \,dx 
= c_0 \sum_{i,j} \sigma_{ij}\! \int \zeta\, \frac12\left(|\nabla\chi_i| 
+|\nabla\chi_j| 
- |\nabla(\chi_i+\chi_j)|\right),
\end{align*}
where now $\zeta =\nabla \cdot \xi$.
The next proposition shows that we can also pass to the limit in the term involving the second derivatives $h \nabla^2G_h$ of the kernel,
which yields the projection $\nu\otimes \nu$ onto the normal direction in the limit.

%
%
\begin{prop}\label{surf_prop_A}
Let $\chi^h,\, \chi$ satisfy the convergence assumptions 
(\ref{surf_conv_ass1}) and (\ref{surf_conv_ass3}). Then for any $A\in C^\infty([0,\Lambda)^d,\R^{d\times d})$
 \begin{align*}
  \lim_{h\to0}  \frac{1}{\sqrt{h}} \sum_{i,j} \sigma_{ij} \!
\int\!\chi^h_i\, A : h\nabla^2 G_h\ast \chi^h_j\, dx 
= c_0 \sum_{i,j} \sigma_{ij}\! \int \! \nu\! \cdot\! A\, \nu\, \frac12\left(|\nabla\chi_i| 
+|\nabla\chi_j| 
- |\nabla(\chi_i+\chi_j)|\right).
 \end{align*}
\end{prop}
%

%
%
The following two statements are used to prove Proposition \ref{surf_prop_A}.
%
%
%
The following lemma yields in particular the construction part in the $\Gamma$-convergence result of $E_h$
to $E$.
We need it in a localized form; the proof closely follows 
the proof of Lemma 4 in Section 7.2 of \cite{EseOtt14}.
\begin{lem}[Consistency]\label{surf_la_cons_G_and_k}
 Let $\chi\in BV(\torus,\{0,1\}^\numphases)$ be admissible in the sense of (\ref{admissible}). Then for any $\zeta\in C^\infty(\torus)$
 \begin{align*}
 \lim_{h\to0}\frac{1}{\sqrt{h}} \sum_{i,j} \sigma_{ij} \int \zeta\,\chi_i\, G_h\ast \chi_j \,dx 
 =& c_0 \sum_{i,j}\sigma_{ij}  \int \zeta \,\frac12\left(|\nabla\chi_i| 
 +|\nabla\chi_j| -|\nabla(\chi_i+\chi_j)|\right)
\end{align*}
 and for any  $A\in C^\infty(\torus,\R^{d\times d})$ 
\begin{align*}
 \lim_{h\to0}\frac{1}{\sqrt{h}} \sum_{i,j} \sigma_{ij} \int \chi_i\,A\colon  h \nabla^2 G_h\ast \chi_j \,dx 
 =& c_0 \sum_{i,j}\sigma_{ij}  \int \nu \cdot A \, \nu \,\frac12\left(|\nabla\chi_i| 
 +|\nabla\chi_j| -|\nabla(\chi_i+\chi_j)|\right).
\end{align*}
\end{lem}
%

%
%
The next lemma shows that under our convergence assumption of $\chi^h$ to $\chi$, the corresponding spatial
covariance functions $f_h$ and $f$ are very close and allows us to pass from Lemma \ref{surf_la_ELG} and Lemma \ref{surf_la_cons_G_and_k} to
Proposition \ref{surf_prop}. 
\begin{lem}[Error estimate]\label{surf_la_err}
Let $\chi^h,\, \chi$ satisfy the convergence assumptions 
(\ref{surf_conv_ass1}) and (\ref{surf_conv_ass3}) and let $k$ be a non-negative kernel such that
\begin{align*}
 k(z)\leq p(|z|) G(z) 
\end{align*}
for some polynomial $p$.
Then
\begin{align}\label{surf_conv_err}
  \lim_{h\to0}\frac1\h \int k_h(z) |f_h( z)-f( z)|\,dz = 0,
\end{align}
where
\begin{align*}
 f_{h}(z) &:= \sum_{i,j} \sigma_{ij} \int \chi^{h}_i(x)\chi^{h}_j(x+z)\,dx\quad \text{and} \quad
 f(z) := \sum_{i,j} \sigma_{ij} \int \chi_i(x)\chi_j(x+z)\,dx.
\end{align*}
\end{lem}

\subsection{Proofs}\label{sub:surf proofs}
\begin{proof}[Proof of Proposition \ref{surf_prop_integrated_in_time}]
The proposition is an immediate consequence of the time-independent analogue, Proposition \ref{surf_prop}.
Indeed, according to Step 1 in the proof of Lemma \ref{la_impl_conv_ass}
we have $E_h(\chi^h) \to E(\chi)$ for a.\ e.\ $t$.
Thus all conditions of Proposition \ref{surf_prop} are fulfilled.
Proposition \ref{surf_prop_integrated_in_time} 
follows then from Lebesgue's dominated convergence theorem.
\end{proof}

%
%
\begin{proof}[Proof of Proposition \ref{surf_prop}]
We may apply Lemma \ref{surf_la_ELG} for $\chi^h$ and obtain by the energy-dissipation estimate
(\ref{energy_dissipation_estimate}) that
\begin{align*}
 \delta E_h (\chi,\xi^h)
= \frac{1}{\sqrt{h}} \sum_{i,j} \sigma_{ij} \int \chi^h_i\, 
\nabla \xi : \left( Id \, G_h - h\nabla^2 G_h\right) \ast
 \chi^h_j \,dx
 + O\left(\|\nabla^2 \xi\|_\infty E_0 \h\right).
\end{align*}
Applying Proposition \ref{surf_prop_A} for the kernel $\nabla^2 G$ with $\nabla \xi$ playing the role of the matrix field $A$
and Lemma \ref{la_impl_conv_ass} for the kernel $G$
with $\zeta = \nabla \cdot \xi$, we can conclude the proof.
\end{proof}

\begin{proof}[Proof of Lemma \ref{surf_la_ELG}]
Recall the definition of $\delta E_h$ in (\ref{variation_E}).
Since $ -\nabla \tilde \chi\cdot \xi = -\nabla \cdot (\tilde \chi\,\xi) +\tilde  \chi\left(\nabla\cdot\xi\right) $
for any function $\tilde \chi\colon \torus \to \R$, we can rewrite the integral on the right-hand side of (\ref{variation_E}):
\begin{align*}
 \int  \chi_i \, G_h\ast\left( -\nabla \chi_j\cdot \xi\right) dx
=&  \int - \chi_i \, G_h\ast\left( \nabla \cdot(\chi_j\,\xi)\right) + \chi_i \,
 G_h\ast\left( \chi_j\,\nabla \cdot \xi\right) dx\\
=& \int - \chi_i\, \nabla G_h \ast \left(\chi_j\, \xi \right)  
+ \chi_j \left(\nabla\cdot\xi\right)  G_h\ast \chi_i \,dx.
\end{align*}
Let us first turn to the first right-hand side term.
For fixed $(i,j)$, we can collect the two
terms in the sum that belong to the interface between phases $i$ and $j$ and obtain by the antisymmetry of the kernel $\nabla G_h$ 
that the resulting term with the prefactor $\frac{2\sigma_{ij}}{\h}$ is
\begin{align*}
 & \int - \chi_i\, \nabla G_h \ast \left(\chi_j\, \xi \right)
  - \chi_j\, \nabla G_h \ast \left(\chi_i\, \xi \right) dx =  \int \chi_i(x) \int \left( \xi(x) - \xi(x-z)\right)\cdot \nabla G_h(z)\,  \chi_j(x-z)\, dz \,dx.
 \end{align*}
A Taylor expansion of $\xi$ around $x$ gives the first-order term
\begin{align*}
\frac{2\sigma_{ij}}{\h} \int \chi_i(x) \int \left( \nabla \xi(x) \,z \right)\cdot \nabla G_h(z)\,  \chi_j(x-z)\, dz \,dx.
\end{align*}
Now we argue that the second-order term is controlled by $\|\nabla^2 \xi\|_\infty E_h(\chi) \h$.
Indeed, since $|z|^3G(z) \lesssim G_2(z)$, the contribution of the second-order term is controlled by
\begin{align*}
 &\|\nabla^2 \xi\|_\infty \frac1\h \sum_{i,j} \sigma_{ij} \int |z|^2 \frac{|z|}{h} G_h(z) \int \chi_i(x)\, \chi_j(x+z)\,dx\,dz\\
 &\lesssim \|\nabla^2 \xi\|_\infty \sum_{i,j} \sigma_{ij} \int G_{2h}(z) \int \chi_i(x)\,\chi_j(x+z)\,dx\,dz \\ 
 &\sim  \|\nabla^2 \xi\|_\infty \h \, E_{2h}(\chi).
 \end{align*}
Using the approximate monotonicity (\ref{approximate monotonicity}) of $E_h$, we have suitable control over this term.
After distributing the first-order term on both summand $(i,j)$ and $(j,i)$ we therefore have
\begin{align*}
 \delta E_h (\chi,\xi)
=& \frac{1}{\sqrt{h}} \sum_{i,j} \sigma_{ij} \int \chi_i(x)  
\nabla \xi(x) : \int \left(2  G_h(z) Id+ z \otimes \nabla G_h(z)\right) \chi_j(x+z) \,dz\,dx\\
 &+ O\left(\|\nabla^2 \xi\|_\infty E_h(\chi)\h\right)
\end{align*}
and since $\nabla^2 G(z) = -Id\, G - z \otimes \nabla G(z)$, we conclude the proof. 
\end{proof}

\begin{proof}[Proof of Proposition \ref{surf_prop_A}]
By Lemma \ref{surf_la_cons_G_and_k} we know that the term converges if we take $\chi$ instead of the approximation $\chi^h$ on the left-hand side of the statement.
Lemma \ref{surf_la_err} in turn controls the error by substituting $\chi^h$ by $\chi$ on the left-hand side.
\end{proof}

\begin{proof}[Proof of Lemma \ref{surf_la_cons_G_and_k}]
Our main focus in this proof lies on the anisotropic kernel $\nabla^2 G$. The statement for $G$ is -- up to the localization -- already contained in
the proof of Lemma 4 in Section 7.2 of \cite{EseOtt14}.
\step{Step 1: Reduction of the statement to a simpler kernel.}
Since $\nabla^2G(z)$ is a symmetric matrix, the inner product
\begin{align*}
 A : \nabla^2G(z) = A^{\textup{sym}} : \nabla^2G(z).
\end{align*}
depends only the symmetric part $A^{\textup{sym}}$ of $A$; 
hence w.\ l.\ o.\ g.\ let $A$ be a symmetric matrix field. 
But then there exist functions $\zeta_{ij}\in C^\infty(\torus)$, such that
\begin{align*}
 A(x) = \sum_{i,j} \frac12 \zeta_{ij}(x) \left( e_i \otimes e_j +e_j \otimes e_i \right).
\end{align*}
We also note
\begin{align*}
 e_i \otimes e_j +e_j \otimes e_i 
 = \left( e_i + e_j\right) \otimes \left( e_i + e_j\right)
 - \left( e_i \otimes e_i +e_j \otimes e_j\right).
\end{align*}
Hence by linearity it is enough to prove the statement for $A$ of the form
\begin{align*}
 A(x) = \zeta(x)\; \xi \otimes \xi
\end{align*}
for some $\xi\in\sphere$.
By rotational invariance we may assume
\begin{align*}
  A(x) = \zeta(x)\; e_1 \otimes e_1.
\end{align*}
Hence the statement can be reduced to
\begin{align}\label{surf_conv_red}
 \lim_{h\to0}&  \frac{1}{\sqrt{h}} \sum_{i,j} \sigma_{ij} \int \zeta\, \chi_i\,
h\partial_1^2 G_h\ast \chi_j \,dx =c_0 \sum_{i,j} \sigma_{ij}  \int \zeta \,\nu_1^2\,\frac12 \left(|\nabla\chi_i| 
+|\nabla\chi_j| 
- |\nabla(\chi_i+\chi_j)|\right)
\end{align}
for any $\zeta\in C^\infty([0,\Lambda)^d).$
In the following we will show that for any such $\zeta$ and 
$\chi,\, \tilde \chi\in BV(\torus,\{0,1\})$ such that
\begin{align}\label{surf_cons_XX=0}
 \chi\,\tilde \chi = 0 \quad \ae
\end{align}
and for the anisotropic kernel $k(z) = z_1^2 G(z)$ we have
\begin{align}\label{surf_cons_proof1}
 \lim_{h\to0} \frac{1}{\sqrt{h}} \int \zeta\,\tilde\chi\, k_h\ast \chi \,dx 
 =& c_0  \int \zeta  \left(  \nu_1^2 +1 \right) \frac12 \left( |\nabla\chi| +|\nabla\tilde\chi| 
-|\nabla(\chi+\tilde\chi)|\right).
\end{align}
The analogous statement for the Gaussian kernel $G$ instead of the anisotropic kernel $k$ is -- up to the localization with $\zeta$ -- contained in \cite{EseOtt14}.
In that case the right-hand side of (\ref{surf_cons_proof1}) turns into the localized energy, i.\ e. replacing the anisotropic term $(\nu_1^2+1)$ by $1$.
Since $\partial_1^2 G(z) = \left(z_1^2-1\right)G(z)$ it is indeed sufficient to prove (\ref{surf_cons_proof1}).
We will prove this in five steps.
Before starting, we introduce spherical coordinates $z=r\xi$ on the left-hand side:
\begin{align}\label{surf_cons_red_0}
\frac{1}{\sqrt{h}} \int \zeta\, \tilde\chi\, k_h\ast \chi \,dx  
= &\frac{1}{\h} \int k(z) \int \zeta(x)\tilde\chi(x) \chi(x+\h z)\,dx\, dz \notag\\
=& \int_0^\infty G(r) r^{d+2} \frac{1}{\h r}\int_\sphere \xi_1^2  \int \zeta(x)\tilde\chi(x) \chi(x+\h r \xi)\,dx\,d\xi\, dr.
\end{align}
In the following two steps of the proof, we
simplify the problem by disintegrating in $r$ (Step 2) and $\xi$ (Step 3).
Then we explicitly calculate an integral that arises in the second reduction 
and which translates the anisotropy of the kernel $k$ into a geometric information about the normal (Step 4).
We simplify further by disintegration in the vertical component (Step 5) and 
conclude by solving the one-dimensional problem (Step 6). 
\step{Step 2: Disintegration in $r$.} We claim that it is sufficient to show
\begin{align}\label{surf_cons_red_1}
 \lim_{h\to0}&
 \frac{1}{\h } \int_\sphere \xi_1^2  
 \int \zeta(x)\tilde\chi(x) \chi(x+\h \xi)\,dx\,d\xi\notag \\
&\qquad\qquad =
\frac{|B^{d-1}|}{d+1} \int \zeta \left(\nu_1^2+1\right) \frac12
\left( |\nabla\chi| +|\nabla\tilde\chi| 
-|\nabla(\chi+\tilde\chi)|\right).
\end{align}
Indeed, note that since $G(z)=G(|z|)$ and $\frac{d}{dr} G(r) = -r G(r)$ we have, using integration by parts,
\begin{align*}
\int_0^\infty G(r)r^{d+2} \,dr = - \int_0^\infty\frac{d}{dr}(G(r))r^{d+1} \,dr
= (d+1) \int_0^\infty G(r)r^{d} \,dr.
\end{align*}
Replacing $\h$ by $\h\,r$ on the left-hand side of (\ref{surf_cons_red_1})
and integrating w.\ r.\ t.\ the non-negative measure $G(r)r^{d}dr$ 
and using the equality from above shows that (\ref{surf_cons_red_1}), 
in view of (\ref{surf_cons_red_0}), formally implies (\ref{surf_cons_proof1}). 
To make this step rigorous, we use Lebesgue's dominated convergence theorem.
A dominating function can be obtained as follows:
\begin{align*}
  \left|\frac{1}{\h r}\int_\sphere \xi_1^2  \int \zeta(x)\tilde\chi(x) \chi(x+\h r \xi)\,dx\,d\xi\right|
&\overset{ (\ref{surf_cons_XX=0})}{=} \left|\frac{1}{\h r}\int_\sphere \xi_1^2 \int \zeta(x)\tilde\chi(x) \left(\chi(x+\h r \xi) - \chi(x)\right) \,dx\,d\xi\right|\\
&\,\,\leq \|\zeta\|_\infty \frac{1}{\h r}\int_\sphere   \int \left|\chi(x+\h r \xi) - \chi(x)\right| \,dx\,d\xi\\
&\,\, \leq \|\zeta\|_\infty \,|\sphere| \int |\nabla \chi|,
\end{align*}
which is finite and independent of $r$. 
Hence, it is integrable w.\ r.\ t.\ the finite measure $G(r)r^{d+2}dr$.
\step{Step 3: Disintegration in $\xi$.}
We claim that it is sufficient to show that for each $\xi\in \sphere$,
\begin{align}\label{surf_cons_red_2}
 \lim_{h\to0}& \frac{1}{\h} \int \zeta(x)\tilde\chi(x) 
 \left(\chi(x+\h \xi) + \chi(x-\h \xi)\right)\,dx 
=  \int\zeta  \left|\xi\cdot\nu\right|\frac12
\left( |\nabla \chi| +|\nabla\tilde\chi| - |\nabla(\chi+\tilde\chi)|\right).
\end{align}
Indeed, if we integrate w.\ r.\ t.\ the non-negative measure $\frac12 \xi_1^2 d\xi$ we obtain the left-hand side of (\ref{surf_cons_red_1}) 
from the left-hand side of (\ref{surf_cons_red_2}).
At least formally, this is obvious because of the symmetry under $\xi\mapsto -\xi$.
The dominating function to interchange limit and integration is obtained as in Step 1:
\begin{align*} 
 &\left| \frac{1}{\h} \int \zeta(x)\tilde\chi(x) \left(\chi(x+\h \xi) + \chi(x-\h \xi)\right)\,dx \right|\\
&\quad\overset{ (\ref{surf_cons_XX=0})}{\leq} 
\frac{1}{\h} \sup|\zeta| \int  \left|\chi(x+\h \xi)  - \chi(x)\right|+ 
\left| \chi(x-\h \xi)-\chi(x)\right|\,dx
\leq 2 \|\zeta\|_\infty \int |\nabla \chi|.
\end{align*}
For the passage from the right-hand side of (\ref{surf_cons_red_2})
to the right-hand side of (\ref{surf_cons_red_1}) we note that since
\begin{align*}
 \int_\sphere \xi_1^2  \int\zeta  \left|\xi\cdot\nu\right| |\nabla \chi| \frac12d\xi 
 =\frac12 \int \int_\sphere \xi_1^2 \left|\xi\cdot\nu\right| d\xi\, \zeta  \left|\nabla\chi\right|
\end{align*}
and $|\nu|=1 $ $|\nabla\chi|$- a.\ e. it is enough to prove
\begin{align}\label{surf_cons_1_plus_nu12}
 \frac12\int_\sphere \xi_1^2 |\xi\cdot \nu| \,d\xi =\frac{|B^{d-1}|}{d+1} \left(\nu_1^2+1\right) \quad \text{for all }\nu\in \sphere
\end{align}
to obtain the equality for the right-hand side.
\step{Step 4: Argument for (\ref{surf_cons_1_plus_nu12}).}
By symmetry of $\int_\sphere d\xi$ under the reflection that maps $e_1$ into $\nu$, we have
\begin{align*}
 \int_\sphere \xi_1^2 |\xi\cdot \nu| \,d\xi = \int_\sphere (\xi\cdot \nu)^2 |\xi_1| \,d\xi.
\end{align*}
Applying the divergence theorem to the vector field $|\xi_1|\left( \xi\cdot \nu \right) \nu$,
we have
\begin{align*}
 \int_\sphere (\xi\cdot \nu)^2 |\xi_1| \,d\xi 
 = \int_B \nabla \cdot \big(|\xi_1|\left( \xi\cdot \nu \right) \nu \big) d\xi.
\end{align*}
Since 
$\nabla \cdot \left(|\xi_1|\left( \xi\cdot \nu \right) \nu \right)
= \sign \xi_1 \left( \xi\cdot \nu \right) \nu_1 + |\xi_1|$,
the right-hand side is equal to
\begin{align*}
 \left( \int_B \sign \xi_1 \,\xi\, d\xi\right) \cdot \nu \, \nu_1 + \int_B |\xi_1| \,d\xi.
\end{align*}
By symmetry of $d\xi$ under rotations that leave $e_1$ invariant, we see that
$\int_B \sign \xi_1 \,\xi\, d\xi$ points in direction $e_1$, so that the above reduces to
\begin{align*}
 \left(\nu_1^2+1\right) \int_B \left| \xi_1\right|d\xi.
\end{align*}
We conclude by observing
\begin{align*}
 \int_B \left| \xi_1\right|d\xi 
 = \int_{-1}^1 \left| \xi_1\right| | B^{d-1}|\left( 1-\xi_1^2\right)^{\frac{d-1}{2}} d\xi_1
 = 2 |B^{d-1}| \int_0^1 \frac{d}{d\xi_1} 
 \left[ -\frac1{d+1}\left( 1-\xi_1^2\right)^{\frac{d-1}{2}}\right] d\xi_1
 = 2\frac{| B^{d-1}|}{d+1}.
\end{align*}
\step{Step 5: One-dimensional reduction.} 
The problem reduces to the one-dimensional analogue, namely: 
For all $\chi,\,\tilde\chi\in BV([0,\Lambda),\{0,1\})$ such that
\begin{align}\label{surf_cons_XX1d}
 \chi\,\tilde \chi = 0 \quad \ae
\end{align}
and every $\zeta\in C^\infty([0,\Lambda))$ we have
\begin{align}\label{surf_cons_red_3}
 \lim_{h\to0} \frac{1}{\h} \int_0^\Lambda \zeta(s)\tilde\chi(s) 
\left(\chi(s+\h) + \chi(s-\h)\right) ds
=  \int_0^\Lambda\zeta  \,\frac12\Big(|\frac{d \chi }{ds}| 
+ |\frac{d \tilde \chi}{ds}| 
-  |\frac{d(\chi+\tilde\chi)}{ds}|\Big).
\end{align}
Indeed, by symmetry, it suffices to prove (\ref{surf_cons_red_2}) for $\xi=e_1$. 
Using the decomposition $x= se_1 + x'$ we see that (\ref{surf_cons_red_2})  
follows from (\ref{surf_cons_red_3}) using the functions 
$ \chi_{x'}(s):=\chi(se_1 + x')$, $\tilde\chi_{x'}$,
 $\zeta_{x'}$ in (\ref{surf_cons_red_3}) and integrating w.\ r.\ t.\ $dx'$. 
For the left-hand side, this is formally clear.
For the right-hand side, one uses $BV$-theory: If $\chi\in BV(\torus)$, 
we have $\chi_{x'}\in BV([0,\Lambda))$ for a.\ e. $x'\in [0,\Lambda)^{d-1}$ and 
\begin{align*}
 \int_{[0,\Lambda)^{d-1}} \int_0^\Lambda \zeta_{x'}(s)\, |\frac{d\chi_{x'}}{ds}|\, dx'
 = \int_\torus \zeta \left|e_1\cdot \nu\right| \left|\nabla \chi \right| 
\end{align*}
for any $\zeta\in C^\infty(\torus).$
To make the argument rigorous, we use again Lebesgue's dominated convergence. 
As before, using (\ref{surf_cons_XX1d}), we obtain
\begin{align*}
&\left| \frac{1}{\h} \int_0^\Lambda \zeta_{x'}(s)\tilde\chi_{x'}(s) \left(\chi_{x'}(s+\h) 
+ \chi_{x'}(s-\h)\right)\,ds\right|\\
&\leq  \|\zeta\|_\infty \frac{1}{\h} \int_0^\Lambda \left|\chi_{x'}(s+\h)- \chi_{x'}(s)\right| 
+ \left|\chi_{x'}(s-\h)-\chi_{x'}(s)\right|\,ds\\
&\leq  2 \|\zeta\|_\infty \int_0^\Lambda |\frac{d \chi_{x'}}{ds}|.
\end{align*}
Since
\begin{align*}
  \int_{[0,\Lambda)^{d-1}} \int_0^\Lambda |\frac{d\chi_{x'}}{ds}| dx' 
= \int_\torus \left|e_1\cdot \nu\right| \left|\nabla \chi \right| 
\leq \int_\torus |\nabla \chi|,
\end{align*}
this is indeed an integrable dominating function.
\step{Step 6: Argument for (\ref{surf_cons_red_3}).}
Since $\chi,\,\tilde\chi$ are $\{0,1\}$-valued, 
every jump has height $1$ and since $\chi,\, \tilde\chi\in BV([0,\Lambda))$,
the total number of jumps is finite.
Let $J,\,\tilde J\subset [0,\Lambda)$ denote the jump sets of $\chi$ and $\tilde \chi$, respectively. 
Now, if $\h$ is smaller than the minimal distance between two different points in $J\cup\tilde J$, 
then in view of (\ref{surf_cons_XX1d}), the only contribution
to the left-hand side of (\ref{surf_cons_red_3}) comes from neighborhoods of points where 
both, $\chi$ and $\tilde\chi$, jump:
\begin{align*}
 &\frac1\h \int_0^\Lambda \zeta(s)\,\tilde\chi(s) \left(\chi(s+\h) + \chi(s-\h)\right)\,ds
= \sum_{s\in J\cap \tilde J}\frac1\h\int_{s-\h}^{s+\h} \zeta(\sigma)\,\tilde\chi(\sigma) 
\left(\chi(\sigma+\h) + \chi(\sigma-\h)\right)\,d\sigma.
\end{align*}
Note that $\chi(\sigma+\h) + \chi(\sigma-\h) \equiv 1$ on each of these intervals and that
\begin{align*}
 \tilde \chi = \chara_{I_s^h}\quad \text{on } (s-\h,s+\h)
\end{align*}
for intervals of the form
\begin{align*}
 I_s^h = (s-\h,s) \quad \text{or}\quad  I_s^h = (s,s+\h).
\end{align*}
Since $|I_s^h|=\h$, we have
\begin{align*}
\frac1\h \int_0^\Lambda \zeta(s)\tilde\chi(s) \left(\chi(s+\h) + \chi(s-\h)\right)\,ds
= \sum_{s\in J\cap \tilde J} \frac1\h \int_{I_s^h} \zeta(\sigma)\,d\sigma 
\longrightarrow \sum_{s\in J\cap \tilde J} \zeta(s).
\end{align*}
Note that by (\ref{surf_cons_XX1d}), $\chi+\tilde\chi$ jumps precisely where 
either $\chi$ or $\tilde \chi$ jumps. Thus
\begin{align*}
 \int_0^\Lambda \zeta  \,\frac12\Big(|\frac{d \chi }{ds}| 
+ |\frac{d \tilde \chi}{ds}| 
-  |\frac{d(\chi+\tilde\chi)}{ds}|\Big)
=\frac12\bigg(\sum_{s\in J} \zeta(s) +\sum_{s\in  \tilde J} \zeta(s) 
- \sum_{s\in J\Delta \tilde J} \zeta(s)\bigg)
=\sum_{s\in J\cap \tilde J} \zeta(s).
\end{align*}
Therefore, (\ref{surf_cons_red_3}) holds, which concludes the proof.
\end{proof}

\begin{proof}[Proof of Lemma \ref{surf_la_err}]
The proof is divided into two steps. First, we prove the claim for $k=G$, 
to generalize this result for arbitrary kernels $k$ in the second step.
\step{Step 1: $k=G$.}
By Lemma \ref{surf_la_cons_G_and_k} and the convergence assumption (\ref{surf_conv_ass3}), 
we already know
\begin{align*}
 \lim_{h\to0}\frac1\h \int G_h(z) \left(f_h( z)-f( z)\right)\,dz = 0.
\end{align*}
Hence, it is sufficient to show that
\begin{align*}
 \lim_{h\to0}\frac1\h \int G_h(z) \left(f( z)-f_h( z)\right)_+\,dz = 0.
\end{align*}
Fix $h_0>0$ and $N\in \N$ and set $h:=\frac{1}{N^2}h_0$. 
We will make use of the following triangle inequality for $f_{(h)}=f,\, f_h$:
\begin{align}\label{surf_mon_triangle_inequality}
 f_{(h)}(z+w) \leq  f_{(h)}(z) +  f_{(h)}(w)
 \quad \text{for all } z,\,w\in \R^d.
\end{align}
This inequality has been proven in the proof of Lemma 3 in Section 7.1 of \cite{EseOtt14}. 
For the convenience of the reader we reproduce the argument here:
Using the admissibility of $\chi$ in the form of $\sum_k \chi_k =1$, we obtain the following
identity for any pair $1\leq i,j\leq \numphases $ of phases and any points $x,x',x''\in \torus$:
\begin{align*}
 &\chi_i(x)\chi_j(x'') - \chi_i(x)\chi_j(x') - \chi_i(x')\chi_j(x'')\\
= & \chi_i(x)\sum_k \chi_k(x')\chi_j(x'') - \chi_i(x)\chi_j(x')\sum_k \chi_k(x'') 
- \sum_k \chi_k(x)\chi_i(x')\chi_j(x'')\\
= & \sum_k \left[ \chi_i(x) \chi_k(x')\chi_j(x'') - \chi_i(x)\chi_j(x') \chi_k(x'')
- \chi_k(x)\chi_i(x')\chi_j(x'')\right].
\end{align*}
Note that the contribution of $k\in\{i,j\}$ to the sum has a sign:
\begin{align*}
 &\sum_{k\in\{i,j\}} \left[\chi_i(x) \chi_k(x')\chi_j(x'') 
 - \chi_i(x)\chi_j(x') \chi_k(x'') - \chi_k(x)\chi_i(x')\chi_j(x'')\right]\\
&\qquad= \chi_i(x) \chi_i(x')\chi_j(x'') - \chi_i(x)\chi_j(x') \chi_i(x'') - \chi_i(x)\chi_i(x')\chi_j(x'')\\
 &\qquad \quad+ \chi_i(x) \chi_j(x')\chi_j(x'') - \chi_i(x)\chi_j(x') \chi_j(x'') - \chi_j(x)\chi_i(x')\chi_j(x'')\\
&\qquad=  - \left[\chi_i(x)\chi_j(x') \chi_i(x'') + \chi_j(x)\chi_i(x')\chi_j(x'') \right] \leq 0.
\end{align*}
We now fix $z,\,w\in \R^d$ and use the above inequality for $x'=x+z$, $x''= x+z+w$
so that after multiplication with $\sigma_{ij}$, summation
over $1\leq i,j\leq \numphases$ and  integration over $x$, 
we obtain $f(z+w) -  f(z) - f(w)$ on the left-hand side.
Indeed, using the translation invariance for the term appearing in $f_\zeta(w)$, we have
\begin{align*}
 &f(z+w) -  f(z) - f(w) \\
 & = \int  \sum_{i\neq j}\sigma_{ij}
 \left[ \chi_i(x)\chi_j(x+z+w) - \chi_i(x)\chi_j(x+z) - \chi_i(x+z)\chi_j(x+z+w)\right] dx\\
\leq & \int  \sum_{i\neq j,k\neq i,j } \sigma_{ij}
\big[ \chi_i(x) \chi_k(x+z)\chi_j(x+z+w) - \chi_i(x)\chi_j(x+z) \chi_k(x+z+w) \\
& \qquad\qquad\qquad\qquad\quad
   - \chi_k(x)\chi_i(x+z)\chi_j(x+z+w)\big]dx.
\end{align*}
Using the triangle inequality for the surface tensions, 
we see that the first right-hand side integral is non-positive:
\begin{align*}
 &\sum_{i\neq j,k\neq i,j } \sigma_{ij}\left( \chi_i(x) \chi_k(x')\chi_j(x'') - \chi_i(x)\chi_j(x') \chi_k(x'') - \chi_k(x)\chi_i(x')\chi_j(x'')\right)\\
 \leq&
 \sum_{i\neq j,k\neq i,j } \sigma_{ik}  \chi_i(x) \chi_k(x')\chi_j(x'')
+\sum_{i\neq j,k\neq i,j } \sigma_{kj}  \chi_i(x) \chi_k(x')\chi_j(x'')\\
&-\sum_{i\neq j,k\neq i,j } \sigma_{ij}  \chi_i(x) \chi_j(x')\chi_k(x'')
-\sum_{i\neq j,k\neq i,j } \sigma_{ij}  \chi_k(x) \chi_i(x')\chi_j(x'')
=0.
\end{align*}
Indeed, the first and the third term, and the second and the last term cancel since the domain
of indices in the sums is symmetric and thus we have (\ref{surf_mon_triangle_inequality}).

By iterating the triangle inequality (\ref{surf_mon_triangle_inequality}) for $f_{(h)}=f,\, f_h$ we have
\begin{align*}
  f_{(h)}(Nz) \leq  N f_{(h)}(z)\quad \text{for all } z\in \R^d.
\end{align*} 
Hence, by the definition of $h$,
\begin{align}\label{surf_err_triangle_2}
 \frac1{\sqrt{h_0}} f_{(h)}(\sqrt{h_0}z) \leq  \frac1\h f_{(h)}(\h z)\quad \text{for all } 
z\in \R^d.
\end{align}
Therefore, using (\ref{surf_err_triangle_2}) for $f_h$, the subadditivity of $u \mapsto u_+$ and finally 
(\ref{surf_err_triangle_2}) for $f$, we obtain
\begin{align*}
 \Big(\frac1\h f(\h z)-\frac1\h f_h(\h z)\Big)_+ 
\leq& \Big(\frac1\h f(\h z)-\frac1{\sqrt{h_0}} f_h(\sqrt{h_0} z)\Big)_+ \\
\leq& \Big(\frac1\h f(\h z)-\frac1{\sqrt{h_0}} f(\sqrt{h_0} z)\Big)_+ + \Big(\frac1{\sqrt{h_0}} f(\sqrt{h_0} z)-\frac1{\sqrt{h_0}} f_h(\sqrt{h_0} z)\Big)_+\\
\leq & \frac1\h f(\h z)-\frac1{\sqrt{h_0}} f(\sqrt{h_0} z) + \frac1{\sqrt{h_0}}\Big| f(\sqrt{h_0} z) - f_h(\sqrt{h_0} z)\Big|.
\end{align*}
Integrating w.\ r.\ t.\ the positive measure $G(z)\,dz$ yields
\begin{align}\label{error_estimate_Fplus}
 \frac1\h \int G_h(z) \left(f( z)-f_h( z)\right)_+\,dz \notag
\leq &\frac1\h \int G(z) f(\h z)\,dz-\frac1{\sqrt{h_0}}\int G(z)  f(\sqrt{h_0} z)\,dz\\
&+ \frac1{\sqrt{h_0}} \int G(z) \Big| f(\sqrt{h_0} z)- f_h(\sqrt{h_0} z)\Big|\,dz\notag\\
= & E_h(\chi) - E_{h_0}(\chi) + \frac1{\sqrt{h_0}} \int G_{h_0}(z) \left| f( z)- f_h( z)\right|dz.
\end{align}
Given $\delta>0$, by Lemma \ref{surf_la_cons_G_and_k} we may first choose $h_0>0$ 
such that for all $0<h<h_0$:
\begin{align*}
\left|E_h(\chi) - E_{h_0}(\chi) \right| < \frac\delta2.
\end{align*}
We note that we may now choose $N\in \N$ so large that for all $0<h<\frac1{N^2}h_0$:
\begin{align*}
\left| f(\sqrt{h_0} z) - f_h(\sqrt{h_0} z)\right| \leq \frac\delta2 \sqrt{h_0}\quad \text{for all } z\in\R^d.
\end{align*}
Indeed, using the triangle inequality and translation invariance we have
\begin{align*}
 &\left| f(\sqrt{h_0} z) - f_h(\sqrt{h_0} z)\right|\\
 &\leq  \sum_{i,j} \sigma_{ij} \int 
 \left|\chi_i(x)\chi_j(x+z)-\chi_i(x)\chi^h_j(x+z)\right| 
 + \left|\chi_i(x)\chi^h_j(x+z)-\chi^h_i(x)\chi^h_j(x+z)\right| \,dx\\
&\lesssim  \sum_{i=1}^\numphases 
\int \left| \chi_i(x)-\chi_i^h(x) \right|\,dx,
\end{align*}
which tends to zero as $h\to 0$ because by Lebesgue's dominated convergence and (\ref{surf_conv_ass1}). 
Hence  also the second term on the right-hand side of (\ref{error_estimate_Fplus}) is small:
\begin{align*}
 \frac1{\sqrt{h_0}} \int G_{h_0}(z) \left|f( z)-f_h( z)\right|dz \leq \frac\delta2.
\end{align*}
\step{Step 2: $k=p\,G$.}
Fix $\epsilon>0$. Since $G$ is exponentially decaying, 
we can find a number $M= M(\epsilon)< \infty$ such that
\begin{align}\label{k<G}
 k(z) \leq \epsilon\,  G\big(\frac{z}{\sqrt{2}}\big)= \epsilon\,  G_2(z)\quad \text{for all } |z| > M.
\end{align}
Hence we can split the integral into two parts. On the one hand, using (\ref{surf_conv_err}) for $k=G$, 
\begin{align*}
 \frac1\h \int_{\{|z|\leq M\}}  k(z)\, |f_h(\h z)-f(\h z)|\,dz \leq& \left(\sup_{[0,M]} p\right) \frac1\h \int G(z) |f_h(\h z)-f(\h z)|\,dz \to 0, \text{ as } h\to 0,
\end{align*}
and on the other hand, using (\ref{k<G}) and the approximate monotonicity in Lemma \ref{lem_AM_EO},
\begin{align*}
\frac1\h \int_{\{|z|> M\}}  k(z)\, |f_h(\h z) - f(\h z)|\,dz
\leq& \epsilon \frac1\h \int G_2(z) \left(f_h(\h z) + f(\h z)\right) dz
\lesssim \epsilon \left( E_{h}(\chi^h) + E_{h}(\chi)\right).
\end{align*}
By the convergence assumption (\ref{surf_conv_ass3}) and the consistency, cf. Lemma \ref{lem_CONS_EO}, we can take the limit $h\to0$ on the right-hand side and obtain
\begin{align*}
 \limsup_{h\to 0} \frac1\h \int k_h(z) |f_h( z)-f( z)|\,dz 
\lesssim \epsilon  \sum_{i,j} \sigma_{ij}  \int\frac12 \left(|\nabla\chi_i| +|\nabla\chi_j| -   |\nabla(\chi_i+\chi_j)|\right).
\end{align*}
Since the left-hand side does not depend on $\epsilon>0$, this implies (\ref{surf_conv_err}).
\end{proof}

%
%
%
%
%
%

\section{Dissipation Functional and Velocity}\label{sec:velocity}
As for any minimizing movements scheme, the time derivative of the solution should arise from
the metric term in the minimization scheme.
For the minimizing movements scheme of our interfacial motion, the time derivative
is the normal velocity.
The goal of this section, which is the core of the paper, 
is to compare the first variation of the dissipation
functional to the normal velocity.

\subsection{Idea of the proof}
Let us first give an idea of the proof in a simplified setting with only two phases, a constant test vector field $\xi$ and no localization.
Then the first variation (\ref{variation_D}) of the metric term reads
\begin{align*}
 \frac2\h \int \left(\chi^n-\chi^{n-1}\right) G_h \ast \left(-\nabla \chi^n\cdot \xi\right) dx.
\end{align*}
Using the distributional equation $\nabla \chi \cdot \xi = \nabla \cdot (\chi\xi) - \left(\nabla \cdot \xi\right) \chi$, this is equal to
\begin{align*}
 \frac2\h \int \left(\chi^n-\chi^{n-1}\right) 
 \left(- \nabla G_h \ast \left( \chi^n \xi\right)
 + G_h\ast(\chi^n \xi)\right) dx
 \approx- 2 \int \frac{\chi^n-\chi^{n-1}}{h} 
 \,\xi \cdot \h \nabla G_h \ast \chi^n dx
\end{align*}
as $h\to 0$. We will prove this in Lemma \ref{lem_1st_var_of_-diss}.
Since $\partial^{-h}_t \chi^h = \frac{\chi^n-\chi^{n-1}}{h} \rightharpoonup V \left| \nabla \chi \right| dt$ and 
$\h \nabla G_h \ast \chi^n \approx c_0 \nu$ only in a weak sense, we cannot pass to the limit a priori.
Our strategy is to freeze the normal and to control
\begin{align}\label{weak times weak}
 \int_0^T \int \partial_t^{-h} \chi^h \,\xi \cdot \h \nabla G_h \ast \chi^h \,dx \,dt
 - \int_0^T \int \partial_t^{-h} \chi^h \,c_0\, \xi \cdot\nu^\ast\, dx \,dt
\end{align}
by the excess
\begin{align*}
 \gap := \int_0^T \left(E_h(\chi^h) - E_h(\chi^\ast)\right)dt,
\end{align*}
where $\chi^\ast = \chara_{\{x\cdot \nu^\ast >\lambda\}}$ is a half space in direction of $\nu^\ast$. By the convergence assumption $\gap$ converges to
\begin{align*}
 \EE :=  c_0 \int_0^T \left( \int\left| \nabla \chi \right| - \int\left| \nabla \chi^\ast \right| \right)dt,
\end{align*}
as $h\to 0$, which is small by De Giorgi's structure theorem -- at least after localization in space and time;
i.\ e.\ sets of finite perimeter have (approximate)
tangent planes almost everywhere.
To be self-consistent we will prove this application of De Giorgi's result in Section \ref{sec:conv}.

The main difficulty in controlling (\ref{weak times weak}) lies in finding good bounds on
\begin{align*}
  \int_0^T \int \left| \partial_t^{h} \chi^h \right| dx\,dt. 
\end{align*}
For the sake of simplicity we set $E_0 =T= \Lambda=1$ and write $\chi$ instead of $\chi^h$ in the following.
In Section \ref{sec:comp} we have seen the bound
\begin{align}\label{cheap bound}
 \iint\left| \partial_t^\tau \chi \right| dx\,dt = O(1) \quad \text{for } \tau \sim \h.
\end{align}
For this, we used the energy-dissipation estimate (\ref{energy_dissipation_estimate}) to bound the dissipation
\begin{align*}
\h \iint \left(G_{h/2}\ast \partial_t^h \chi \right)^2dx\,dt
= &\sum_{n} \frac1\h \int \left(\chi^n-\chi^{n-1}\right) G_h \ast \left(\chi^n-\chi^{n-1}\right) dx \lesssim 1
\end{align*}
and Jensen's inequality gave us control over the function
\begin{align}\label{overview alpha function}
 \alpha^2(t) := \frac1\h \int \left( G_{h/2}\ast \left(\chi(t+\tau) - \chi(t)\right) \right)^2 dx
 = \alpha^2 \h \int \left(G_{h/2}\ast \partial_t^\tau \chi \right)^2dx
\end{align}
by the fudge factor $\alpha$ appearing in the definition of the mesoscopic time scale $\tau=\alpha\h$:
\begin{align}\label{overview alpha and alpha}
 \int_0^T \alpha^2(t)\,dt \lesssim \alpha^2.
\end{align}
This estimate is the reason for the slight abuse of notation: We call the function in (\ref{overview alpha function}) $\alpha^2(t)$ in order to keep the relation (\ref{overview alpha and alpha})
between the two quantities in mind.
In the following we will always carry along the argument $t$ of the function $\alpha^2(t)$ to make the difference clear.
Writing $\chi^\tau$ short for $\chi(\, \cdot\,+\tau)$ we have shown in the proof of Lemma \ref{comp_lem_tau_shift} that (\ref{cheap bound})
holds in the more precise form of
\begin{align}\label{cheap bound precise}
 \iint \left| \chi^\tau -\chi\right| dx\,dt
 \lesssim \h \int \alpha^2(t)\,dt + \h \int E_h(\chi)\,dt
 \lesssim  \h \left( \gap+1\right) + \frac{\tau^2}\h.
\end{align}

In this section we will derive the following more subtle bound:
\begin{align}\label{subtle bound}
 \iint\left| \partial_t^\tau \chi \right| dx\,dt = O(1) \quad \text{for } \tau =o(\h). 
\end{align}
While the argument for (\ref{cheap bound}) was based on
\begin{align*}
 \chi^\tau - \chi  = G_h \ast(\chi^\tau -\chi) + (\chi-G_h\ast\chi) + (\chi^\tau-G_h\ast \chi^\tau)
\end{align*}
we now start from the thresholding scheme:
\begin{align*}
 \chi^\tau - \chi  =  \chara_{\{u^\tau >\frac12\}} - \chara_{\{u>\frac12\}}
\quad \text{with} \quad u^\tau := G_h\ast \chi^{\tau-h}\quad \text{and}\quad  u := G_h\ast \chi^{-h}.
\end{align*}
We will use an elementary one-dimensional estimate, Lemma \ref{la_periodic} (cf. Corollary \ref{cor_1_d_introducing_h} for this rescaled version),
in direction $\nu^\ast = e_1$ (w.\ l.\ o.\ g.) and integrate transversally to obtain
 \begin{align}\label{overview 1dla}
 \frac1\h \int\left| \chi^\tau -\chi \right| dx
\lesssim \frac1\h \int_{\frac13 \leq u\leq \frac23} \left(\h\partial_1 u -\overline c \right)_-^2 dx +  s + \frac1{s^2}\frac1\h \int (u^\tau-u)^2dx.
\end{align}
The first right-hand side term measures the monotonicity of the phase function $u$ in normal direction in the transition zone $\{\frac13 \leq u \leq \frac23\}$.
It is clear that this term vanishes for $\chi^{-h} = \chi^\ast$, provided the universal constant $\overline c>0$ is sufficiently small.
In Lemma \ref{lem_local_estimates} we will indeed bound this term by the excess
\begin{align*}
 \varepsilon^2(-h) := E_h(\chi^{-h}) - E_h(\chi^\ast)
\end{align*}
at the previous time step.
Compared to the first approach which yielded (\ref{cheap bound precise}),
where the limiting factor is that the first right-hand side term is only $O(\h)$,
the result of the latter approach yields the improvement
\begin{align}\label{subtle bound exact result}
  \iint\left| \chi^\tau -\chi \right| dx\,dt \lesssim \h \left(\gap+s\right)+  \frac1{s^2} \frac{\tau^2}\h
\end{align}
for an arbitrary (small) parameter $s>0$.
Now we show how to use the bound (\ref{subtle bound exact result}) in order to estimate (\ref{weak times weak}).
First, in Lemma \ref{lem_1st_var_of_-diss} by freezing time for $\chi$ on the mesoscopic time scale $\tau = \alpha \h$ and using a
telescoping sum for the first term $\partial_t^h \chi$ we will show that
\begin{align}\label{overview 1}
 \iint \partial_t^{h} \chi \,\xi \cdot \h \nabla G_h\ast \chi \,dx \, dt
 =\iint \partial_t^{\tau} \chi\, \xi \cdot \h \nabla G_h\ast \frac{\chi+\chi^\tau}2 \,dx\,dt +
 O\left( \left(\frac{\tau}\h \iint \left| \partial_t^{\tau} \chi \right| dx\,dt \right)^{\frac12}\right).
\end{align}
By (\ref{subtle bound exact result}) the error term is controlled by 
\begin{equation}\label{overview 1 error}
 \left(\gap +s +  \frac1{s^2} \alpha^2\right)^{\frac12} 
 \lesssim \frac1\alpha \gap + \alpha^{\frac13}
\end{equation}
by choosing $s\sim \alpha^{\frac23}$.
Second, in Lemma \ref{lem_1st_var_of_-diss_K_h} we will show how to use the algebraic relation $(\chi^\tau-\chi)(\chi^\tau + \chi) = \chi^\tau -\chi$ 
for the product $(\chi^\tau-\chi)\h \nabla G_h\ast (\chi^\tau + \chi) $  so that we can rewrite the right-hand side of (\ref{overview 1}) as
\begin{align}\label{overview 2}
 \iint \partial_t^{\tau} \chi \, c_0 \, \xi \cdot e_1 \,dx\,dt
 +O\left(\iint \left| \partial_t^\tau \chi \right| \, k_h \ast \left| \chi^\tau -\chi \right| dx \,dt\right)
 +O\left(\gap\right)
\end{align}
for some kernel $k$.
Third, in Lemma \ref{lem_1st_var_of_-diss_remainder} we will control the first error term by using its quadratic structure and the estimate (\ref{subtle bound exact result})
before the transversal integration in $x'$:
\begin{align}
 \int \left| \partial_t^\tau \chi \right|\, k_h \ast \left| \chi^\tau -\chi \right| dx 
 & \lesssim \frac1\tau \int \left( \int\left| \chi^\tau-\chi\right| dx_1 \right)\notag
 k_h' \ast'\left[ 1 \wedge \left( \frac1\h  \int\left| \chi^\tau - \chi\right| dx_1 \right)  \right] dx'\\
 &\lesssim \frac1\alpha \left[ \gap+\frac1{s^2} \alpha^2 \label{overview 3}
 + s\left( \h \left(\tilde s+\gap\right)+ \frac1{\tilde s^2} \alpha^2 \right)\right]\\
 & \lesssim \frac1\alpha \gap +  \frac1\alpha s\tilde s + \left( \frac s{\tilde s^2} + \frac1{s^2} \right) \alpha\notag
 \sim\frac1\alpha \gap + \alpha^{\frac19},
\end{align}
by choosing $\tilde s \sim \alpha^{\frac23}$ and $s \sim \alpha^{\frac49}$.
We note that the values of the exponents of $\alpha$ in (\ref{overview 1 error}) and (\ref{overview 3}) do not play any role and can be easily improved.
We only need the extra terms, here $\alpha^{\frac13}$ and $\alpha^{\frac19}$, to be $o(1)$ as $\alpha\to0$; the prefactor of the excess $\gap$, here $\frac1\alpha$, can be large.
Indeed, after sending $h\to 0$ we will obtain the error $\frac1\alpha \EE + \alpha^{\frac19}$. We will handle this term in Section \ref{sec:conv}
by first sending the fineness of the localization to zero so that $\EE$ vanishes, and then sending the parameter $\alpha\to0$.

In the following we will make the above steps rigorous and give a full proof in the multi-phase case.
First we state the main result, Proposition \ref{prop_velocity_good_balls}, then we explain the tools we will be using more carefully in the subsequent lemmas.
We turn first to the two-phase case to present the one-dimensional estimate (\ref{overview 1dla}) in Lemma \ref{la_periodic}, its rescaled and localized
version Corollary \ref{cor_periodic} and the estimate for the error term Lemma \ref{lem_local_estimates}.
Subsequently we state the same results in Lemma \ref{la_1d_multiphase} and Corollary \ref{lem_local_estimates_multiphase} for the multi-phase case.
These estimates are the core of the proof of Proposition \ref{prop_velocity_good_balls} and use the explicit structure of the scheme.
Let us note that in these estimates we are using the two steps of the scheme, the convolution step (\ref{convolve}) and the
thresholding step (\ref{threshold}), in a well-separated way. Indeed, the one-dimensional estimate,
Lemma \ref{la_1d_multiphase}, analyzes the thresholding step (\ref{threshold}); and Corollary \ref{lem_local_estimates_multiphase}
brings the (transversally integrated) error term in the form of the excess $\gap$ at the previous time step by analyzing the convolution step (\ref{convolve}).

\subsection{Results}\label{sub:velocity results}
The main result of this section is the following proposition which will be used
for small time intervals in Section \ref{sec:conv}
where we will control the limiting error terms which appear here
with soft arguments from Geometric Measure Theory.
In view of the definition of $\EE$ below, the proposition assumes
that $\chi_3,\dots,\chi_\numphases$ are the \emph{minority phases}
 in the space-time cylinder $(0,T)\times B_r$; likewise it assumes that the normal between $\chi_1$ and 
 $\chi_2$ is close to the first unit vector $e_1$.
 This can be assumed since on the one hand we can relabel the phases in case we want to
 treat another pair of phases as the majority phases.
 On the other hand, due to the rotational invariance, it is no restriction to assume that $e_1$ is
 the approximate normal.
 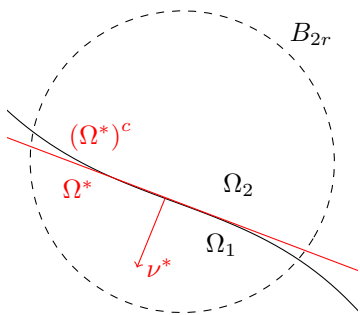
\begin{figure}[H]
\centering
  \begin{tikzpicture}[scale=1]
\clip (-2.3,-2.05) rectangle (2.3,2.05);

\begin{scope}[rotate=-30]

\draw[domain=-3:3,smooth,variable=\x] plot 
({\x},{.1*sin(50*\x)+.1*sin(50*(\x-.5))-.5-.007*\x*\x*\x});
\node at (1,-.7) {$\Omega_1$};
\node[red] at (-1,-1) {$\Omega^\ast$};
\node at (.8,.1) {$\Omega_2$};
\node[red] at (-1.1,-.25) {$\left(\Omega^{\ast}\right)^c$};
\draw[red] plot [domain=-3:3,variable=\x]  
  ({\x},{0.16*\x-.5-.03});
\draw[red,->] (275:.525)-- +(278:1) node [right] {$\nu^\ast$};
  
\draw[dashed] (0,0) circle (2); 
\node at (45+30:2.4) {$B_{2r}$}; 
\end{scope}

 \end{tikzpicture}
\caption{The majority phases $\Omega_1$ and $\Omega_2$ 
and the half space $\Omega^\ast = \{x \cdot \nu^\ast > \lambda\}$
approximating $\Omega_1$ inside the ball $B_{2r}$. 
Its complement $\left( \Omega^\ast \right)^c $ approximates $\Omega_2$ inside $B_{2r}$.}
\label{fig_tangent_space}
\end{figure}

%
%
\begin{prop}\label{prop_velocity_good_balls}
For any $\alpha \ll 1$, $T>0$, $\xi\in C_0^\infty((0,T)\times B_r,\R^d)$
and any $\eta\in C_0^\infty(B_{2r})$
radially symmetric and radially non-increasing cut-off 
for $B_r$ in $B_{2r}$ with $\left| \nabla \eta\right| \lesssim \frac1r$
 and $\left| \nabla^2 \eta\right| \lesssim \frac1{r^2}$, we have
\begin{align}\label{error_in_velocity_local_in_time}
&\limsup_{h\to0}\left|
 \int_0^T\! -\delta E_h(\,\cdot\,-\chi^h(t-h)) 
(\chi^h(t),\xi(t)) \,dt
+ 2c_0 \sigma_{12} \int_0^T\left(\int \xi_1\, V_1 \left|\nabla \chi_1 \right|
 -\int \xi_1\, V_2 \left|\nabla \chi_2 \right|\right)
dt
\right| \notag\\
&\qquad\qquad \qquad  \lesssim \|\xi\|_{\infty}\left[ 
  \int_0^T  
  \Big(\frac1{\alpha^2}\EE(t)  + \alpha^{\frac19} r^{d-1}\Big) dt + \alpha^{\frac19} \iint \eta \,d\mu \right].
\end{align}
Here we use the notation 
\begin{align*}
\EE(t) := 
 \sum_{i=3}^\numphases \int \eta \left| \nabla \chi_{i}(t)\right|
 +\inf_{\chi^\ast}
\Bigg\{&
\left|\int \eta \left(\left| \nabla \chi_{1}(t) \right|
 -\left| \nabla \chi^\ast\right| \right) \right|
  + \frac1r \int_{B_{2r}} \left|\chi_{1}(t)-\chi^\ast\right|dx\\
+&\left|\int \eta \left(\left| \nabla \chi_{2}(t) \right|
 -\left| \nabla \chi^\ast\right| \right)\right|
+ \frac1r \int_{B_{2r}} \left|\chi_{2}(t)-\left(1-\chi^\ast\right)\right|dx \Bigg\},
\end{align*}
where the infimum is taken over all half spaces
$\chi^\ast = \chara_{\{ x_1 > \lambda \}}$ 
in direction $e_1$.
\end{prop}
The exponents of $\alpha$ in this statement are of no importance and can be easily improved.
It is only relevant that the two extra error terms, i.\ e.\ $r^{d-1}T$ and $\iint \eta\, d \mu$,
are equipped with prefactors which vanish as $\alpha\to 0$.
In Section \ref{sec:conv} we will show that -- even after summation -- the excess will vanish as the fineness of the localization, i.\ e.\ the radius $r$
of the ball in the statement of Proposition \ref{prop_velocity_good_balls}
tends to zero.
There we will take first the limit $r\to 0$ and then $\alpha \to 0$ to prove Theorem \ref{thm1}.
The prefactor of the excess, here $\frac1{\alpha^2}$ differs from the one in the two-phase case since the one-dimensional estimate is slightly different in the multi-phase case.

Let us comment on the structure of $\EE$. 
The first term, describing the surface area of Phases $3,\dots,\numphases$ inside the ball $B_{2r}$, 
will be small in the application when $\chi_3,\dots,\chi_\numphases$ are indeed the minority phases.
The second term, sometimes called the \emph{excess energy} describes how far $\chi_1$ and $\chi_2$ are
away from being half spaces in direction $e_1$ or $-e_1$, respectively.
The terms comparing the surface energy inside $B_{2r}$ do not see the orientation of 
the normal, whereas the bulk terms measuring the $L^1$-distance inside the ball
$B_{2r}$ do see the orientation of the normal.
%

%
The estimates in Section \ref{sec:comp} are not sufficient to understand the link between the first
variation of the metric term and the normal velocities.
For this, we need refined estimates which we will first present for the two-phase case, where only one
interface evolves.
%
%
The main tool of the proof is the following one-dimensional lemma.
For two functions $u,\,\tilde u$, it estimates the $L^1$-distance between the characteristic functions
$\chi = \chara_{\{u\geq \frac12\}}$ and $\tilde \chi = \chara_{\{\tilde u\geq \frac12\}}$ in terms of the $L^2$-distance between the $u$'s - at the
expense of a term that measures the strict monotonicity of one of the functions $u$.
We will apply it in a rescaled version for $x_1$ being the normal direction.
\begin{lem}\label{la_periodic}
Let $I\subset\R$ be an interval,
Let $u,\, \tilde u\in C^{0,1}(I)$, 
$\chi:= \chara_{\{u \geq \frac12\}}$ and $\tilde \chi := \chara_{\{\tilde u \geq \frac12\}}$. Then
\begin{align}
\int_I \left|\chi-\tilde \chi\right|dx_1  \lesssim 
\int_{| u-\frac12|<s}\left( \partial_1 u -1 \right)_-^2 dx_1
+s + \frac{1}{s^2}\int_I \left(u-\tilde u \right)^2dx_1
\end{align}
for every $s>0$.
\end{lem}
%

%
%
%
The following modified version of Lemma \ref{la_periodic} is the estimate one would use in the two-phase case.
\begin{cor}\label{cor_periodic}
Let $u,\, \tilde u\in C^{0,1}(I)$, 
$\chi:= \chara_{\{u \geq \frac12\}}$, $\tilde \chi := \chara_{\{\tilde u \geq \frac12\}}$
and $\eta\in C_0^\infty(\R)$, $0\leq \eta\leq 1$ radially non-increasing.
Then
\begin{align*}
\frac{1}{\sqrt h}\int \eta \left|\chi-\tilde \chi\right| dx_1
\lesssim \frac{1}{\sqrt h}\int_{| u -\frac12|<s}
\eta \left( \sqrt{h}\,\partial_1 u -1 \right)_-^2dx_1 +s
+ \frac1{s^2}\frac{1}{\sqrt h}\int \eta \left(u-\tilde u \right)^2dx_1
\end{align*}
for any $s>0$.
\end{cor}
%
%
In the previous corollary, it was crucial to control strict monotonicity of one of the two functions
via the term 
\begin{align*}
 \frac{1}{\sqrt h}\int_{| u -\frac12|<s}
\eta \left( \sqrt{h}\,\partial_1 u -1 \right)_-^2dx_1.
\end{align*}
In the following lemma, we consider the $d$-dimensional version, i.\ e.\ $dx_1$ replaced by $dx$, of this term in case of $u=G_h\ast \chi$.
We show that this term can be controlled in terms of the excess, measuring the energy
difference to a half space $\chi^\ast$ in direction $e_1$.
\begin{lem}\label{lem_local_estimates}
Let $\chi\colon \torus \to \{0,1\}$, $\chi^\ast=\chara_{\{x_1>\lambda\}}$
a half space in direction $e_1$ and $\eta\in C_0^\infty(B_{2r})$ a cut-off 
of $B_r$ in $B_{2r}$ with $\left| \nabla \eta\right| \lesssim \frac1r$
 and $\left| \nabla^2 \eta\right| \lesssim \frac1{r^2}$.
Then there exists a universal constant $\overline{c}>0$ such that
\begin{align}
\frac1\h \int_{z_1 \lessgtr 0} G_h(z) \int\eta(x)
\left( \chi(x+z) -\chi(x)\right)_\pm dx \,dz 
\lesssim &  \gap  + \h \frac1{r^2},
 \label{eq_periodic_1}\\
 \frac1\h\int_{\frac13 \leq G_{h}\ast \chi \leq \frac23}
\eta \left( \h\,\partial_1(G_{h}\ast \chi) - \overline{c}\right)_-^2 dx 
\lesssim &\gap + \h \frac1{r^2}
 +\h \frac1{r} E_h(\chi),
\label{eq_periodic_2}
\end{align}
where $\gap$ is defined via
\begin{align*}
 \gap := \frac1\h \int  \eta \left[ \left(1-\chi\right) G_h \ast \chi +\chi\, G_h \ast \left(1-\chi\right)\right] dx
 -&\frac1\h \int \eta \left[\left(1-\chi^\ast\right) G_h\ast \chi^\ast+\chi^\ast\, G_h \ast \left(1-\chi^\ast\right)\right]dx\\
+&\frac1r \int_{B_{2r}}  \left|\chi-\chi^\ast\right|dx
\end{align*}
and the integral on the left-hand side of (\ref{eq_periodic_1}) with the two cases $<,+$ and $>,-$, respectively
is a short notation for the sum of the two integrals.
\end{lem}
In our application, we use the following lemma
which is valid for any number of phases with arbitrary surface tensions instead of Lemma \ref{la_periodic} or Corollary \ref{cor_periodic}.
Nevertheless, the core of the proof is already contained in the respective estimates in the two-phase case above.
As in Proposition \ref{prop_velocity_good_balls}, we assume that $\chi_1$ and $\chi_2$ are 
the majority phases and that $e_1$ is the approximate normal to $\Omega_1=\{\chi_1=1\}$.
%
%
\begin{lem}\label{la_1d_multiphase}
Let $I\subset \R$ be an interval, $h>0$, $\eta\in C_0^\infty(\R)$, $0\leq \eta\leq 1$ radially non-increasing and
$u,\,\tilde u \colon I \to \R^{\numphases}$ be two smooth maps into the standard simplex 
$\{U_i\geq 0,\,\sum_i U_i = 1\} \subset \R^\numphases$.
We define $\phi_i := \sum_j \sigma_{ij} u_j$, $\tilde \phi_i := \sum_j \sigma_{ij} \tilde u_j$, $\chi_i := \chara_{\{\phi_i < \phi_j\,\forall j\neq i\}}$ and 
$\tilde \chi_i := \chara_{\{\tilde \phi_i < \tilde \phi_j\,\forall j\neq i\}}$.
Then
\begin{align}
\frac1\h \int \eta \left|\chi-\tilde \chi \right| dx_1
   \lesssim &\frac1\h \int_{\frac13 \leq u_1 \leq \frac23 }\eta \left(\h \partial_1 u_1 - \overline c\right)_-^2 dx_1 
   +
 \frac1s \frac1\h \sum_{j\geq 3}\int \eta \left[ u_j \wedge(1-u_j) \right] dx_1\notag\\
  &+s   + \frac1{s^2} \frac1\h\int \eta \left| u-\tilde u \right|^2 dx_1\label{1dmulti_eq1}
\end{align}
for any $s\ll 1$.
\end{lem}

As Lemma \ref{lem_local_estimates} can be used to estimate the integrated version of the error in Corollary \ref{cor_periodic} against the excess,
the following corollary shows that the integrated version of the corresponding error term in the multi-phase version, Lemma \ref{la_1d_multiphase}, can be estimated
against a multi-phase version of the excess $\gap$.

\begin{cor}\label{lem_local_estimates_multiphase}
Let $\chi$ be admissible, $\chi^\ast=\chara_{\{x_1>\lambda\}}$
a half space in direction $e_1$ and $\eta\in C_0^\infty(B_{2r})$ a cut-off 
of $B_r$ in $B_{2r}$ with $\left| \nabla \eta\right| \lesssim \frac1r$
 and $\left| \nabla^2 \eta\right| \lesssim \frac1{r^2}$.
Then there exists a universal constant $\overline{c}>0$ such that for $u=G_h\ast \chi$
 \begin{align*}
  \frac1\h \int_{\frac13 \leq u_1 \leq \frac23 }\eta \left(\h\, \partial_1 u_1 - \overline c\right)_-^2 dx +
 \frac1\h \sum_{j\geq 3}\int \eta \left[ u_j \wedge(1-u_j) \right] dx
 \lesssim \gap(\chi) + \h \frac1{r^2} +\h \frac1{r} E_h(\chi),
 \end{align*}
where the functional $\gap(\chi)$ is defined via
\begin{align*}
 \gap(\chi) := \sum_{i\geq 3} F_h(\chi_j,\eta) +  &F_h(\chi_1,\eta) - F_h(\chi^\ast,\eta) + \frac1r \int_{B_{2r}} \left|\chi_1 - \chi^\ast \right| dx\\
 +& F_h(\chi_2,\eta) - F_h(\chi^\ast,\eta)+ \frac1r \int_{B_{2r}} \left|\chi_2 - (1-\chi^\ast) \right| dx
\end{align*}
and the functional $F_h$ is the following localized version of the approximate energy in the two-phase case
\begin{align*}
 F_h(\tilde \chi,\eta) := \frac1\h \int \eta 
 \left[ \left(1-\tilde \chi\right) G_h\ast \tilde \chi +  \tilde \chi\, G_h\ast \left(1-\tilde \chi\right)\right] dx, \quad \tilde \chi \in \{0,1\}.
\end{align*}
\end{cor}

%
%

With these tools  we can now turn to the rigorous proof of (\ref{overview 1})--(\ref{overview 3}) in the following lemmas.
In the next two lemmas, we approximate the first variation of the metric term by an expression
that makes the normal velocity appear.
The main idea is to work, as for Lemma \ref{comp_lem_tau_shift}, on a mesoscopic time scale $\tau\sim \h$,
introducing a fudge factor $\alpha$, cf.\ Remark \ref{timescales}.
The first lemma shows that we may coarsen the first variation from the microscopic time scale $h$
to the mesoscopic time scale $\alpha\h$ and is therefore the rigorous analogue of (\ref{overview 1}).
It also shows that we may pull the test vector field $\xi$ out of the convolution.
\begin{lem}\label{lem_1st_var_of_-diss}
 Let $\xi\in C_0^\infty((0,T)\times B_r,\R^d)$. Then
\begin{align*}
 \lefteqn{ \int_0^T -\delta E_h(\,\cdot\,-\chi^h(t-h)) (\chi^h(t),\xi(t)) \,dt}\\
&\approx
\sum_{i,j}\sigma_{ij}
\tau \sum_{l=1}^\lmax 
\int  \frac{\chi_i^{\kmax l} - \chi_i^{\kmax (l-1)}}\tau \,
\xi(l \tau)\cdot \left( \h\,\nabla G_h\right)
\ast\left(\chi_j^{\kmax(l-1)}+\chi_j^{\kmax l}\right) dx
\end{align*}
in the sense that the error is controlled by
\begin{align*}
 \|\xi\|_\infty \left(\frac1\alpha \tau\sum_{l=1}^\lmax \ggap(\chi^{\kmax l-1}) + \alpha^{\frac13} r^{d-1}T + \alpha^{\frac13} \iint \eta \,d\mu_h
 \right) + o(1),
\quad \text{as }h\to0,
\end{align*}
where  $\eta\in C_0^\infty(B_{2r})$ is
 a radially symmetric, radially
non-increasing cut-off for $B_r$ in $B_{2r}$  with $\left| \nabla \eta \right|\lesssim \frac1r$ 
and the functional $\ggap(\chi)$ is defined in Corollary \ref{lem_local_estimates_multiphase}.
\end{lem}
%

%
%
%
While the first lemma made the mesoscopic time derivative 
$\frac1\tau \big( \chi_i^{\kmax l} - \chi_i^{\kmax(l-1)}\big)$ appear, 
the upcoming second lemma makes the approximate normal, here $e_1$, appear. This is the analogue of (\ref{overview 2}).
\begin{lem}\label{lem_1st_var_of_-diss_K_h}
Given $\xi$ and $\eta$ as in Lemma \ref{lem_1st_var_of_-diss} we have
\begin{align*}
 & \sum_{i,j}\sigma_{ij}\,
\tau \sum_{l=1}^\lmax 
\int  \frac{\chi_i^{\kmax l}
- \chi_i^{\kmax (l-1)}}\tau\;
\xi(l \tau)\cdot \left( \h\,\nabla G_h\right)
\ast\left(\chi_j^{\kmax(l-1)}+\chi_j^{\kmax l}\right) dx\\
&\approx
-2c_0 \,\sigma_{12}\, \tau \sum_{l=1}^\lmax 
\left(  \int \xi_1(l \tau)\, \frac{\chi_1^{\kmax l}- \chi_1^{\kmax(l-1)}}{\tau} dx
-\int \xi_1(l \tau)\,\frac{\chi_2^{\kmax l}- \chi_2^{\kmax(l-1)}}{\tau} dx \right),
\end{align*}
in the sense that the error is controlled by
\begin{align*}
 \|\xi\|_{\infty} \Bigg[\frac1\alpha& \tau \sum_{l=1}^\lmax \ggap(\chi^{\kmax l})
 + \alpha \iint \eta \, d\mu_h
 + \tau \sum_{l=1}^\lmax 
 \frac1\tau \int \eta \,\big| \chi^{\kmax l} - \chi^{\kmax(l-1)}\big| \,k_h\ast
 \left( \eta \,\big| \chi^{\kmax l} - \chi^{\kmax(l-1)}\big|\right) dx
 \Bigg]
 +o(1),
\end{align*}
as $h\to0$, where $0\leq k(z) \leq |z|G(z)$
and the functional $\ggap(\chi)$ is defined in Corollary \ref{lem_local_estimates_multiphase}.
\end{lem}
Let us comment on the error term:
The first part of the error term arises because $e_1$ is only the approximate normal.
The last part arises in the passage from a diffuse to a sharp interface and formally is of
quadratic nature.
%

%
%
The following lemma deals with the error term in the foregoing lemma and brings it into
the standard form. The only difference to the two-phase case in (\ref{overview 3}) is the
prefactor in front of the excess $\gap$ which comes from the slight difference in the two one-dimensional estimates.
\begin{lem}\label{lem_1st_var_of_-diss_remainder}
With $\eta$ as in Lemma \ref{lem_1st_var_of_-diss} we have
\begin{align*}
&\tau \sum_{l=1}^\lmax \frac1\tau \int \eta \,\big| \chi^{\kmax l} 
- \chi^{\kmax(l-1)}\big| k_h\ast
 \left( \eta \,\big| \chi^{\kmax l} - \chi^{\kmax(l-1)}\big|\right) dx\\
&\lesssim \frac{1}{\alpha^2}  \tau\sum_{l=1}^\lmax \ggap(\chi^{\kmax l-1}) +\alpha^{\frac19} r^{d-1}T
    + \alpha^{\frac19} \iint \eta \,d\mu_h,
\end{align*}
where the functional $\ggap(\chi)$ is defined in Corollary \ref{lem_local_estimates_multiphase}.
\end{lem}

With the above lemma we can conclude the proof of Proposition \ref{prop_velocity_good_balls}.
Since one of the error terms includes the factor $r^{d-1}$ we will only use the proposition in case
there the behavior in the ball $B_r$ is non-trivial.
In the trivial case -- meaning that the measure of the boundary inside $B$ is much smaller than $r^{d-1}$ --
we can use the following easy estimate.
%
%

\begin{lem}\label{lem_rough_error_in_velocity}
In the situation as in Proposition \ref{prop_velocity_good_balls}, we have
\begin{align*}
&\left| 
\sum_{i,j} \sigma_{ij} \int_0^T \int \left(\nabla\cdot\xi - \nu_i \cdot \nabla \xi \,\nu_i
- 2\,\xi \cdot \nu_i \,V_i \right) 
\left(\left|\nabla\chi_i\right| + \left|\nabla\chi_j\right| - \left|\nabla(\chi_i+\chi_j)\right|\right)  dt
\right|\\
&\lesssim \|\xi\|_{\infty}
\left[  \sum_{i=1}^\numphases\int_0^T \int \eta \left(\frac1\alpha + \alpha\, V_i^2 \right)
\left| \nabla\chi_i\right|dt
+ \alpha \iint \eta\,d\mu
\right].
\end{align*}
\end{lem}

\subsection{Proofs}\label{sub:velocity proofs}
\begin{proof}[Proof of Proposition \ref{prop_velocity_good_balls}]
\step{Step 1: The discrete analogue of (\ref{error_in_velocity_local_in_time}).} 
The statement follows easily from
\begin{align}\label{error_in_velocity_local_in_time_reduced1}
 &\left|
 \int_0^T\! -\delta E_h(\,\cdot\,-\chi^h(t-h)) 
(\chi^h(t),\xi(t)) \,dt
+ 2c_0\sigma_{12} \int_0^T\left(\int \xi_1\, V_1 \left|\nabla \chi_1 \right|
 -\int \xi_1\, V_2 \left|\nabla \chi_2 \right|\right)
dt
\right|\notag\\
&\qquad\qquad \lesssim
 \|\xi\|_{\infty}\left[ 
 \frac1{\alpha^2} \int_0^T   \ggap(t)\, dt  + \alpha^{\frac19} r^{d-1}T
  + \alpha^{\frac19} \iint \eta \,d\mu_h \right] +o(1),\quad \text{as } h\to0.
\end{align}
Here we use the notation $\ggap(t) := \ggap(\chi^h(t))$, 
where the functional $\ggap(\chi)$ is defined in Corollary \ref{lem_local_estimates_multiphase}.
The infimum is taken over all half spaces $\chi^\ast=\chara_{\{x_1>\lambda\}}$ in direction $e_1$.
All terms appearing in $\ggap$ correspond to terms in $\EE$.
The first term is the sum of the localized approximate energies of $\chi_3,\dots,\chi_\numphases$,
the second term describes the approximate energy excess of Phases 1 and 2.
The convergence of these terms as $h\to 0$ for a fixed half space $\chi^\ast$ follows as in 
the proof of Lemma \ref{la_impl_conv_ass}.
Taking the infimum over the half spaces yields (\ref{error_in_velocity_local_in_time}).
\step{Step 2: Choice of appropriately shifted mesoscopic time slices.}
In order to prove (\ref{error_in_velocity_local_in_time_reduced1}),
we use the machinery that we develop later on in this section. 
There we work on the mesoscopic time scale $\tau=\alpha\h$ instead of the microscopic time scale $h$, see Remark \ref{timescales} for the notation.
To apply these results, we have to adjust the time shift of time slices of mesoscopic distance.
At the end, we will choose a microscopic time shift $ k_0\in\{1,\dots, \kmax\}$ such that
the average over time slices of mesoscopic distance is controlled by the average over all time slices:
\begin{align}\label{choose_time_coordinates_gap}
 \tau \sum_{l=1}^\lmax \left[\ggap(\chi^{\kmax l +  k_0})
 + \ggap(\chi^{\kmax l +  k_0-1})\right]
 \lesssim  h\sum_{n=1}^N\ggap(\chi^n)  = \int_0^T \ggap(t)\,dt.
\end{align}
This follows from the simple fact that $\ggap(k_0) \leq \kmean \ggap(k)$ for some $k_0$.
For notational simplicity, we shall assume that $k_0=0$ in (\ref{choose_time_coordinates_gap}).
\step{Step 3: Argument for (\ref{error_in_velocity_local_in_time_reduced1}).}
Using Lemmas \ref{lem_1st_var_of_-diss}, \ref{lem_1st_var_of_-diss_K_h} and 
\ref{lem_1st_var_of_-diss_remainder}, we obtain 
\begin{align}\label{error_in_velocity_local_in_time_reduced2}
 &\int_0^T\! -\delta E_h(\,\cdot\,,\chi^h(t-h)) 
(\chi^h(t),\xi(t)) \,dt\notag\\
&\approx -2c_0 \sigma_{12}\,\tau \!\sum_{l=1}^\lmax \!\left(  
 \int \! \xi_1(l\tau) \frac{\chi_1^{\kmax l}- \chi_1^{\kmax(l-1)}}{\tau} \,dx
-\int \!\xi_1(l\tau) \frac{\chi_2^{\kmax l}- \chi_2^{\kmax(l-1)}}{\tau} \,dx \right)
\end{align}
up to an error
\begin{align*}
\|\xi\|_\infty  \left(\frac1{\alpha^2} \int_0^T \ggap(t)\, dt
 + \alpha^{\frac19} r^{d-1}T + \alpha^{\frac19} \iint \eta \,d\mu_h\right)+o(1),\quad \text{as }h\to0,
\end{align*}
where we used the choice of time slices (\ref{choose_time_coordinates_gap}).
Since $\xi$ has compact support in $(0,T)$, a discrete integration by parts yields
\begin{align*}
 \tau \sum_{l=1}^\lmax \int \xi_1( l\tau)
 \frac1\tau\left(\chi_i^{\kmax l}-\chi_i^{\kmax (l-1)}\right)dx
= - \tau \sum_{l=0}^{\lmax-1} \int 
 \frac1\tau\left(\xi_1( (l+1)\tau)-\xi_1(l\tau)\right)
 \chi_i^{\kmax l}\, dx.
\end{align*}
By the H\"older-type bounds in Lemma \ref{comp_lem_hoelder} we can replace the mesoscopic scale on 
the right-hand side by the microscopic scale for $\chi$:
\begin{align*}
 & \left| \tau \sum_{l=0}^{\lmax-1} \int 
 \frac1\tau\left(\xi_1( (l+1)\tau)-\xi_1(l\tau)\right)
 \chi_i^{\kmax l}\, dx
 -\tau \sum_{l=0}^{\lmax-1} \frac1\kmax \sum_{k=1}^\kmax   \int 
 \frac1\tau\left(\xi_1( (l+1)\tau)-\xi_1(l\tau)\right)
 \chi_i^{\kmax l+k }\, dx \right|\\
& \leq \|\partial_t \xi\|_\infty h \sum_{l=0}^{\lmax-1} \sum_{k=1}^\kmax
 \int \left|\chi^{\kmax l} - \chi^{\kmax l + k}\right|dx
 \lesssim  \|\partial_t \xi\|_\infty E_0 T \sqrt{\tau}.
\end{align*}
By the smoothness of $\xi$, we can easily do the same for $\xi$ to obtain by (iii) 
in Proposition \ref{lem_dtX<<DX} that for $h\to 0$
\begin{align*}
 \tau \sum_{l=0}^{\lmax-1} \int 
 \frac1\tau\left(\xi_1( (l+1)\tau)-\xi_1(l\tau)\right)
 \chi_i^{\kmax l} \,dx \to
\int_0^T \int \partial_t \xi_1\, \chi_i \,dx \,dt
= - \int_0^T\int \xi_1 V_i\left|\nabla\chi_i \right| dt.
\end{align*}
Using this for the right-hand side of
(\ref{error_in_velocity_local_in_time_reduced2}) establishes 
(\ref{error_in_velocity_local_in_time_reduced1}) and thus concludes the proof.
\end{proof}

\begin{proof}[Proof of Lemma \ref{la_periodic}]
\step{Step 1: An easier inequality.}
We claim that for any function $u\in C^{0,1}(I)$, we have
\begin{align}\label{1dlemma_step1}
\left|\{ |u|\leq 1 \}\right| \lesssim \int_{\{|u|\leq 1\}}\! \left(\partial_1 u -1  \right)_-^2 dx_1 +1.
\end{align}
In order to prove (\ref{1dlemma_step1}), we decompose the set that we want to measure on the left-hand side
\begin{align*}
	\{|u|\leq1 \} = \bigcup_{J\in\J} J
\end{align*}
into countably many pairwise disjoint intervals.
As illustrated in Figure \ref{fig_1dlemma}, 
we distinguish the following four different 
cases for an interval $J=[a,b]\in \J$:
\begin{enumerate}[(i)]
	\item $J\in \J_{\nearrow}$: $u(a) = -1$ and $u(b)=1$
	\item $J\in \J_{\searrow}$: $u(a) = 1$ and $u(b)=-1$
	\item $J\in \J_{\to}$: $u(a) =u(b)$,
	\item $J\in \J_{\partial}$: $J$ contains a boundary point of $I$.
\end{enumerate}
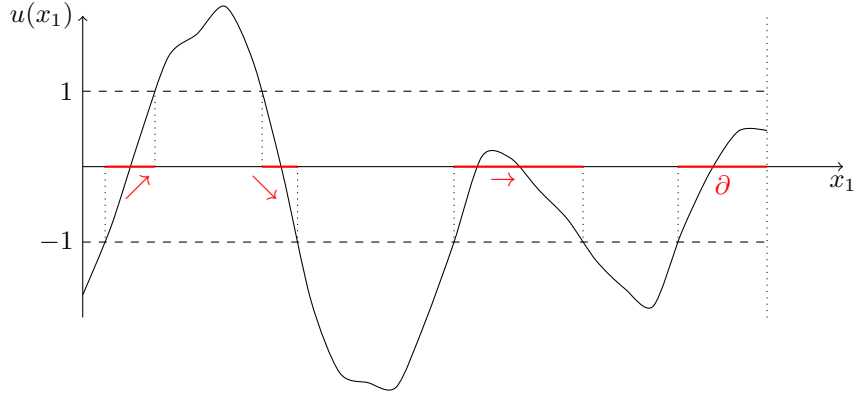
\begin{figure}
\centering
 \begin{tikzpicture}[scale=1]
\draw [->] (0,0) -- (9+1,0);
\node [below] at (9+1,0) {$x_1$};

\draw[dashed] (0,1)  -- (9,1);
\draw[dashed] (0,-1) -- (9,-1);

\node[left] at (0,1) {$1$};
\node[left] at (0,-1) {$-1$};

\draw plot [smooth,domain=0:9,variable=\x] 
  ({\x},{-1.5*cos(100*\x)
  +sin(64*\x)-.1*sin(deg(3*pi*\x))-.2*cos(deg(4*pi*\x))-.1*\x-.04*\x*\x
  +.0006*\x*\x*\x*\x});

  \draw [dotted]
    (.295,0) -- (.295,-1);
  \draw [dotted]
    (.95,0) -- (.95,1);
  \draw [dotted]
    (2.36,0) -- (2.36,1);
  \draw [dotted]
    (2.825,0) -- (2.825,-1);
  \draw [dotted]
    (4.883,0) -- (4.883,-1);
  \draw [dotted]
    (6.58,0) -- (6.58,-1);
  \draw [dotted]
    (7.83,0) -- (7.83,-1);
%
  \draw[thick,red] (.295,0) -- (.95,0);
  \draw[thick,red] (2.36,0) -- (2.825,0);
  \draw[thick,red] (4.883,0) -- (6.58,0);
  \draw[thick,red] (7.83,0) -- (9,0);
  \node [below,red] at (.5*.295+.5*.95+.1,0) {$\nearrow$};
  \node [below,red] at (.5*2.36+.5*2.825-.2,0) {$\searrow$};
  \node [below,red] at (.5*4.883+.5*6.58-.2,0) {$\rightarrow$};
  \node [below,red] at (.5*7.83+.5*9,0) {$\partial$};
  
  \draw [->] (0,-2) -- (0,2);
  \node [left] at (0,2) {$u(x_1)$};
  \draw [dotted] (9,-2) -- (9,2);

\end{tikzpicture}
\caption{The four cases (i)--(iv) for an interval $J\subset I$ 
(from left to right).}
\label{fig_1dlemma}
\end{figure}

By Jensen's inequality for the convex function $z\mapsto z_-^2$, we have
\begin{align*}
\frac{1}{|J|}\int_J \left( \partial_1 u -1\right)_-^2 dx_1
 \geq & \bigg( \frac{1}{|J|}\int_J \left( \partial_1 u -1\right) dx_1\bigg)_-^2
 =  \left( 1- \frac{u(b)-u(a)}{|J|}\right)_+^2\\
 = & \begin{cases}
 \left( 1-\frac2{|J|}\right)_+^2, &\text{if } J \in\J_\nearrow,\\
 \left( 1+\frac2{|J|}\right)^2, & \text{if } J\in\J_\searrow,\\
 1, & \text{if } J\in\J_\to.
 \end{cases} 
\end{align*}
If $|J|\geq 4$, then $-1\leq 2(u(b)-u(a))/|J| \leq1$ and so
\begin{align}\label{1dlem_foranyJ}
\frac{1}{|J|}\int_J \left( \partial_1 u -1\right)_-^2 dx_1  \geq \frac14.
\end{align}
Thus, we have 
\begin{align*}
 |J|\lesssim 1 \vee \int_J \left( \partial_1 u -1\right)_-^2 dx_1
\end{align*}
for any interval $J\in \J$.
Since $\# \J_\partial \leq 2$, we have
\begin{align*}
 \sum_{J\in \J_\partial} |J| \lesssim 1 + \int_J \left( \partial_1 u -1\right)_-^2 dx_1,
\end{align*}
which is enough in case (iv).
In case (iii), we immediately have
\begin{align}\label{1dlem_case_stay}
 |J| \lesssim \int_J \left(\partial_1 u-1 \right)_-^2dx_1,
\end{align}
while in case (ii) we even have the stronger estimate
\begin{align*}
	\int_J \left(\partial_1 u-1 \right)_-^2dx_1 
\gtrsim |J| \left( 1+\frac2{|J|}\right)^2 \gtrsim 1 \vee |J|
\end{align*}
since $1+s^2\geq 1$ and $1+s^2 \geq 2 s $ for all $s\in \R$. 
Thus on the one hand we can estimate the measure of such an interval $J\in \J_{\searrow}$ as in (\ref{1dlem_case_stay}).
On the other hand, we can bound the total number of these intervals:
\begin{align}\label{numberJdown}
\#\J_{\searrow} \lesssim \sum_{J\in \J_{\searrow}} \int_J \left(\partial_1 u-1 \right)_-^2dx_1 \leq \int_{|u|\leq1}\left(\partial_1 u-1\right)_-^2dx_1,
\end{align}
which clearly yields
\begin{align*}
	\# \J_{\nearrow} \leq 	\#\J_{\searrow} +1 
	\lesssim  \int_{|u|\leq1}\left(\partial_1 u-1\right)_-^2dx_1+1.
\end{align*}
Hence, using (\ref{1dlem_foranyJ}) for those $J\in \J_{\nearrow}$ with $|J| \geq 4$, we have
\begin{align*}
	\sum_{J\in \J_\nearrow} |J| = & \sum_{\substack{J\in \J_\nearrow \\ |J|\geq 4}} |J| + \sum_{\substack{J\in \J_\nearrow \\ |J|< 4}} |J| 
	\lesssim  \int_{|u|\leq1}\left(\partial_1 u-1\right)_-^2dx_1 + \# \J_\nearrow 
	\lesssim  \int_{|u|\leq1}\left(\partial_1 u-1\right)_-^2dx_1 + 1.
\end{align*}
Using these estimates, we derive
\begin{align*}
	|\{|u|\leq 1\}| = \sum_{J\in \J} |J| 
	\lesssim \int_{|u|\leq1}\left(\partial_1 u-1\right)_-^2dx_1 + 1.
\end{align*}
\step{Step 2: Rescaling (\ref{1dlemma_step1}).} 
Let $s>0$. We use Step 1 for $\hat{u}$ and set $u:=s\hat u$, $x_1=s\hat x_1$. 
Then $\partial_1 u = \hat \partial_1 \hat u$ and 
\begin{align*}
	|\{|u|\leq s\}| = s\left|\{ \hat u \leq 1 \}\right| 
	\overset{(\ref{1dlemma_step1})}{\lesssim} s \int_{|\hat u|\leq1}\left(\hat \partial_1 \hat u-1\right)_-^2d
	\hat x_1 + s
	=  \int_{|u|\leq s}\left(\partial_1 u-1\right)_-^2dx_1 + s.
\end{align*}
Therefore, using this for $u - \frac12$ instead of $u$, we have
\begin{align}\label{1dlemma_step1_s}
  |\{|u-\tfrac12|\leq s\}| \lesssim \int_{|u-\frac12|\leq s}\left(\partial_1 u-1\right)_-^2dx_1 + s.
\end{align}
\step{Step 3: Introducing $\tilde u$.} By Chebyshev's inequality, we have
\begin{align*}
	\left|\{|u-\tilde u|\geq s\} \right| \leq \frac{1}{s^2} \int_I (u-\tilde u)^2dx_1
\end{align*}
for all $s>0$.
Set 
\begin{align*}
	E:= \{|u-\tfrac12|\leq s\} \cup \{|u-\tilde u|\geq s\}\subset I. 
\end{align*}
Then, since e.\ g. $u\geq\frac12>\tilde u$ and $|u-\frac12|>s$ imply $|\tilde u-u|>s$,
\begin{align*}
	\{\chi \neq \tilde \chi \} = \{u\geq \tfrac12\} \Delta \{\tilde u\geq \tfrac12\} \subset E.
\end{align*}
Hence,
\begin{align*}
	\int_I \left|\chi-\tilde \chi\right| dx_1 \leq |E| 
	\lesssim \int_{| u -\frac12|<s}\left( \partial_1 u -1 \right)_-^2dx_1 
	+s + \frac{1}{s^2}\int_I \left(u-\tilde u \right)^2 dx_1,
\end{align*}
which concludes the proof.
\end{proof}

\begin{proof}[Proof of Corollary \ref{cor_periodic}]
By rescaling $x_1= \sqrt{h}\,\hat{x_1}$,
 $\hat u (\hat x_1)=u(\sqrt{h}\,\hat x_1)$, and analogously for $\tilde u$ and using Lemma \ref{la_periodic}
for the transformed functions we obtain:
\begin{align}\label{cor_1_d_introducing_h}
 \frac{1}{\sqrt h}\int_I \left|\chi-\tilde \chi\right|dx_1 
\lesssim \frac{1}{\sqrt h}\int_{| u -\frac12|<s}
\left( \sqrt{h}\,\partial_1 u -1 \right)_-^2dx_1+s + 
\frac{1}{s^2}\frac{1}{\sqrt h}\int_I \left(u-\tilde u \right)^2dx_1.
\end{align}
Now we approximate $\eta$ by simple functions: 
Let
\begin{align*}
 \tilde\eta := \frac{[N\eta]}{N} = \frac1N \sum_{n=1}^N  \chara_{J_n},
\quad \text{where}\quad 
J_n := \left\{ x\in I \colon \eta(x) > \frac nN \right\}.
\end{align*}
Then $0 \leq \tilde\eta \leq \eta$,
 $\left| \eta -\tilde \eta \right| \leq \frac1N$ and since $\eta$ is
radially non-increasing, each $J_n$ is an open interval.
We can apply 
(\ref{cor_1_d_introducing_h}) with $J_n$ playing the role of $I$.
By linearity we have
\begin{align*}
 \frac{1}{\sqrt h}\int \tilde\eta \left|\chi-\tilde \chi\right| dx_1
&\lesssim \frac{1}{\sqrt h}\int_{| u -\frac12|<s}
\tilde\eta \left( \sqrt{h}\,\partial_1 u -1 \right)_-^2dx_1+s
 + \frac{1}{s^2}\frac{1}{\sqrt h}\int \tilde\eta  \left(u-\tilde u \right)^2dx_1\\
&\leq \frac{1}{\sqrt h}\int_{| u -\frac12|<s}
\eta \left( \sqrt{h}\,\partial_1 u -1 \right)_-^2dx_1+s + 
\frac{1}{s^2}\frac{1}{\sqrt h}\int \eta  \left(u-\tilde u \right)^2dx_1.
\end{align*}
Passing to the limit $N\to \infty$, the left-hand side converges to
$\frac{1}{\sqrt h}\int \eta \left|\chi-\tilde \chi\right| dx_1$
and we obtain the claim.
\end{proof}

\begin{proof}[Proof of Lemma \ref{lem_local_estimates}]
\step{Argument for (\ref{eq_periodic_1}).}
As in Step 1 of the proof of Lemma \ref{comp_lem_z_shift}, by (\ref{comp_x-Ghx}) we have
\begin{align}\label{rewritelocalenergy}
  \frac1\h\int\eta \left[(1-\tilde\chi)\, G_h\ast\tilde\chi 
+ \tilde\chi\,G_h\ast(1-\tilde\chi)\right] dx=  \frac1\h\int G_h(z)  \int   \eta(x)\left| \tilde\chi(x+z) - \tilde\chi(x) \right| dx\,dz.
\end{align}
Using $\left| \chi^\ast(x+z) - \chi^\ast(x) \right|
= \sign(z_1)\left( \chi^\ast(x+z) - \chi^\ast(x) \right)$,
and $2u_+ = |u|+u$ on the set $\{z_1>0\}$ and $2u_- = |u|-u$ on $\{z_1<0\}$, we thus obtain
\begin{align*}
 \lefteqn{\frac2\h \int_{z_1 \lessgtr 0} G_h(z) 
\int\eta(x)\left( \chi(x+z) -\chi(x)\right)_\pm dx \,dz}\\
=& \frac1{\h}\int G_h(z)  \int  \eta(x) 
\left(\left| \chi(x+z) - \chi(x) \right| - 
\left| \chi^\ast(x+z) - \chi^\ast(x) \right| \right) dx\,dz \\
&- \frac1{\h}\int\sign(z_1) \, G_h(z) \int  \eta(x) 
\left((\chi^\ast-\chi)(x+z) - (\chi^\ast-\chi)(x) \right) dx\,dz\\
\leq & \gap
 -  \frac1{\h}\int  \sign(z_1)\,G_h(z) 
\int  \left(\eta(x)-\eta(x-z)\right) (\chi^\ast-\chi)(x)\,dx\,dz,
\end{align*}
where we used again $<,+$ and $>,-$, respectively
as a short notation for the sum of the two integrals.
Now we can apply a Taylor expansion for $\eta$ around $x$, 
i.\ e.\ write $\eta(x) - \eta(x-z) =  \nabla\eta(x)\cdot z + O(|z|^2)$, 
where the constant in the $O(|z|^2)$-term depends linearly on 
$\left\|\nabla^2\eta\right\|_{\infty}$. 
By symmetry, the first-order term is
\begin{align*}
\frac1{\h}\int  \sign(z_1)\, z\,G_h(z) \,dz\cdot  
\int  \nabla\eta(x) (\chi^\ast-\chi)(x)\,dx 
=  \int  \frac{|z_1|}\h G_h(z) \,dz \int  \partial_1\eta(x) 
(\chi^\ast-\chi)(x)\,dx.
\end{align*}
Note that the right-hand side can be controlled by
\begin{align*}
 \| \partial_1 \eta \|_\infty \int_{B_{2r}} \left| \chi - \chi^\ast\right| dx
 \lesssim \frac1r\int_{B_{2r}} \left| \chi - \chi^\ast\right| dx\leq \gap.
\end{align*}
The second-order term is controlled by
\begin{align*}
\left\|\nabla^2\eta\right\|_{\infty}\frac{1}{\h}\int |z|^2 G_h(z) \,dz 
= \left\|\nabla^2\eta\right\|_{\infty} \h \int |z|^2 G(z)\,dz 
\lesssim  \h\frac1{r^2},
\end{align*}
which completes the  proof of (\ref{eq_periodic_1}).
\step{Argument for (\ref{eq_periodic_2}).}
For the first arguments let w.\ l.\ o.\ g.\  $h=1$. 
The first ingredient is the identity
\begin{align}\label{eq_periodic_2_1}
	\partial_1 (G\ast \chi)(x) = \int |z_1| G(z) \left|\chi(x+z)-\chi(x)\right|dz 
	- 2 \!\int_{z_1\lessgtr 0} |z_1| G(z) \left( \chi(x+z)-\chi(x) \right)_\pm dz,
\end{align}
where the last term is the sum of the two integrals. 
Indeed, since $\partial_1 G(z) = -z_1 G(z)$ is odd in $z_1$,
\begin{align*}
	\partial_1 (G\ast \chi)(x) =  \int\partial_1 G(z)\chi(x-z)\,dz
 = \int z_1 G(z) \left( \chi(x+z) - \chi(x) \right) dz
\end{align*}
and splitting the integrand in the form $u=|u|-2u_-$ on the set $\{z_1>0\}$ 
and $-u=|u|-2u_+$ on $\{z_1<0\}$, respectively, we derive
\begin{align*}
\partial_1 (G\ast \chi)(x) = &\int_{z_1>0} 
|z_1| G(z) \left|\chi(x+z)-\chi(x)\right|dz + \int_{z_1<0} 
|z_1| G(z) \left|\chi(x+z)-\chi(x)\right|dz \\
 -& 2 \!\int_{z_1> 0} |z_1| G(z) \left( \chi(x+z)-\chi(x) \right)_-dz  
 - 2 \!\int_{z_1< 0} |z_1| G(z) \left( \chi(x+z)-\chi(x) \right)_+dz,
\end{align*}
which is (\ref{eq_periodic_2_1}).\\
The second ingredient for (\ref{eq_periodic_2}) is
\begin{equation}\label{eq_periodic_2_3}
 \int |z_1| G(z) |\chi(x+z)-\chi(x)|dz \gtrsim \left( \int G(z) |\chi(x+z)-\chi(x)|dz \right)^2.
\end{equation}
To obtain (\ref{eq_periodic_2_3}), we estimate
\begin{align*}
 \int |z_1| G(z) |\chi(x+z)-\chi(x)|dz 
\geq & \int_{|z_1|\geq \epsilon} |z_1| G(z) |\chi(x+z)-\chi(x)|dz \\
\geq & \epsilon \int_{|z_1|\geq \epsilon} G(z) |\chi(x+z)-\chi(x)|dz \\
 =   & \epsilon \int G(z) |\chi(x+z)-\chi(x)|dz \\
 & -\epsilon \int_{|z_1|< \epsilon} G(z) |\chi(x+z)-\chi(x)|dz.
\end{align*}
We recall that we $G$ factorizes in a one-dimensional Gaussian $G^1$ and a $(d-1)$-dimensional Gaussian $G^{d-1}$, i.\ e.\ $G(z) = G^1(z_1)\, G^{d-1}(z')$ so that the second integral can be estimated from above by $2 G^1(0)\epsilon$. Therefore we have
\begin{equation*}
 \int |z_1| G(z) |\chi(x+z)-\chi(x)|dz \geq \epsilon \int G(z) |\chi(x+z)-\chi(x)|dz 
-2G^1(0)\epsilon^2.
\end{equation*}
Optimizing in $\epsilon$ yields (\ref{eq_periodic_2_3}).\\
Using the fact that $\chi\in\{0,1\}$,
\begin{align*}
\int G(z) |\chi(x+z)-\chi(x)|dz = (1-\chi)(x) (G\ast \chi)(x) + \chi(x)(G\ast(1-\chi))(x)
\end{align*}
implies the third ingredient:
\begin{align}\label{eq_periodic_2_2}
\int G(z) \left| \chi(x+z)-\chi(x) \right| dz \geq
 (G\ast \chi)(x)  \wedge (1-G\ast \chi)(x).
\end{align}
Combining (\ref{eq_periodic_2_1}), (\ref{eq_periodic_2_3}) and (\ref{eq_periodic_2_2}),
 one finds a positive constant $\overline{c}$ such that
\begin{equation*}
 \partial_1 (G\ast\chi)(x) \geq 18 \,\overline{c} \left[  (G\ast\chi)(x)\wedge 
(1-G\ast\chi)(x) \right]^2 
- 2\! \int_{z_1\lessgtr 0} |z_1| G(z) \left( \chi(x+z)-\chi(x) \right)_\pm dz,
\end{equation*}
where we recall that the last term is the sum of the two integrals.
We consider the ``bad'' set
\begin{equation*}
 E := \Bigg\{  x \colon \int_{z_1\lessgtr 0} |z_1| G(z) \left( \chi(x+z)-\chi(x) \right)_\pm dz \geq \frac{\overline{c}}{2} \Bigg\}.
\end{equation*}
By construction of $E$ we have a good estimate on $E^c$:
\begin{equation*}
 \partial_1 (G\ast\chi)(x) \geq 18\,\overline{c} \left[ \min\left\{(G\ast\chi)(x) 
 ,\, (1-G\ast\chi)(x)\right\} \right]^2 - \overline{c} \quad\text{on } E^c,
\end{equation*}
and thus we obtain strict monotonicity of $G\ast \chi$ in $e_1$-direction outside $E$ as long as
the first term on the left-hand side dominates the second term:
\begin{equation*}
 \partial_1 (G\ast\chi) \geq \overline{c} \quad \text{on } E^c\cap \left\{\frac13 \leq G\ast\chi\leq 
 \frac23\right\}.
\end{equation*}
Therefore
\begin{align*}
 \int_{\frac13 \leq G\ast \chi \leq \frac23} \eta \left( \partial_1 G\ast \chi - \overline{c}\right)_-^2 dx
 =  \int_{E \cap \left\{\frac13 \leq G\ast \chi \leq \frac23\right\}} \eta \left( \partial_1 G\ast \chi - \overline{c}\right)_-^2 dx \lesssim \int_E \eta \,dx.
\end{align*}
We introduce the parameter $h$ again. Then this turns into 
\begin{align*}
 \frac1\h 
 \int_{\frac13 \leq G_h\ast \chi \leq \frac23} \eta \left(\h \partial_1 G_h\ast \chi - \overline{c}\right)_-^2 dx \lesssim \frac1\h  \int_{E_h} \eta \,dx,
\end{align*}
with now
\begin{align*}
 E_h := \Bigg\{  x \colon \frac1\h\int_{z_1\lessgtr 0} \frac{|z_1|}{\h} G_h(z) \left( \chi(x+z)-\chi(x) \right)_\pm dz \geq \frac{\overline{c}}{2} \Bigg\}.
\end{align*}
By construction of $E$ and since $|z|G_h(z) \lesssim \h\, G_h(\frac z2)$, we have
\begin{align}\label{argumentforlocestimatesusedfor2ndlemma}
\frac1\h\int_{E_h} \eta \,dx \lesssim& 
 \frac1 h\int_{z_1\lessgtr 0} |z_1| G_h(z) \int \eta(x)\left( \chi(x+z)-\chi(x) \right)_\pm dx\, dz\\
\notag\lesssim &\frac1\h\int_{z_1\lessgtr 0} G_h(z/2) \int \eta(x)\left( \chi(x+z)-\chi(x) \right)_\pm dx\, dz\\
\notag\lesssim &\frac1\h\int_{z_1\lessgtr 0} G_h(z) \int \eta(x)\left( \chi(x+z)-\chi(x) \right)_\pm dx\, dz\\
\notag+ & \frac1\h\int_{z_1\lessgtr 0} G_h(z) \int \eta(x)\left( \chi(x+2z)-\chi(x+z) \right)_\pm dx\, dz
\end{align}
by a change of coordinates $z\mapsto 2z$ and the subadditivity of the functions $u\mapsto u_\pm$.
The last term can be handled using a Taylor expansion of $\eta$ around $x$:
\begin{align*}
\frac1\h \lefteqn{ \int_{z_1\lessgtr 0} G_h(z) \int \eta(x)\left( \chi(x+2z)-\chi(x+z) \right)_\pm dx\, dz}\\
=& \frac1\h\int_{z_1\lessgtr 0} G_h(z) \int \eta(x-z)\left( \chi(x+z)-\chi(x) \right)_\pm dx\, dz\\
=&\frac1\h\int_{z_1\lessgtr 0} G_h(z) \int \eta(x)\left( \chi(x+z)-\chi(x) 
\right)_\pm dx\, dz + O\big(\h\,\big),
\end{align*}
where the constant in the $O(\h)$-term depends linearly on $E_h(\chi)$ and
$\left\|\nabla\eta\right\|_{\infty}$.
Indeed, the error in the equation above is 
- up to a constant times $\left\|\nabla\eta\right\|_{\infty}$ 
- estimated by
\begin{align*}
 \int  \frac{|z|}\h G_h(z) \int\left|\chi(x+z)-\chi(x) \right| dx\, dz
\lesssim  \int G_h(\frac z 2) \int\left|\chi(x+z)-\chi(x) \right| dx\, dz
\overset{(\ref{approximate monotonicity})}{\lesssim} \h  E_h(\chi).
\end{align*}
Using (\ref{eq_periodic_1}), we obtain 
\begin{align*}
 \frac1\h\int_{E_h} \eta \,dx 
\lesssim \gap  + \h \frac1{r^2}
 + \h\frac1r  E_h(\chi)
\end{align*}
and thus (\ref{eq_periodic_2}) holds.
\end{proof}

\begin{proof}[Proof of Lemma \ref{la_1d_multiphase}]
By the same argument as in Corollary \ref{cor_periodic}, we can ignore the cut-off $\eta$ and the parameter $h>0$ and reduce the claim to the following version:
\begin{align}\label{la_1d_multiphase_pf0}
\int_I \left|\chi-\tilde \chi \right| dx_1
   \lesssim \int_{| u_1 - \frac12| \lesssim s } \left(\partial_1 u_1 - \overline c\right)_-^2 dx_1 + \frac1s \int_I\sum_{j\geq 3}\left[ u_j \wedge(1-u_j) \right] \,dx_1 
  +s   + \frac1{s^2} \int_I \left| u-\tilde u \right|^2 dx_1.
\end{align}
We will prove
\begin{align}\label{la_1d_multiphase_pf1}
 \{\chi \neq \tilde \chi  \} \subset \Big \{|u_1-\tfrac12| \lesssim s \Big \} \cup \Big \{\sum_{j\geq  3}\left[ u_j \wedge(1-u_j) \right] \gtrsim s \Big \} \cup  \{|u-\tilde u| \gtrsim s  \}.
 \end{align}
Then (\ref{la_1d_multiphase_pf0}) follows from the one-dimensional case in the form of (\ref{1dlemma_step1_s}) for the first right-hand side set and 
Chebyshev's inequality for the second and third right-hand side set. The fact that we replaced the $1$ in (\ref{1dlemma_step1_s}) by the universal constant $\overline c>0$ can be justified by a simple rescaling.

In order to prove (\ref{la_1d_multiphase_pf1}), we fix $i\in \{1,\dots,\numphases\}$ and define the functions
\begin{align*}
 v:= \min_{j\neq i} \phi_j - \phi_i \in C^{0,1}(I)
\end{align*}
and $\tilde v$ in the same way, so that $\chi_i = \chara_{\{v>0\}}$, $\tilde\chi_i = \chara_{\{\tilde v>0\}}$ and
\begin{align}\label{chiandv}
 \{\chi_i \neq \tilde \chi_i  \} \subset 
  \{|v| < s  \} \cup  \{|v-\tilde v| \geq s  \}.
\end{align}
We clearly have
\begin{align*}
  \left|v-\tilde v \right| 
 \leq   \big| \phi_i - \tilde \phi_i\big| 
 +  \big|\min_{j\neq i} \phi_j - \min_{j\neq i} \tilde \phi_j \big|
 \leq \sum_{i=1}^\numphases \big| \phi_i - \tilde\phi_i\big| \lesssim \left| u-\tilde u\right|,
\end{align*}
which together with Chebyshev's inequality yields the desired bound on the measure of the second right-hand side set of (\ref{chiandv}).
Therefore our goal is to prove
\begin{align}\label{la_1d_multiphase_pf2}
  |u_1-\tfrac12| \lesssim s \quad \text{or}\quad   \sum_{j\geq  3} \left[ u_j \wedge(1-u_j) \right]\gtrsim s  \qquad \text{on } \{|v| <s\},
\end{align}
which then implies (\ref{la_1d_multiphase_pf1}).

Now we give the argument for (\ref{la_1d_multiphase_pf2}).
First, we decompose the set $\{|v|<s\}$ in the following way:
$$
\{\left| v\right|<s\} = 
\bigcup_{j \neq i} E_j,
\quad E_j := \big\{\left| \phi_i-\phi_j\right|<s,\, \phi_j = \min_{k \neq i} \phi_k\big\}.
$$
We claim that
\begin{align}\label{uleq12}
 u_i,\, u_j \leq \frac12+\frac s{2\sigma_{ij}}, \quad
  u_k \leq \frac12, \; k\notin \{i,j\} \quad  \text{on } E_j.
\end{align}
Indeed, plugging in the definition of $\phi$, using the triangle inequality
for the surface tensions and $\sum_{l} u_l =1$, for $k\notin \{i,j\}$, we have on $E_j$
\begin{align*}
 \phi_j \leq \phi_k = \sum_{ l \neq k} \sigma_{kl} u_l
 \leq \sum_{ l \neq k} \sigma_{jl} u_l + \sigma_{jk} \sum_{ l\neq k}  u_l
 = \phi_j - \sigma_{jk} u_k + \sigma_{jk} (1-u_k)
 = \phi_j+ \sigma_{jk} (1-2u_k).
\end{align*}
Subtracting $\phi_j$ on both sides, we obtain $u_k\leq \frac12$.
Since $\phi_j-s \leq \phi_i$ on $E_j$ with the same chain of inequalities as before we obtain
\begin{align*}
   -s \leq \sigma_{ij} (1-2u_i).
\end{align*}
The same inequality holds for $u_j$ since $\phi_i -s  \leq \phi_j$, which concludes the argument for (\ref{uleq12}).

On the one hand, (\ref{uleq12}) gives us the upper bound for $u_1$ on $\{|v|<s\}$.
On the other hand, since $ u \wedge(1-u)  
=  u- (2u-1)_+$ for any $u$ 
we infer from (\ref{uleq12}) that on the set $\{u_1 \leq \frac12 - C s\}\cap \{|v|<s\}$ we have
\begin{align*}
 \sum_{j\geq 3} \left[ u_j \wedge(1-u_j) \right] =  1-u_1-u_2 - \sum_{j\geq 3} (2u_j-1)_+ \geq\left( C- \frac 1{\sigma_{\min}}\right) s \gtrsim s,
\end{align*}
if $C \geq \frac{2}{\sigma_{\min}}$.
This concludes the argument for (\ref{la_1d_multiphase_pf2}) and therefore the proof of the lemma.
\end{proof}

\begin{proof}[Proof of Corollary \ref{lem_local_estimates_multiphase}]
 By Lemma \ref{lem_local_estimates}, the claim follows from the obvious inequality
 \begin{align*}
   u_j \wedge(1-u_j) = G_h\ast \chi_j \wedge G_h \ast (1-\chi_j) \leq \left(1- \chi\right) G_h\ast \chi_j +   \chi\, G_h\ast \left(1- \chi_j\right).
 \end{align*}
\end{proof}

\begin{proof}[Proof of Lemma \ref{lem_1st_var_of_-diss}]
We recall the definition of the inner variation of $-E_h(\chi-\tilde \chi)$ in
(\ref{variation_D}): We have for any pair of admissible functions $\chi,\tilde\chi$
and any test function $\xi \in C^\infty(\torus,\R^d)$
\begin{align*}
-\delta E_h(\,\cdot\,-\tilde \chi) (\chi,\xi) 
 =& \frac2\h \sum_{i,j}\sigma_{ij} \int \left( \chi_i-\tilde\chi_i\right)G_h\ast
  \left( \xi\cdot \nabla \chi_j\right) dx\\
= & \frac2\h \sum_{ij}\sigma_{ij} \int \left( \chi_i-\tilde\chi_i\right)
  \left[ \nabla G_h\ast\left( \xi \,\chi_j\right)
  -G_h \ast \left(\left( \nabla\cdot \xi\right)
  \chi_j\right)\right] dx.
\end{align*}
In our case, after integration in time, this yields
\begin{align*}
 &\int_0^T -\delta E_h(\,\cdot\,-\chi^h(t-h)) (\chi^h(t),\xi(t)) \,dt\\
 &= \sum_{i,j}\sigma_{ij} h\sum_{n=1}^N \frac2\h \int \left( \chi_i^n-\chi_i^{n-1}\right)
\left[ \nabla G_h\ast\left( \overline\xi^n \chi_j^n\right)
- G_h \ast \left(\left( \nabla\cdot \overline\xi^n\right)\chi_j^n\right)\right] dx,
\end{align*}
where
\begin{align*}
 \overline\xi^n := \frac1h\int_{nh}^{(n+1)h} \xi(t)\,dt
\end{align*}
denotes the time average of $\xi$ over a microscopic time interval $[nh,(n+1)h)$.\\
Now we prove step by step that 
\begin{enumerate}
 \item  the $(\nabla\cdot\xi)$-term is negligible as $h\to0$;
 \item  we can freeze mesoscopic time for $\xi$, that is, substitute
	$ \overline\xi^n $ by some nearby value $\xi(l_n\tau)$ at the expense of an $o(1)$-term;
 \item	we can smuggle in $\eta$ at the expense of an $o(1)$-term;
 \item  we can freeze mesoscopic time for $\chi^h$ and substitute
	$\chi^n$ in the second factor by the mean \\
	$\frac12\left(\chi^h((l_n-1)\tau) + 
\chi^h(l_n\tau)\right)$, which is the main step;
 \item  we can get rid of $\eta$ again at the expense of an $o(1)$-term; and finally
 \item  we can pull $\xi$ out of the convolution at the expense of an $o(1)$-term.
\end{enumerate}
Note that Step 3 and Step 5 are just auxiliary steps for Step 4. 
\step{Step 1: The $(\nabla\cdot\xi)$-term vanishes as $h\to0$.}
 By Jensen's inequality, for any pair $i, j$ we have
\begin{align*}
\left| 
h \sum_{n=1}^N \frac1\h \int \left( \chi_i^n-\chi_i^{n-1}\right)G_h \ast
\left(\left(\nabla\cdot \overline\xi^n\right)\chi_j^n\right)\,dx\right|
\leq &\|\nabla\xi\|_{{\infty}}T \frac1\h \nmean \int \left|G_h \ast
\left( \chi_i^n-\chi_i^{n-1}\right)\right| dx\\
\lesssim &\|\nabla\xi\|_{{\infty}} T\frac1\h \left(\nmean \int 
\left|G_h \ast\left( \chi^n-\chi^{n-1}\right)\right|^2 dx\right)^{\frac12}.
\end{align*}
Since the $L^2$-norm of $G_h\ast u$ is decreasing in $h$ and by
the energy-dissipation estimate (\ref{energy_dissipation_estimate}),
the error is controlled by
\begin{align*}
  \|\nabla\xi\|_{{\infty}} T\frac1\h 
\left(\frac1N\h E_0\right)^{\frac12} 
\leq \|\nabla \xi\|_{{\infty}} E_0^{\frac12} T^{\frac12} h^{\frac14} = o(1).
\end{align*}
\step{Step 2: Time freezing for $\xi$.}
We can approximate $\overline\xi^n$ by a nearby value $\xi(l_n\tau)$,
where $l_n\in\{1,\dots \lmax\}$ is chosen such that 
$\kmax (l_n-1) < n\leq \kmax l_n$.
Note that $|\overline\xi^n - \xi^{l_n}| \leq \tau \|\partial_t \xi\|_{\infty}$. 
Therefore, by Jensen's inequality, we have for any pair $i,j$
\begin{align*}
 \left|h\sum_{n=1}^N \frac1\h \int \left( \chi_i^n-\chi_i^{n-1}\right)
\nabla G_h\ast\big((\xi^{l_n}- \overline\xi^n ) \,\chi_j^n\big) dx\right|
\leq &\alpha \|\partial_t \xi\|_{\infty}T \nmean 
\int \left|\nabla G_h\ast\left( \chi_i^n-\chi_i^{n-1}\right)\right| dx\\
\lesssim &\alpha \|\partial_t \xi\|_{\infty} T\left( \!\! \nmean 
\int \left|\nabla G_h\ast\left( \chi^n-\chi^{n-1}\right)\right|^2 dx \!\right)^{\frac12}\!\!.
\end{align*}
But $\h\|\nabla G_h\ast u\|_{L^2} \lesssim\| G_{h/2}\ast u\|_{L^2}  $ yields
\begin{align*}
\int \left|\nabla G_h\ast\left( \chi^n-\chi^{n-1}\right)\right|^2 dx 
\lesssim \frac1h \int \left[G_{h/2}\ast\left( \chi^n-\chi^{n-1}\right)\right]^2 dx.
\end{align*}
Using the energy-dissipation estimate (\ref{energy_dissipation_estimate}), 
the error is controlled by
\begin{align*}
  \alpha \|\partial_t \xi\|_{\infty}T \left( \frac1N\frac1\h E_0\right)^{\frac12}
 = \alpha \|\partial_t \xi\|_{\infty} E_0^{\frac12}T^{\frac12} h^{\frac14}=o(1).
\end{align*}
\step{Step 3: Smuggling in $\eta$.}
We claim
\begin{align*}
 &h\sum_{n=1}^N \frac1\h \int \left( \chi_i^n-\chi_i^{n-1}\right)
\nabla G_h\ast\left( \xi(l_n\tau) \chi_j^n\right) dx\\
&=h\sum_{n=1}^N \frac1\h \int \eta \, G_{h/2}\ast \left( \chi_i^n-\chi_i^{n-1}\right)
\nabla G_{h/2}\ast\left( \xi(l_n\tau) \chi_j^n\right)
 dx + o(1)
\quad\text{as } h\to 0.
\end{align*}
Using $\nabla G_h = G_{h/2}\ast \nabla G_{h/2}$, the left-hand
side is equal to
\begin{align*}
 h\sum_{n=1}^N \frac1\h \int G_{h/2}\ast \left( \chi_i^n-\chi_i^{n-1}\right)
 \nabla G_{h/2}\ast\left( \xi(l_n\tau) \chi_j^n\right) 
 dx.
\end{align*}
Note that since $\eta \equiv 1$ on the support of $\xi$ and 
$\left|z\right| \left|\nabla G_{1/2}(z) \right| \lesssim \left|z\right|^2 G(z)$ has
finite integral, we have
for any $\chi\in\{0,1\}$,
\begin{align*}
\left| (1-\eta) \nabla G_{h/2} \ast \left( \xi \chi \right) \right|
=&\left| \int \nabla G_{h/2}(z) (\eta(x+z)-\eta(x))  \xi(x+z) \chi(x+z) \, dz \right|\\
\lesssim& \|\nabla\eta \|_\infty \|\xi\|_\infty \int \left|z\right| \left|\nabla G_{h/2}(z) \right| \,dz
\lesssim  \frac1r \|\xi\|_\infty.
\end{align*}
Thus, using the Cauchy-Schwarz inequality and 
the energy-dissipation estimate (\ref{energy_dissipation_estimate}),
the error is controlled by
\begin{align*}
h^{\frac14}\left( \sum_{n=1}^N \frac1\h \int \left| G_{h/2}\ast 
\left( \chi^n-\chi^{n-1}\right)\right|^2 dx \right)^{\frac12}
\left( h\sum_{n=1}^N \left( \frac1r \|\xi\|_\infty\right)^2\right)^{\frac12}
\lesssim E_0^{\frac12} T^{\frac12}\frac1r \|\xi\|_\infty h^{\frac14}=o(1).
\end{align*}
\step{Step 4: Time freezing for $\chi^h$.}
We claim that for any pair of indices $i, j$
\begin{align*}
  &h\sum_{n=1}^N \frac2\h \int \eta \,G_{h/2}\ast \left( \chi_i^n-\chi_i^{n-1}\right)
  \nabla G_{h/2}\ast\left( \xi(l_n\tau) \chi_j^n\right)
  dx\\
&\approx h\sum_{n=1}^N \frac1\h \int \eta \,G_{h/2}\ast \left( \chi_i^n-\chi_i^{n-1}\right)
\nabla G_{h/2}\ast\left( \xi(l_n\tau) \left( \chi_j^h((l_n-1)\tau)+
\chi_j^h(l_{n}\tau)\right)\right)  dx,
\end{align*}
in the sense that the error is controlled by
\begin{align*}
\|\xi\|_\infty \left(\frac1\alpha \tau \sum_{l=1}^\lmax \ggap(\chi^{Kl-1}) + \alpha^{\frac13} \iint \eta \,d\mu_h
 + \alpha^{\frac13} r^{d-1}T
 \right) + o(1),
\quad \text{as }h\to0.
\end{align*}
Here, we assumed that Phases $1$ and $2$ are the majority phases in the support of $\eta$.
Indeed, we can control the error using the Cauchy-Schwarz inequality by
\begin{align*}
&\left( \sum_{n=1}^N\frac1\h \int \eta^2 \left| G_{h/2}\ast 
\left(\chi^n-\chi^{n-1}\right)\right|^2 dx  \right)^{\frac12}\times\\
 &  \qquad \left(
 \tau \sum_{l=1}^\lmax \frac1\kmax \sum_{k=1}^\kmax
 \frac1\h\int \Big[ \h
\nabla G_{h/2} \ast \Big(\xi(l\tau) 
\big[  \chi_j^{\kmax (l-1) + k} -\tfrac12 \big(\chi_j^{\kmax (l-1)}
+\chi_j^{\kmax l} \big) \big]\Big)\Big]^2 dx \right)^{\frac12}.
\end{align*}
Since $0\leq \eta\leq 1$, the term in the first parenthesis
is controlled by $\iint \eta \, d\mu_h$.
For the term in the second parenthesis, we fix 
the mesoscopic block index $l$ and the microscopic time step index $k$
and sum at the end. 
Let $l=1$ and write $\xi$ instead of $\xi(l\tau)$ for notational simplicity.
We use the $L^2$-convolution estimate and introduce $\eta$ in the second integral, which is equal to
$1$ on the support of $\xi$:
\begin{align*}
\lefteqn{\frac1\h\int \Big[ \h
\nabla G_{h/2} \ast \Big(\xi(l\tau) 
\big[  \chi_j^{k} -\tfrac12 \big(\chi_j^{0}
+\chi_j^{\kmax} \big) \big]\Big)\Big]^2 dx}\\
&\leq  \frac1\h \left( \int  \big| \h\nabla G_{h/2}\big|dz\right)^2 \int 
\left| \xi\right|^2\left[  \chi_j^{ k} 
-\tfrac12\left(\chi_j^{0}+\chi_j^{\kmax}\right)\right]^2 dx\\
&\lesssim      \left\| \xi \right\|^2_\infty\left( 
\frac1\h\int\eta \left|\chi^{k} -\chi^{0}\right|dx + 
\frac1\h\int\eta \left|\chi^{\kmax } -\chi^{k}\right|dx \right).
\end{align*}
With Lemma \ref{la_1d_multiphase} in the integrated form and Corollary \ref{lem_local_estimates_multiphase}, 
we can estimate these terms in the following way.
We set for abbreviation
\begin{align*}
 \alpha^2(k,k'):= 
\frac1\h\int\eta \left(G_{h}\ast\big(\chi^{k}-\chi^{k'}\big)\right)^2dx.
\end{align*}
By Minkowski's triangle inequality w.\ r.\ t.\ the measure $\eta \,dx$,
we see that $\alpha$ also satisfies a triangle inequality. Thus, thanks to 
Jensen's inequality,
\begin{align*}
 \alpha^2(k-1,-1) \leq& \left(\sum_{n=0}^{k-1}\alpha(n,n-1) \right)^2
\leq  k \sum_{n=0}^{k-1}\alpha^2(n,n-1)
\leq  \kmax \sum_{n=0}^{\kmax-1}\alpha^2(n,n-1).
\end{align*}
Therefore, by integrating Lemma \ref{la_1d_multiphase} over the tangential directions $x_2,\dots,x_d$ and using Corollary \ref{lem_local_estimates_multiphase}, we have
\begin{align}\label{apply 1d lem}
  \frac1\h\int \eta \left|\chi^{k} -\chi^{0}\right|dx
\lesssim &\frac1s \gap(\chi^{-1})
+ sr^{d-1} + \frac1{s^2}\kmax \sum_{n=0}^{\kmax-1}\alpha^2(n,n-1) + o(1).
\end{align}
By (\ref{diss meas}) we have $\sum_n \alpha^2(n,n-1) \leq \iint \eta \, d\mu_h$ and the relation $K\tau = \alpha^2$, we have
\begin{align}\label{alpha and alpha}
 \tau \sum_{l=1}^\lmax \alpha^2(\kmax l -1, \kmax(l-1)-1) \leq \alpha^2 \iint \eta \, d\mu_h.
\end{align}
This justifies the name $\alpha^2(k,k')$, since the term arising from $\alpha^2(k,k')$ is estimated in (\ref{alpha and alpha}) by $\alpha^2$,
the square of the fudge factor in the definition of the
mesoscopic time scale $\tau = \alpha \h$.
Therefore, after summation over the mesoscopic block index $l$, (\ref{apply 1d lem}) in conjunction with (\ref{alpha and alpha}) yields
\begin{align*}
  \tau \sum_{l=1}^\lmax \frac1\kmax \sum_{k=1}^\kmax \frac1\h\int \eta \left|\chi^{\kmax l+k} 
 -\chi^{\kmax l}\right|dx 
\lesssim \frac1s  \tau \sum_{l=1}^\lmax \ggap(\chi^{Kl-1})
+  s r^{d-1}T
+\frac1{s^2}\alpha^2 \iint \eta \, d\mu_h.
\end{align*}
Using Young's inequality, the total error in this step is controlled by $\| \xi\|_\infty$ times
\begin{align*}
&\left( \iint \eta \, d\mu_h\right)^{\frac12} \left( \frac1s \tau \sum_{l=1}^\lmax \ggap(\chi^{Kl-1})
+ s r^{d-1}T + \left(\frac\alpha s\right)^2\iint \eta \, d\mu_h\right)^{\frac12}\\
&\leq 
\frac1{\alpha}  \tau \sum_{l=1}^\lmax \ggap(\chi^{Kl-1}) +  \frac{s^2}\alpha r^{d-1}T +  \frac{\alpha}s \iint \eta \, d\mu_h.
\end{align*}
If we now choose $s= \alpha^{\frac23} \ll 1$, this is the desired error term.
\step{Step 5: Getting rid of $\eta$ again.} As in Step 3, we can estimate
\begin{align*}
& h \sum_{n=1}^N\frac1\h \int \eta \,G_{h/2}\ast \left( \chi_i^n-\chi_i^{n-1}\right)
\nabla G_{h/2}\ast\left[ \xi(l_n\tau) \left( \chi_j^h((l_n-1)\tau)+
\chi_j^h(l_{n}\tau)\right)\right] dx\notag\\
&= h \sum_{n=1}^N \frac1\h \int \left( \chi_i^n-\chi_i^{n-1}\right)
\nabla G_{h}\ast\left[ \xi(l_n\tau) \left( \chi_j^h((l_n-1)\tau)+
\chi_j^h(l_{n}\tau)\right)\right] dx + o(1),
\end{align*}
as $h\to 0$.
%
\step{Step 6: Pulling out $\xi$.}
First, fix $l$ and write $\xi=\xi(l\tau)$. For simplicity of the formula,
we will ignore $l$ and formally set $l=1$.
Note that since $\nabla G$ is antisymmetric, we have for any two functions $\chi,\,v$,
\begin{align*}
 &\int v
\left[ \xi\cdot \nabla G_h\ast \chi -
 \nabla G_h\ast\left( \xi\,\chi\right)\right]dx
=  \int \nabla G_h(z) \int v(x+z)\, \chi(x)
 \left( \xi(x+z)-\xi(x) \right)dx\,dz.
\end{align*}
Set $ K(z) := z\otimes z \;G(z)$, take a Taylor-expansion of 
$\xi$ around $x$: $\xi(x+z)-\xi(x) = \nabla \xi(x)\,z + O(|z|^2)$, 
where the constant in the $O(|z|^2)$-term is depending linearly on 
$\|\nabla^2\xi\|_{\infty}$.
Then the error on this single time interval splits into two terms.\\
The one coming from the first-order term in the expansion of $\xi$ is
\begin{align*}
 \left| \kmean \frac1\h \int  \nabla \xi \colon 
K_h \ast \left(\chi_i^k-\chi_i^{k-1}\right) \left(\chi_j^0+\chi_j^\kmax\right) dx\right|
\lesssim \|\nabla \xi\|_{{\infty}} \frac1\h \left(\kmean \int \left| K_h \ast
 \left(\chi^k-\chi^{k-1} \right)\right|^2 dx \right)^{\frac12}\!\!,
\end{align*}
where we used Jensen's inequality. 
Since $K_h = h \nabla^2 G_h +   G_h\,Id$,
$ \| h \nabla^2 G_h\ast u\|_{L^2} \lesssim \|G_{h/2}\ast u\|_{L^2}$
for any $u$ and since
the $L^2$-norm of $G_h\ast u$ is non-increasing in $h$, 
we have for any function $v$
\begin{align*}
 \int \left| K_h \ast v\right|^2 dx \leq
  \int \left| h\nabla^2 G_h \ast v\right|^2 dx
+ \int \left( G_h \ast v\right)^2 dx
\lesssim  \int \left( G_{h/2} \ast v\right)^2 dx.
\end{align*}
Plugging this into the inequality above with $v$ playing the role of $\chi_i^{k}-\chi_i^{k-1}$, 
multiplying by $\tau$, summing over the block index $l$ and 
using Jensen's inequality, we can control the contribution to the error coming 
from the first-order term by
\begin{align*}
   T \|\nabla \xi\|_{{\infty}}\frac1\h \left(\nmean
 \int \left| G_{h/2} \ast \left(\chi^n-\chi^{n-1} \right)\right|^2 dx \right)^{\frac12}
\leq \|\nabla \xi\|_{{\infty}} E_0^{\frac12}T^{\frac12} h^{\frac14} =o(1).
\end{align*}
By Lemma \ref{comp_lem_tau_shift}, the contribution coming from 
the second-order term in the expansion of $\xi$ is controlled by
\begin{align*}
 & \|\nabla^2\xi\|_{{\infty}} h \sum_{n=1}^N
\int \left( \frac{|z|}{\h}\right)^3 G_h(z) 
\int \left| \chi^n - \chi^{n-1}\right| dx\,dz\\
&\lesssim \|\nabla^2\xi\|_{{\infty}}  \int_0^T 
\int \left| \chi^h(t) - \chi^h(t-h)\right| dx\,dt
 \lesssim \|\nabla^2\xi\|_{{\infty}} E_0 (1+T)\h=o(1).
\end{align*}
Finally, we note that by the time freezing in Step 4, we constructed a telescope sum:
Rewriting the summation over the microscopic time step index $n=1,\dots,N$ 
as the double sum over the microscopic time step index $k=1,\dots,\kmax$
in the respective mesoscopic time intervals and 
the mesoscopic block index $l=1,\dots,\lmax$, we have for each $l$,
\begin{align*}
&\sum_{k=1}^\kmax\left( \chi_i^{\kmax (l-1) + k}-\chi_i^{\kmax (l-1) + k -1}\right)
 \xi(l \tau)\cdot\nabla G_{h}\ast \left( \chi_j^{\kmax(l-1)}+
\chi_j^{\kmax l}\right) \\
&=   \left( \chi_i^{\kmax l}-\chi_i^{\kmax(l-1)}\right)
\xi(l \tau)\cdot\nabla G_{h}\ast \left( \chi_j^{\kmax(l-1)}+
\chi_j^{\kmax l}\right).
\end{align*}
Thus we obtain
\begin{align*}
 &h \sum_{n=1}^N \frac1\h \int \left( \chi_i^n-\chi_i^{n-1}\right)
\nabla G_{h}\ast\left( \xi(l_n\tau) \left( \chi_j^h((l_n-1)\tau)+
\chi_j^h(l_{n}\tau)\right)\right) dx \\
&=
\tau \sum_{l=1}^{\lmax} \frac1\tau \int \left( \chi_i^{\kmax l}-\chi_i^{\kmax(l-1)}\right)
\xi(l \tau)\cdot\left( \h \nabla G_{h}\right)\ast \left( \chi_j^{\kmax(l-1)}+
\chi_j^{\kmax l}\right) dx + o(1),
\end{align*}
which concludes the proof.
\end{proof}

\begin{proof}[Proof of Lemma \ref{lem_1st_var_of_-diss_K_h}]
\step{Step 1: Rough estimate for minority phases.}
We first argue that if $\{i,j\} \neq \{1,2\}$, that is if the product involves at least one minority phase,
then we can estimate this term.
Let us first assume that $j\notin \{1,2\}$.
By a manipulation as in the proof of Lemma \ref{lem_1st_var_of_-diss} and
the Cauchy-Schwarz inequality, we have
\begin{align*}
&\left| \tau \sum_{l=1}^\lmax 
\int  \frac{\chi_i^{\kmax l}
- \chi_i^{\kmax (l-1)}}\tau\, 
\xi(l \tau)\cdot \left( \h\,\nabla G_h\right)
\ast\left(\chi_j^{\kmax(l-1)}+\chi_j^{\kmax l}\right)dx\right| \\
&\lesssim \| \xi\|_\infty\! \left( \sum_{l=1}^\lmax 
\int \eta \left| G_{h/2}\ast \left( \chi_i^{\kmax l}-\chi_i^{\kmax (l-1)} \right) \right|^2 dx
\right)^{\frac12}
\left( \sum_{l=0}^\lmax 
\int \eta \left| \h \nabla G_{h/2}\ast  \chi_j^{\kmax l}\right|^2 dx
\right)^{\frac12} + o(1),
\end{align*}
as $h\to 0$.
Note that for any characteristic function $\chi\in\{0,1\}$, since $\int \nabla G(z)\,dz=0$,
\begin{align*}
\frac1\h \int \eta \left| \h \nabla G_{h/2} \ast \chi \right|^2 dx
&\lesssim \frac1\h \int \eta \left| \h \nabla G_{h/2} \ast \chi \right| dx\\
&\lesssim\frac1\h \int \left| \h \nabla G_{h/2}(z)\right| \int\eta(x) 
\left|\chi(x+z) -\chi(x)\right| dx \,dz\\
&\lesssim \frac1\h \int G_{h}(z) \int\eta(x) 
\left|\chi(x+z) -\chi(x)\right| dx \,dz\\
&=\frac1\h \int \eta \left[ \left(1-\chi\right) G_h\ast \chi 
+ \chi \,G_h\ast \left(1-\chi\right)  \right]dx.
\end{align*}
Treating the metric term as in the proof of Lemma \ref{lem_1st_var_of_-diss} 
with the triangle inequality and Jensen's inequality afterwards, we obtain the bound
\begin{align*}
 \| \xi\|_\infty
 \left( \iint \eta \,d\mu_h \right)^{\frac12}
\left( \tau \sum_{l=1}^\lmax \ggap(\chi^{Kl})\right)^{\frac12} + o(1) 
\leq \| \xi\|_\infty
\left( \frac1\alpha  \tau \sum_{l=1}^\lmax \ggap(\chi^{Kl}) + \alpha \iint \eta \,d\mu_h  \right) \!+ o(1).
\end{align*}
If instead $i\notin\{1,2\}$, using a discrete integration by parts, the antisymmetry of $\nabla G$
and a manipulation as in the proof of Lemma \ref{lem_1st_var_of_-diss}
for $\xi$, we can exchange the roles of Phase $i$ and Phase $j$:
\begin{align*}
& \tau \sum_{l=1}^\lmax 
\int  \frac{\chi_i^{\kmax l}
- \chi_i^{\kmax (l-1)}}\tau\;
\xi(l \tau)\cdot \left( \h\,\nabla G_h\right)
\ast\left(\chi_j^{\kmax(l-1)}+\chi_j^{\kmax l}\right) dx\\
&= -\tau \sum_{l=1}^\lmax 
\int  \frac{\chi_j^{\kmax l}
- \chi_j^{\kmax (l-1)}}\tau\;
\xi(l \tau)\cdot \left( \h\,\nabla G_h\right)
\ast\left(\chi_i^{\kmax(l-1)}+\chi_i^{\kmax l}\right) dx + o(1).
\end{align*}
Thus, we can use the above argument also in this case and the only terms contributing
to the sum as $h\to 0$ are the terms involving both majority phases.\\
In the following we assume that $i=1$ and $j=2$.
In the other case, we can just exchange the roles of $\chi_1$ and $\chi_2$ in the following steps
and use $-e_1$ as the approximate normal to $\chi_2$ instead and the proof is the same.
\step{Step 2: Substituting $\chi_2$ by $1-\chi_1$.}
We claim that
\begin{align*}
 &\tau \sum_{l=1}^\lmax 
\int  \frac{\chi_1^{\kmax l}
- \chi_1^{\kmax (l-1)}}\tau\;
\xi(l \tau)\cdot \left( \h\,\nabla G_h\right)
\ast\left(\chi_2^{\kmax(l-1)}+\chi_2^{\kmax l}\right) dx\\
&= -
\tau \sum_{l=1}^\lmax 
\int  \frac{\chi_1^{\kmax l}
- \chi_1^{\kmax (l-1)}}\tau\;
\xi(l \tau)\cdot \left( \h\,\nabla G_h\right)
\ast\left(\chi_1^{\kmax(l-1)}+\chi_1^{\kmax l}\right) dx +o(1).
\end{align*}
Since $\nabla G \ast 1 =0$, the claim is clearly equivalent to proving
that we can replace $\chi_2$ by $1-\chi_1$ in the second left-hand side term of the claim.
But by $\sum_i \chi_i =1$ and linearity the resulting error term is
\begin{align*}
 \left|\sum_{j=3}^\numphases \tau \sum_{l=1}^\lmax 
\int  \frac{\chi_1^{\kmax l}
- \chi_1^{\kmax (l-1)}}\tau\;
\xi(l \tau)\cdot \left( \h\,\nabla G_h\right)
\ast \left(\chi_j^{\kmax(l-1)}+\chi_j^{\kmax l}\right) dx \right|,
\end{align*}
which can be handled by Step 1.
\step{Step 3: Substitution of $\nabla G$.}
We want to replace the convolution with $\nabla G$ on the left-hand side of the claim
by a convolution with the anisotropic kernel
\begin{align*}
 K(z):=\sign(z_1) \,z\,G(z).
\end{align*} 
To that purpose, we claim that for any characteristic function $\chi\in \{0,1\}$,
\begin{align}\label{eq_periodic_kernel_estimate}
 \frac1\h \int \!\eta \Big|  \h \;\nabla G_h\ast \chi - 
\left( \chi \,K_{h}\ast(1-\chi) + (1-\chi)K_{h}\ast \chi \right) \Big| dx
 \lesssim  \varepsilon^2+ \frac\h {r^2} + \frac \h r E_h(\chi).
\end{align}
Here,
\begin{align*}
  \varepsilon^2 := \inf_{\chi^\ast} \Bigg \{
  &\frac1\h \int  \eta \left[\chi \,G_h\ast
 (1-\chi) + (1-\chi)\,G_h\ast \chi \right] dx\\
 -&\frac1\h \int \eta\left[\chi^\ast\, G_h\ast (1-\chi^\ast)+ 
(1-\chi^\ast)\,G_h\ast \chi^\ast \right]dx
 + \frac1r \int_{B_{2r}}  \left|\chi-\chi^\ast\right|dx
 \Bigg\},
\end{align*}
where the infimum is taken over all half spaces $\chi^\ast = \chara_{\{x_1>\lambda\}}$ in direction $e_1$.
Using this inequality for $\chi_1^{\kmax (l-1)}$ and $\chi_1^{\kmax l}$,
we can substitute those two summands and the error is estimated as desired:
\begin{align*}
&
\frac1\alpha \|\xi\|_{\infty} \tau \sum_{l=0}^\lmax \frac1\h \int\!\eta \left| \h\, 
\nabla G_h\ast \chi_1^{\kmax l}
- \left[ \chi_1^{\kmax l} \,K_h\ast\left(1-\chi_1^{\kmax l}\right) + 
\left(1-\chi_1^{\kmax l}\right)K_h \ast \chi_1^{\kmax l} \right] \right| dx\\
&\lesssim\frac1\alpha  \|\xi\|_{\infty} 
  \tau \sum_{l=0}^\lmax  \ggap(\chi^{\kmax l}) +  o(1), \quad \text{as } h\to 0.
\end{align*}
Now we give the argument for (\ref{eq_periodic_kernel_estimate}).
By measuring length in terms of $\h$, we may assume that $h=1$. 
Since $\int \nabla G\,dz=0$ and $\nabla G(z) = -z\, G(z)$, 
using the identities $u=|u|-2u_-$ and $u=-|u|+2u_+$ on the sets $\{z_1>0\}$ and $\{z_1<0\}$, respectively,
\begin{align*}
 \nabla G\ast \chi
= &\int  z\; G(z) \left( \chi(x+z)-\chi(x)\right)dz\\
=& \int_{\{z_1>0\}} K(z) \left| \chi(x+z)-\chi(x)\right|dz 
- 2  \int_{\{z_1>0\}}\! z \,G(z) \left(\chi(x+z)-\chi(x)\right)_-dz\\
+& \int_{\{z_1<0\}} K(z) \left| \chi(x+z)-\chi(x)\right|dz 
+ 2  \int_{\{z_1<0\}} \!z\, G(z) \left(\chi(x+z)-\chi(x)\right)_+dz.
\end{align*}
Using $|\chi_1-\chi_2|=(1-\chi_1)\chi_2 + \chi_1(1-\chi_2)$ for $\chi_1,\,\chi_2\in\{0,1\}$, 
this implies the pointwise identity
\begin{align*}
 \nabla G\ast \chi 
= \chi\,K\ast(1-\chi) + (1-\chi) K\ast \chi
 - 2 \int_{\{ z_1\lessgtr 0 \}}\!\sign(z_1)\, z\, G(z) \left(\chi(x+z)-\chi(x)\right)_\pm dz,
\end{align*}
where the last term stands for the sum of the two integrals.
Integration w.\ r.\ t.\ $\eta \,dx$ now yields:
\begin{align*}
 \int \eta \Big|  \nabla G\ast \chi - \left( \chi K\ast(1-\chi) + 
 (1-\chi)K\ast \chi \right) \Big| dx\lesssim \int_{\{z_1\lessgtr 0\}} |z| G(z) \int \eta(x)\left( \chi(x+z)-\chi(x) \right)_\pm dx\,dz.
\end{align*}
As in the argument for (\ref{eq_periodic_2}), we can follow the lines from 
(\ref{argumentforlocestimatesusedfor2ndlemma}) on so 
that (\ref{eq_periodic_1}) yields (\ref{eq_periodic_kernel_estimate}).
\step{Step 4: An identity for $K$.}
We claim that for any two characteristic functions $\chi,\,\tilde\chi\in\{0,1\} $,  
we have the pointwise identity
\begin{align*}
&\left( \chi-\tilde\chi\right)\big( \chi\, K_h\ast(1-\chi) + (1-\chi)K_h\ast \chi
+ \tilde \chi\, K_h\ast(1-\tilde \chi) + (1-\tilde \chi)K_h\ast \tilde \chi\big)\\
&\qquad\qquad= 2c_0\, e_1 \left( \chi-\tilde\chi\right) -
\left| \chi-\tilde\chi\right| K_h\ast \left( \chi-\tilde\chi\right).
\end{align*}
Indeed, by scaling, we may w.\ l.\ o.\ g.\ assume $h=1$ and start with
\begin{align*}
 \left(\chi-\tilde\chi\right)\tilde\chi\,K\ast \left(1-\tilde\chi\right)
 + \left(\chi-\tilde\chi\right)\left(1-\tilde\chi\right)K\ast \tilde\chi &=   \left(\chi-1\right)\tilde\chi\left(\int \!K-K\ast \tilde\chi\right)
 + \chi\left(1-\tilde\chi\right)K\ast \tilde\chi\\
&=   \left(\chi-1\right)\tilde\chi \left(\int\! K\right)
 + \big( \left(1-\chi\right)\tilde\chi + \chi\left(1-\tilde\chi\right) \big) K\ast\tilde\chi\\
&=   \left(\chi-1\right)\tilde\chi  \left(\int\! K\right)
 + \left|\chi-\tilde\chi\right| K\ast \tilde\chi.
\end{align*}
Exchanging the roles of $\chi$ and $\tilde \chi$, one obtains
for the second part
\begin{align*}
  \left(\chi-\tilde\chi\right)\chi\,K\ast \left(1-\chi\right)
 + \left(\chi-\tilde\chi\right)\left(1-\chi\right)K\ast \chi 
= -\left(\tilde\chi-1\right)\,\chi \left(\int\! K\right)
 - \left|\chi-\tilde\chi\right| K\ast \chi.
\end{align*}
Using the factorization property of $G$ and the symmetry $\int z'G^{d-1}(z')dz'=0$, one computes that
for any vector $\xi\in \R^d$
\begin{align*}
 \xi \cdot\int K 
=& \int \sign(z_1) \int \left(\xi_1\,z_1+ \xi'\cdot z'\right)
	  G^{d-1}(z')\,dz'\,G^1(z_1)\,dz_1
=  \xi_1 \int |z_1|\, G^{1}(z_1)\,dz_1\\
= & 2\,\xi_1\int_0^\infty z_1\,G^1(z_1)\,dz_1
= 2\,\xi_1 \int_0^\infty-\frac{d}{d z_1} G^1(z_1)\,dz_1
=  2 \,\xi_1 G ^1(0) = 2\, \xi_1 \frac1{\sqrt{2\pi}} = 2\,c_0 \xi_1.
\end{align*}
Hence the identity follows from $\left(\chi-1\right)\tilde\chi - \left(\tilde \chi-1\right)\chi
 = \chi-\tilde \chi$.
\step{Step 5: Conclusion.}
Applying Steps 1 and 2, using the identity in Step 3 for the remaining 
two terms involving Phases $1$ and $2$, we end up with the right-hand side of the claim.
The error is controlled by
\begin{align*}
  \|\xi\|_{\infty} \Bigg(\frac1\alpha& \tau \sum_{l=1}^\lmax \ggap(\chi^{\kmax l})
 + \alpha \iint \eta \, d\mu_h
 + \tau \sum_{l=1}^\lmax 
 \frac1\tau \int \eta^2 \,\big| \chi^{\kmax l} - \chi^{\kmax(l-1)}\big| \, \left| K_h\right|\ast
 \,\big| \chi^{\kmax l} - \chi^{\kmax(l-1)}\big|\, dx
 \Bigg)
 +o(1),
\end{align*}
as $h\to0$.
Note that $\left| K\right|= k$, where $k$ is the kernel defined in the statement of the lemma.
It remains to argue that $\eta$ can be equally distributed on both copies of
$\left|\chi^{\kmax l}-\chi^{\kmax (l-1)}\right|$.
For this, note that for $u=\left| \chi^{\kmax l} - \chi^{\kmax(l-1)}\right|\in [0,1]$,
\begin{align*}
 \frac1\h \left|  \int \eta^2 u
\,k_h\ast u\,dx -  \int \eta \,
u
\,k_h\ast (\eta\, u)\,dx\right| 
&\leq \frac1\h\int k_h(z) \int \eta(x) \,u(x) \,
u(x+z) \left|\eta(x+z) - \eta(x)\right| dx \,dz\\
& \leq \left\|\nabla \eta \right\|_\infty\int \frac{|z|}\h k_h(z)\, dz
\int \eta \, u\, dx
\lesssim \frac1r\int u\, dx.
\end{align*}
Thus, in our case, we can use Lemma \ref{comp_lem_hoelder} and bound the error by
\begin{align*}
\frac1\alpha \|\xi\|_\infty \frac1r \tau \sum_{l=1}^\lmax 
\int \left|\chi^{\kmax l}-\chi^{\kmax(l-1)}\right|dx
\lesssim \frac1{\alpha^{\frac12}} \|\xi\|_\infty \frac1r  E_0 T h^{\frac14 }=o(1).
\end{align*}
\end{proof}

\begin{proof}[Proof of Lemma \ref{lem_1st_var_of_-diss_remainder}]
First, we note that it is enough to prove the following similar statement for a fixed mesoscopic time interval:
\begin{align}\label{decayofrem_pf}
 &\frac1\tau \int \eta \left| \chi^{\kmax} 
- \chi^{0}\right|\, k_h\ast
 \left( \eta \left| \chi^{\kmax} - \chi^{0}\right|\right) dx
\lesssim 
  \frac1{\alpha^2}\ggap(\chi^{-1}) + \alpha^{\frac19} r^{d-1}
 + \alpha^{\frac19} \frac1\tau \int_0^\tau\int  \eta\,d\mu_h.
\end{align}
Indeed, if we multiply (\ref{decayofrem_pf}) by $\tau$ and sum over the mesoscopic block index $l$
we obtain the statement.

In the proof of (\ref{decayofrem_pf}),
we will exploit the convolution in the normal direction $e_1$ in Step 1, which will
allow us in Step 2 to make use of the quadratic structure of this term.
\step{Step 1:}
We can estimate the kernel $k$ by a kernel that factorizes in
two kernels $k^1$, $k'$ in normal- and tangential
direction, respectively, which are of the form
\begin{align*}
 k^1(z_1) := &(1+z_1^2)^{\frac12}\,G^1(z_1),\\
 k'(z')  := & (1+|z'|^2)^{\frac12}\,G^{d-1}(z').
\end{align*}
Let us still denote the kernel by $k$.
We have
\begin{align*}
k_h\ast \left(\eta |\chi^\kmax-\chi^0|\right)
   \leq   \sup_{x_1} \left\{  k'_h\ast' 
k^1_h\ast_1\left(\eta |\chi^\kmax-\chi^0|\right) \right\} 
\leq & \,k'_h\ast' \sup_{x_1} 
\left\{k^1_h\ast_1\left(\eta |\chi^\kmax-\chi^0|\right) \right\} .
\end{align*}
The second factor in the right-hand side convolution can be estimated in two ways:
\begin{align*}
 \sup_{x_1} \left\{k^1_h\ast_1 \left(\eta |\chi^\kmax-\chi^0|\right) \right\} 
\leq & \min \left\{ \int k^1_h\, dz_1 \,  \sup_{x_1}\left(\eta |\chi^\kmax-\chi^0|\right)
  , \, \left(\sup_{x_1} k^1_h\right) \int  \eta\left|\chi^\kmax-\chi^0\right| dx_1\right\}\\
\lesssim & \min \left\{ 1,\, \frac1\h \int \eta \left|\chi^\kmax-\chi^0\right|dx_1 \right\}.
\end{align*}
Therefore, we obtain a quadratic term with two copies of
$\frac1\h \int \eta \left|\chi^\kmax-\chi^0\right|dx_1 $:
\begin{align}
&\frac1\tau \int \eta \left|\chi^\kmax-\chi^0\right|
\,k_h\ast \left(\eta\left|\chi^\kmax-\chi^0\right|\right)dx \notag\\
&\qquad\lesssim 
\frac1\alpha \int \left(\frac{1}{\sqrt{h}}\int \eta\left|\chi^\kmax-\chi^0\right| dx_1\right) 
k_h'\ast'\left(1\wedge \frac{1}{\sqrt{h}} \int \eta \left|\chi^\kmax-\chi^0\right|dx_1 \right)dx'.\label{pf quadratic term}
\end{align}
\step{Step 2:}
Now we use Lemma \ref{la_1d_multiphase} before integration in $x'$.
We write $\gap(x')$ for the first error term in (\ref{1dmulti_eq1}) and set
\begin{align*}
 \alpha^2(x'):=  \frac1\h \int \eta \left| G_h\ast \left( \chi^{\kmax-1}
 - \chi^{-1}\right)\right|^2 dx_1,
\end{align*}
so that the statement of Lemma \ref{la_1d_multiphase} turns into
\begin{align*}
 \frac{1}{\sqrt{h}}\int \eta\left|\chi^\kmax-\chi^0\right| dx_1 \lesssim \frac1s \gap(x') + s \chara_{\{|x'| < 2r\}} + \frac1{s^2} \alpha^2(x').
\end{align*}
We recall the link between the function $\alpha^2(x')$ and the fudge factor $\alpha$ as mentioned in (\ref{alpha and alpha}) but now before summation over the
mesoscopic block index $l$:
\begin{align}\label{pf quadratic term alpha and alpha}
 \int \alpha^2(x')\,dx' \leq \frac{\alpha^2}\tau \int_0^\tau\int\eta\,d\mu_h.
\end{align}
Then for any two parameters $s, \tilde s \ll1$ the right-hand side of (\ref{pf quadratic term}) is estimated by
\begin{align}\label{pf quadratic term last}
\frac1\alpha  \int \left( \frac1s \gap(x') + s + \frac1{s^2} \alpha^2(x') \right)
 k_h'\ast' \left( 1\wedge \left( \frac1{\tilde s} \gap(x') + \tilde s \chara_{\{|x'| < 2r\}}+ \frac1{\tilde s^2} \alpha^2(x')\right)
\right) dx'.
\end{align}
For the first and the last summand in the first factor, $\frac1s \gap(x')$ and $\frac1{s^2}\alpha^2(x')$,
we use the $1$ in the minimum on the right.
For the second summand on the left, $ s$, 
we use the second term in the minimum for the pairing.
Using the $L^1$-convolution estimate and (\ref{pf quadratic term alpha and alpha}) we can control (\ref{pf quadratic term last}) by
\begin{align*}
 \frac1\alpha \left( \frac 1s + \frac s{\tilde s}\right) \int \gap(x')\,dx'
 + \frac{s \tilde s}\alpha r^{d-1}+ \left( \frac{\alpha s}{ \tilde s^2} + \frac \alpha{ s^2}\right) \frac1\tau \int_0^\tau\int\eta\,d\mu_h.
\end{align*}
By Corollary \ref{lem_local_estimates_multiphase} we can estimate $\int \gap(x')\, dx'$ as desired
and thus obtain (\ref{decayofrem_pf}) by choosing $\tilde s =\alpha^{\frac23} \ll1$ and then $s = \alpha^{\frac49} \ll1$.
\end{proof}

\begin{proof}[Proof of Lemma \ref{lem_rough_error_in_velocity}]
Thanks to the convergence assumption
(\ref{conv_ass}), we can apply Proposition \ref{surf_prop_integrated_in_time}.
Using the Euler-Lagrange equation (\ref{ELG}) for $\chi^n$ and (\ref{surf_conv_ass3_spacetime}),
we can identify the first term on the left-hand side as the limit of the first variation 
of the dissipation functional as $h\to0$. Following Step 1 of the proof of
Lemma \ref{lem_1st_var_of_-diss} and then estimating directly as in Step 3, but for $\xi$
instead of $\eta$, we obtain
\begin{align*}
 &c_0\sum_{i,j} \sigma_{ij}\int_0^T \int 
\left( \nabla\cdot \xi - \nu_i\cdot \nabla \xi\,\nu_i\right) \frac12 
\left(\left|\nabla\chi_i\right| + \left|\nabla\chi_j\right| - \left|\nabla(\chi_i+\chi_j)\right|\right) dt\\
&= \lim_{h\to0} \sum_{i,j} \sigma_{ij}\sum_{n=1}^N \int 
\left[ G_{h/2}\ast (\chi_i^n-\chi_i^{n-1})\right] 
\overline \xi^n \cdot 
\left[
\left(
\h \nabla G_{h/2}
\right)
\ast\chi_j^n \right]
dx.
\end{align*}
Using the Cauchy-Schwarz inequality, for any pair $i,j$ we have
\begin{align*}
 &\left| 
 \sum_{n=1}^N \int 
\left[ G_{h/2}\ast (\chi_i^n-\chi_i^{n-1})\right]
\overline \xi^n \cdot \left[
\left(
\h \nabla G_{h/2}
\right)
\ast\chi_j^n \right] 
dx
\right|\\
&\lesssim 
\|\xi \|_{\infty}\! \left(\sum_{n=1}^N\frac1\h \int \eta 
\left[ G_{h/2}\ast (\chi_i^n-\chi_i^{n-1})\right]^2 dx\right)^{\frac12} 
\left(h\sum_{n=1}^N \frac1\h\int \eta 
\left[ \h \nabla G_{h/2} \ast\chi_j^n \right]^2 dx \right)^{\frac12}.
\end{align*}
The first right-hand side factor is bounded by $\iint \eta \,d\mu_h$.
As in Step 1 in the proof of Lemma \ref{lem_1st_var_of_-diss_K_h},
the second right-hand side factor can be controlled by
\begin{align*}
 h\sum_{n=1}^N \frac1\h\int \eta \left[ \left(1-\chi_j^n\right) G_h\ast \chi_j^n
	      + \chi_j^n\, G_h\ast\left(1-\chi_j^n \right) \right] dx 
\to 2c_0  \int_0^T \int \eta \left| \nabla \chi_j\right| dt,
\end{align*}
as $h\to 0$, where we used Lemma \ref{la_impl_conv_ass} to pass to the limit.
Thus, using Young's inequality, we have
\begin{align*}
&\left| 
\sum_{i,j} \sigma_{ij}
 \int_0^T \int \left(\nabla\cdot\xi - \nu_i \cdot \nabla \xi \,\nu_i \right) 
\left(\left|\nabla\chi_i\right| + \left|\nabla\chi_j\right| - \left|\nabla(\chi_i+\chi_j)\right|\right)  dt
\right|\\
&\lesssim  \|\xi\|_\infty\left(
\frac1\alpha \sum_{i=1}^\numphases \int_0^T \int \eta \left| \nabla \chi_i\right| dt
+\alpha \iint \eta \, d\mu\right).
\end{align*}
To estimate the second term in the lemma, note that by Young's inequality we have
$$|\xi\cdot \nu_i \, V_i| \leq \|\xi\|_\infty \eta \left( \frac1\alpha V_i^2 + \alpha\right).$$
Integrating w.\ r.\ t.\ $\left|\nabla\chi_i\right|dt$ yields
\begin{align*}
 \left|\int_0^T \int \xi \cdot \nu_i \, V_i 
\left| \nabla \chi_i \right| dt\right|
 \leq \|\xi\|_\infty \left( \int_0^T \int \eta\, V_i^2 
\left|  \nabla\chi_i\right|dt+
\int_0^T \int \eta \left| \nabla\chi_i \right|dt\right),
\end{align*}
which concludes the proof.
\end{proof}
%
%
%
%
\section{Convergence}\label{sec:conv}
In Section \ref{sec:curvature}, we identified the limit of the first variation of the energy;
in Section \ref{sec:velocity}, we identified the limit of first variation of the metric term up to an error
that measures the local approximability by a half space.
In this section, we show by soft arguments from Geometric Measure Theory that this error can be made
arbitrarily small.
Before that, we will state the main ingredients of the proof here.
%
%
\begin{defn}\label{def_covering}
Given $r>0$, we define the covering
\begin{align*}
 \B_r := \left\{ B_r(i) \colon i\in \L_r \right\}
\end{align*}
of $\torus$, where $\L_r = \torus \cap \frac r{\sqrt{d}} \Z^d $ is a regular grid of midpoints on $\torus$.
By construction, for each $n\geq 1$ and each $r>0$, the covering
\begin{align}\label{finiteoverlap}
 \left\{ B_{nr}(i) \colon i \in \L_r\right\} \quad \text{is locally finite,}
\end{align}
in  the sense that for each point in $\torus$, the number of balls containing
this point is bounded by a constant $c(d,n)$ which is independent of $r$.
For given $\delta>0$ and $\chi: \torus \to \{0,1\}\in BV$, we define 
$\B_{r,\delta}$ to be the subset of $\B_r$ consisting of all balls $B$ such that
the following two conditions hold:
\begin{align}
  \inf_{\nu^\ast} 
\int& \eta_{2B} \left| \nu-\nu^\ast\right|^2\left| \nabla \chi \right|
 \leq \delta \, r^{d-1}\label{goodball_tilt}\quad \text{and}\\
\int_{2B}& \left| \nabla \chi \right|
 \geq \frac12 \omega_{d-1} (2r)^{d-1},\label{goodball_surf}
\end{align}
where $\eta_{2B}$ is a cut-off for $2B$.
\end{defn}
%
%
\begin{lem}
\label{lem_approximate_normal_L2}
For every $\varepsilon>0$ and $\chi: \torus \to \{0,1\}$,
there exists an $r_0>0$ such that for all $r\leq r_0$ there exist unit vectors 
$\nu_B\in\sphere$ such that
\begin{align*}
 \sum_{B\in\B_r}
\frac12\int \eta_{2B}\left| \nu-\nu_B\right|^2\left| \nabla \chi \right|
\lesssim \varepsilon^2 \int\left| \nabla \chi\right|.
\end{align*}
\end{lem}
%
%
The following lemma will be used to control the error terms obtained in
Section \ref{sec:velocity} on the ``bad'' balls $B\in\B_r-\B_{r,\delta}$.
\begin{lem}\label{lem_badballs}
For any $\delta >0$ and any
 $\chi\colon \torus \to \{0,1\}\in BV$, we have
\begin{align*}
 \lim_{r\to0 } \sum_{B\in\B_r-\B_{r,\delta}} \;\int_{2B} \left|\nabla\chi\right| =0.
\end{align*}
\end{lem}
%
%
In a rescaled version, the following lemma can be used to control the error terms 
on the ``good'' balls $B\in \B_{r,\delta}.$
\begin{lem}\label{lem_goodballs}
Let $\eta$ be a radially symmetric cut-off for the unit ball $B$.
Then for any  $\varepsilon>0$ there exists $\delta = \delta(d,\varepsilon)>0$ such 
that for any $\chi\colon \torus \to \{0,1\}$
with
\begin{align}\label{goodballs_premiss}
 \int\eta \left|\nu-e_1 \right|^2 \left|\nabla\chi \right| \leq \delta^2
\end{align}
there exists a half space $\chi^\ast$ in direction $e_1$ such that
\begin{align}\label{goodballs_claim}
 \left| \int_B \left( \left|\nabla\chi \right|-\left|\nabla\chi^\ast \right|\right)\right|
\leq \varepsilon^2, \quad
\int_B \left| \chi-\chi^\ast\right|dx\leq \varepsilon^2.
\end{align}
\end{lem}
%
%
\begin{lem}\label{lem_onehalfspace}
Let $\eta$ be a cut-off for the unit ball $B$.
 Then for any $\varepsilon>0$ there exists $\delta=\delta(d,\numphases,\varepsilon)>0$
 such that for any $\chi\colon \torus \to \{0,1\}^\numphases$ with $\sum_{i}\chi_i=1$, the following
 statement holds:
Whenever we can approximate each normal separately, i.\ e.\
 \begin{align*}
  \sum_{i=1}^\numphases 
  \inf_{\nu_i^\ast} \frac12 \int \eta \left| \nu_i - \nu_i^\ast \right|^2 \left| \nabla \chi_i\right|
    \leq \delta^2,
 \end{align*}
then we can do so with one normal $\nu^\ast\in \sphere$ and its inverse $-\nu^\ast$:
\begin{align*}
 \min_{i\neq j}
  \inf_{\nu^\ast} \Bigg\{ 
  \sum_{k\notin \{i,j\}}\int_B \left| \nabla \chi_k\right|
  +\frac12 \int_B\left| \nu_{i} - \nu^\ast \right|^2 \left| \nabla \chi_{i}\right|
  +\frac12 \int_B \left| \nu_{j} + \nu^\ast \right|^2 \left| \nabla \chi_{j}\right|
 \Bigg\}
 \leq \varepsilon^2.
\end{align*}
\end{lem}
\subsection{Proof of Theorem \ref{thm1}}
%
%
Using Proposition \ref{prop_velocity_good_balls} and the lemmas from above, we can give the proof
of the main result.
The proof consists of three steps:
\begin{enumerate}
 \item Post-processing Propositions  \ref{surf_prop_integrated_in_time} and \ref{prop_velocity_good_balls},
       using the Euler-Lagrange equation (\ref{ELG}) and by making the half space time-dependent,
 \item Estimates for fixed time and
 \item Integration in time.
\end{enumerate}
\begin{proof}[Proof of Theorem \ref{thm1}]
\step{Step 1: Post-processing Propositions  \ref{surf_prop_integrated_in_time} 
and \ref{prop_velocity_good_balls}.} 
Let us first link the results we obtained in Sections \ref{sec:curvature} and \ref{sec:velocity}.
For any fixed vector $\nu^\ast\in \sphere$ and any test function $\xi_B\in C_0^\infty((0,T)\times B, \R^d)$,
supported in a space-time cylinder $(0,T)\times B$, we claim
\begin{align}\label{error_in_velocity_postprocessed}
 &\left| \sum_{i,j} \sigma_{ij} \int_0^T \int \left(\nabla\cdot\xi_B - \nu_i \cdot \nabla \xi_B \,\nu_i
- 2\,\xi_B \cdot \nu_i \,V_i \right) 
\left(\left|\nabla\chi_i\right| + \left|\nabla\chi_j\right| - \left|\nabla(\chi_i+\chi_j)\right|\right) 
dt\right|\\
&\lesssim \|\xi_B\|_{\infty}\left[ 
  \left(\min_{i\neq j}\int_0^T
 \left( \frac1{\alpha^2}\E^2_{ij}(\nu^\ast,t) + \alpha^{\frac19} r^{d-1} \right)dt\right) \wedge
  \left(\frac1\alpha \sum_{i=1}^\numphases \int_0^T \int_B 
\left| \nabla\chi_i\right|dt\right) \right.\notag\\
 &\qquad \qquad \qquad\qquad \qquad \qquad\qquad \qquad + \left.\alpha^{\frac19}\iint \eta_B\, d\mu + \alpha \sum_{i=1}^\numphases
 \int_0^T \int \eta_B  \,V_i^2
\left| \nabla\chi_i\right|dt \right],\notag
\end{align}
where $\E^2_{ij}$ is defined via
\begin{align*}
\E^2_{ij}(\nu^\ast,t)
 := 
  \sum_{k\notin \{i,j\}}\int \eta_{2B} \left| \nabla \chi_{k}(t)\right|
 +&\int \eta_{2B}  \left| \nu_i(t)-\nu^\ast\right|^2 \left|\nabla \chi_i(t) \right| 
 + \int \eta_{2B}  \left| \nu_j(t)+\nu^\ast\right|^2 \left|\nabla \chi_j(t) \right|\\
 +\inf_{\chi^\ast}
\Bigg\{
&\left|\int \eta_B \left(\left| \nabla \chi_{i}(t) \right|
 -\left| \nabla \chi^\ast\right| \right) \right|
  + \frac1r \int_{2B} \left|\chi_{i}(t)-\chi^\ast\right|dx\\
+&\left|\int \eta_B \left(\left| \nabla \chi_{j}(t) \right|
 -\left| \nabla \chi^\ast\right| \right)\right|
+ \frac1r \int_{2B} \left|\chi_{j}(t)-\left(1-\chi^\ast\right)\right|dx \Bigg\}.
\end{align*}
The infimum is taken over all half spaces
$\chi^\ast = \chara_{\{ x\cdot \nu^\ast > \lambda \}}$ 
in direction $\nu^\ast$.\\
Argument for (\ref{error_in_velocity_postprocessed}):
By symmetry, we may assume w.\ l.\ o.\ g.\ that the minimum on the right-hand side of 
(\ref{error_in_velocity_postprocessed}) is realized for $i=1$ and $j=2$.
The Euler-Lagrange equation (\ref{ELG}) of the minimizing movements interpretation (\ref{MMinterpretation})
links Proposition \ref{surf_prop_integrated_in_time} with the metric term:
\begin{align*}
\lefteqn{\lim_{h\to0} \int_0^T -\delta E_h(\,\cdot\,-\chi^h(t-h)) 
(\chi^h(t),\xi_B(t)) \,dt} \\
&=-c_0\sum_{i,j} \sigma_{ij} \int_0^T \int 
\left( \nabla\cdot \xi_B - \nu_i\cdot \nabla \xi_B\,\nu_i\right) 
\frac12\left(\left|\nabla\chi_i\right| + \left|\nabla\chi_j\right| - \left|\nabla(\chi_i+\chi_j)\right|\right)
dt.
\end{align*}
Before applying the results of Section \ref{sec:velocity}
we  symmetrize the second term on the left-hand side of (\ref{error_in_velocity_local_in_time}):
We claim that we can replace
\begin{align}\label{V1V2}
 \sigma_{12}\left( \int \xi_{B} \cdot \nu^\ast\, V_1 \left|\nabla \chi_1 \right|
 +\int \xi_{B}\cdot (-\nu^\ast)\, V_2 \left|\nabla \chi_2 \right|\right)
\end{align}
which appears on the left-hand side of (\ref{error_in_velocity_local_in_time}) by the symmetrized term
\begin{align}\label{Vi}
 \sum_{i,j} \sigma_{ij} \int \xi_B \cdot \nu_i \, V_i 
\frac12
\left(\left|\nabla\chi_i\right| + \left|\nabla\chi_j\right| - \left|\nabla(\chi_i+\chi_j)\right|\right)
\end{align}
which appears in the weak formulation (\ref{H=v}).
Then using Proposition \ref{prop_velocity_good_balls} and this symmetrization
or the rough estimate Lemma \ref{lem_rough_error_in_velocity} yields (\ref{error_in_velocity_postprocessed}).
Now we show how to replace (\ref{V1V2}) by (\ref{Vi}).\\
We start by noting that the sum in (\ref{Vi}) contains two terms involving only Phases 1 and 2.
The contribution to the sum is
\begin{align*}
 \sigma_{12} \int \xi_{B} \cdot \left( \nu_1\, V_1 +  \nu_2\, V_2\right) 
 \frac12
\left(\left|\nabla\chi_1\right| + \left|\nabla\chi_2\right| - \left|\nabla(\chi_1+\chi_2)\right|\right),
\end{align*}
which can be brought into the form of (\ref{V1V2}) at the expense of an error which is controlled by 
$\|\xi_B\|_\infty$ times
\begin{align*}
    \int \eta_B
   \left| \nu_1 - \nu^\ast \right| \left| V_1 \right| \left| \nabla \chi_1\right| +
    \int \eta_B \left| \nu_2 + \nu^\ast \right| \left| V_2 \right| \left| \nabla \chi_2\right|.
\end{align*}
Note that by Young's inequality we have 
\begin{align*}
\left| \nu_1 - \nu^\ast \right| \left| V_1 \right| 
\leq \frac 1\alpha \left| \nu_1 - \nu^\ast \right|^2
+ \alpha V_1^2,
\end{align*}
so that we can estimate both terms after integration in time by
\begin{align*}
 \frac1\alpha\int_0^T \left(\int \eta_{B}  \left| \nu_1-\nu^\ast\right|^2 \left|\nabla \chi_1 \right| 
 + \int \eta_{B}  \left| \nu_2+\nu^\ast\right|^2 \left|\nabla \chi_2 \right|\right) dt
 + \alpha \sum_{i=1}^\numphases
 \int_0^T \int \eta_B  \,V_i^2
\left| \nabla\chi_i\right|dt.
\end{align*}
We are left with estimating the summands in (\ref{Vi}) with $\{i,j\} \neq \{1,2\}$.
For those terms we can use Young's inequality in the following form
\begin{align*}
 \left| \nu_i\right| \left| V_i\right| \leq \frac1\alpha + \alpha V_i^2
\end{align*}
so that after integration in time these terms are controlled by $\|\xi_B\|_\infty$ times
\begin{align*}
 \frac1\alpha \sum_{i=3}^\numphases \int \eta_{B} \left| \nabla \chi_{i}\right|
+ \alpha \sum_{i=1}^\numphases
 \int_0^T \int \eta_B  \,V_i^2
\left| \nabla\chi_i\right|dt,
\end{align*}
which concludes the argument for the symmetrization and thus for (\ref{error_in_velocity_postprocessed}).\\
Here, we see, why we needed to introduce extra terms in $\E_1$ compared to the terms that were already
present in the definition of $\E_1$ in Section \ref{sec:velocity}.
These different terms are sometimes called \emph{tilt-excess} and \emph{excess energy},
respectively.

\medskip

Now let $\xi\in C_0^\infty((0,T)\times\torus, \R^d)$ be given.
First, we localize $\xi$ in space according to the covering $\B_r$ from Definition \ref{def_covering}.
To do so, we introduce a subordinate partition of unity $\{\varphi_B\}_{B\in \B_r}$ and set
$\xi_B:= \varphi_B \xi$. Then $\xi = \sum_{B\in \B_r} \xi_B$,
$\xi_B\in C_0^\infty(B)$ and $\|\xi_B\|_\infty \leq\|\xi\|_\infty$.
Given a radially symmetric and radially non-increasing cut-off $\eta$ of $B_1(0)$ in $B_2(0)$, for
each ball $B$ in the covering, we can construct a cut-off $\eta_B$ of $B$ in $2B$ by shifting and rescaling.
Given any measurable function $\nu^\ast\colon (0,T)\to \sphere$ and any $\alpha\in(0,1)$ we claim
\begin{align}\label{error_in_velocity}
&\left| \sum_{i,j} \sigma_{ij} \int_0^T \int \left(\nabla\cdot\xi_B - \nu_i \cdot \nabla \xi_B \,\nu_i
- 2\,\xi_B \cdot \nu_i \,V_i \right) 
\left(\left|\nabla\chi_i\right| + \left|\nabla\chi_j\right| - \left|\nabla(\chi_i+\chi_j)\right|\right) 
dt\right|\\
&\lesssim \|\xi\|_{\infty}\left[ 
  \int_0^T
 \left( \frac1{\alpha^2}\E^2_B(\nu^\ast(t),t) + \alpha^{\frac19} r^{d-1} \right)
 \wedge \left(\frac1\alpha\sum_{i=1}^\numphases \int_{B}  \left| \nabla\chi_i\right| \right)dt
 + \alpha^{\frac19}\iint \eta \, d\mu\right.\notag\\
 &\qquad\qquad\qquad \qquad\qquad\qquad\qquad\qquad\qquad\qquad\qquad\qquad\left. 
 + \alpha \sum_{i=1}^\numphases\int_0^T \int \eta  \,V_i^2
\left| \nabla\chi_i\right|dt\right],\notag
\end{align}
where $\E^2_B(\nu^\ast,t):= \min_{i\neq j} \E^2_{ij}(\nu^\ast,t)$ for $\nu^\ast\in\sphere$.\\
Now we give the argument that (\ref{error_in_velocity_postprocessed}) implies (\ref{error_in_velocity}).
We approximate the measurable function $\nu^\ast$ in time by a piecewise constant function.
Let $0 =T_0< \dots < T_M=T$ denote a partition of $(0,T)$ such that the approximation $\nu^\ast_M$
of $\nu^\ast$ is constant on each interval $[T_{m-1},T_m)$.
Since the measures on the left-hand side are absolutely continuous in time, we can approximate
$\xi_B$ by vector fields which vanish at the points $T_m$ and both, the 
curvature and the velocity term converge. 
Therefore, we can apply (\ref{error_in_velocity_postprocessed}) on each time interval $(T_{m-1},T_m)$.
Lebesgue's dominated convergence gives us the convergence of the integral on the right-hand side and thus
(\ref{error_in_velocity}) holds. 
\step{Step 2: Estimates for fixed time.}
Let $t\in(0,T)$ be fixed. We will omit the argument $t$ in the following.
Let $\varepsilon>0$ and let $\delta=\delta(\varepsilon)$ (to be determined later).
Let $\B_{r,\delta}$ be defined as the set of good balls in the lattice:
\begin{align*}
 \B_{r,\delta} := \left\{ B \in \B_r \colon  
 \sum_{i=1}^\numphases\inf_{\nu^\ast} \int \eta_{2B} \left| \nu_i-\nu^\ast\right|^2\left| \nabla \chi_i \right|
 \leq \delta r^{d-1}\text{ and }
\sum_{i=1}^\numphases \,\int_{2B} \left| \nabla \chi_i \right|
 \geq \frac12 \omega_{d-1} (2r)^{d-1} \right\}.
\end{align*}
For $B\in\B_{r,\delta}$, and $i=1,\dots,\numphases$, we denote by $\nu_{B,i}$ 
the vector $\nu^\ast$ for which the 
infimum is attained, so that
\begin{align*}
 \sum_{i=1}^\numphases \frac12 \int  \eta_{2B} \left| \nu_i-\nu_{B,i}\right|^2\left| \nabla \chi_i \right|
 \leq \delta\, r^{d-1}.
\end{align*}
By a rescaling and since $\eta$ is radially symmetric, we can upgrade Lemma \ref{lem_onehalfspace}, so that
for given $\gamma>0$, we can find $\delta=\delta(d,\gamma)>0$ (independent of $\chi$)
and $\nu_B\in\sphere$, such that
\begin{align*}
 \min_{i\neq j}\left\{
  \sum_{k\notin\{i,j\}}\int \eta_B \left| \nabla \chi_k\right|
  +\frac12 \int \eta_B\left| \nu_{i} - \nu_B \right|^2 \left| \nabla \chi_{i}\right|
  +\frac12 \int \eta_B\left| \nu_{j} + \nu_B \right|^2 \left| \nabla \chi_{j}\right|
 \right\}
 \leq \gamma\, r^{d-1}.
\end{align*}
Rescaling Lemma \ref{lem_goodballs}, we can define $\gamma=\gamma(\varepsilon)>0$ and a half space
$\chi^\ast$ in direction $\nu_B$, such that 
\begin{align*}
 \E^2_B(\nu_B,t) \leq \varepsilon^2 r^{d-1}.
\end{align*}
These two steps give us the dependence of $\delta$ on $\varepsilon$.
Using the lower bound on the perimeters on $B\in \B_{r,\delta}(t)$, we obtain
\begin{align*}
  \sum_{B\in \B_{r,\delta}} \left(\frac1{\alpha^2}\E^2_B(\nu_B,t) + \alpha^{\frac19} r^{d-1}\right)
  \lesssim  \sum_{B\in \B_{r,\delta}} \left(\frac1{\alpha^2}\varepsilon^2 +\alpha^{\frac19}\right)r^{d-1}
  \lesssim \left(\frac1{\alpha^2}\varepsilon^2 +\alpha^{\frac19}\right) 
  \sum_{i=1}^\numphases\int \left| \nabla \chi_i\right|.
\end{align*}
Note that for the balls $B\in \B_r-\B_{r,\delta}$, we have by Lemma \ref{lem_badballs}:
\begin{align}\label{badballs_pointwisenocontribution}
 \sum_{B\in \B_r-\B_{r,\delta}} \sum_{i=1}^\numphases \int_{B}\left| \nabla \chi_i\right| 
 \to 0,\quad \text{as } r\to 0.
\end{align}
The speed of convergence depends on $\chi$ and $\varepsilon$ (through $\delta$).
\step{Step 3: Integration in time.}
Using Lebesgue's dominated convergence theorem, we can integrate the pointwise-in-time estimates 
of Step 2.
Recalling the decomposition $\xi=\sum_B \xi_B$ and using the finite overlap 
(\ref{finiteoverlap}), we have
\begin{align*}
&\left| \sum_{i,j} \sigma_{ij} \int_0^T \int \left(\nabla\cdot\xi - \nu_i \cdot \nabla \xi \,\nu_i
- 2\,\xi \cdot \nu_i \,V_i \right) 
\left(\left|\nabla\chi_i\right| + \left|\nabla\chi_j\right| - \left|\nabla(\chi_i+\chi_j)\right|\right)  dt
\right|\\
&\lesssim \sum_{B\in\B_r }\left| 
\sum_{i,j} \sigma_{ij} \int_0^T \int \left(\nabla\cdot\xi_B - \nu_i \cdot \nabla \xi_B \,\nu_i
- 2\,\xi_B \cdot \nu_i \,V_i \right) 
\left(\left|\nabla\chi_i\right| + \left|\nabla\chi_j\right| - \left|\nabla(\chi_i+\chi_j)\right|\right)  dt
\right|\\
&\lesssim \|\xi\|_{\infty}\left[   
  \left(\frac1{\alpha^2}\varepsilon^2 +\alpha^{\frac19}\right) \int_0^T \sum_{i=1}^\numphases
  \int \left| \nabla \chi_i\right|dt
 + \int_0^T \sum_{B\in \B_r-\B_{r,\delta}(t)} \frac1\alpha
 \sum_{i=1}^\numphases \int_{B}  \left| \nabla\chi_i\right| dt\right.\\
 &\qquad\qquad\;\qquad\qquad\qquad\qquad\qquad\qquad\qquad+ \left.\alpha^{\frac19} \iint 
 d\mu + \alpha \sum_{i=1}^\numphases \int_0^T \int V_i^2 \left| \nabla \chi_i\right|dt  \right].
\end{align*}
Since by  the energy-dissipation estimate (\ref{energy_dissipation_estimate})
we have $E(\chi(t)) \leq E_0$ and can control the first term.
By Lebesgue's dominated convergence and (\ref{badballs_pointwisenocontribution}),
the second term vanishes as $r\to0$.
By (\ref{mu finite}) and Proposition \ref{lem_dtX<<DX}, we can handle the last two terms.
Thus we obtain
\begin{align*}
 & \left| \sum_{i,j} \sigma_{ij} \int_0^T \int \left(\nabla\cdot\xi - \nu_i \cdot \nabla \xi \,\nu_i
- 2\,\xi \cdot \nu_i \,V_i \right) 
\left(\left|\nabla\chi_i\right| + \left|\nabla\chi_j\right| - \left|\nabla(\chi_i+\chi_j)\right|\right)  dt
\right|\\
&\qquad\qquad\lesssim \|\xi\|_{\infty}\left(\frac1{\alpha^2}\varepsilon^2 E_0 T + \alpha^{\frac19}(1+T)E_0\right).
\end{align*}
Taking first the limit $\varepsilon$ to zero and then $\alpha$ to zero yields (\ref{H=v}), which concludes the proof of Theorem \ref{thm1}.
\end{proof}
\subsection{Proofs of the lemmas}
%
%
\begin{proof}[Proof of Lemma \ref{lem_approximate_normal_L2}]
Let $\varepsilon>0$ be given and w.\ l.\ o.\ g.\ $\int \left| \nabla \chi\right|>0$.
Since the normal $\nu$ is measurable, we can approximate it by a continuous vector field
$\tilde \nu \colon \torus \to \overline B$ in the sense that
\begin{align*}
 \sum_{B\in\B_r}
\frac12\int_{B} \left| \nu-\tilde \nu\right|^2\left| \nabla \chi \right|
\lesssim 
\int \left| \nu-\tilde \nu\right|^2\left| \nabla \chi \right|
\leq \varepsilon^2 \int \left| \nabla \chi\right|,
\end{align*}
where we have used the finite overlap property (\ref{finiteoverlap}).
Since $\tilde\nu$ is continuous,
we can find $r_0>0$ such that for any $r\leq r_0$ we can find vectors $\tilde \nu_B$ with $\left| \tilde\nu_B\right|\leq 1$ with
\begin{align*}
 \sum_{B\in\B_r}
\frac12\int_{B} \left| \tilde\nu-\tilde \nu_B\right|^2\left| \nabla \chi \right|
\leq \varepsilon^2 \int\left| \nabla \chi\right|.
\end{align*}
The only missing step is to argue that we can also choose $\nu_B\in \sphere$.
If $\left| \tilde \nu_B\right|\geq 1/2$, this is clear
because then 
$\left| \nu - \tilde \nu_B/|\tilde \nu_B| \right|
\leq 2 \left| \nu - \tilde \nu_B \right|$. 
If $\left| \tilde \nu_B\right|\leq 1/2$,
we have the easy estimate
\begin{align*}
 \left|\nu - \tilde\nu_B\right| \geq \frac12 
\geq \frac14 \left( \left|  \nu\right|+\left| \nu_B\right|\right)
\geq \frac14 \left| \nu - \nu_B\right|
\end{align*}
for any $\nu_B\in \sphere$.
\end{proof}
%
%
\begin{proof}[Proof of Lemma \ref{lem_badballs}]
 Let $\varepsilon,\delta>0$ be arbitrary.
Note that a ball in $\B_r-\B_{r,\delta}$ satisfies
\begin{align}
  \inf_{\nu^\ast} 
\int_{2B}& \left| \nu-\nu^\ast\right|^2\left| \nabla \chi \right|
 \geq \delta r^{d-1}\label{badball_tilt}\quad \text{or}\\
\int_{2B}& \left| \nabla \chi \right|
 \leq \frac12 \omega_{d-1} r^{d-1}.\label{badball_surf}
\end{align}
\step{Step 1: Balls satisfying (\ref{badball_tilt}).}
By Lemma \ref{lem_approximate_normal_L2}, for any $\gamma>0$, to be chosen later, there exists
 $r_0 = r_0(\gamma,\delta,\chi)>0$, such that for every $r\leq r_0$
we can find vectors $\nu_B\in \sphere$ such that 
\begin{align}\label{appr_normal_in_L2}
  \sum_{B\in\B_r} \;\int_{2B}\left|  \nu - \nu_B\right|^2
 \left|\nabla\chi\right| \lesssim \gamma\delta   \int \left|\nabla \chi\right|.
\end{align}
Thus we have
\begin{align}\label{number_of_bad_balls_tilt}
 \# \bigg\{B \colon \int_{2B}\left|  \nu - \nu_B\right|^2
 \left|\nabla\chi\right| \geq \delta r^{d-1} \bigg\}
\leq \sum_B \frac{1}{\delta r^{d-1}}\int_{2B}\left|  \nu - \nu_B\right|^2
 \left|\nabla\chi\right|
\overset{(\ref{appr_normal_in_L2})}{\lesssim} \frac{\gamma }{r^{d-1}}
 \int \left|\nabla \chi\right|.
\end{align}
Using that the covering is locally finite and De Giorgi's structure result, we have
\begin{align*}
 \sum_{B: (\ref{badball_tilt})}\, \int_{2B} \left|\nabla \chi\right|
\lesssim& \int_{\bigcup_{(\ref{badball_tilt})} 2B }  \left|\nabla \chi\right|
=\H^{d-1}\bigg( \partial^\ast \Omega \cap \bigcup_{(\ref{badball_tilt})} 2B\bigg).
 \end{align*}
Since $\partial^\ast \Omega$ is rectifiable, we can find Lipschitz graphs
$\Gamma_n$ such that $ \partial^\ast \Omega \subset \bigcup_{n=1}^\infty \Gamma_n.$
Therefore,
\begin{align*}
\H^{d-1}\bigg( \partial^\ast \Omega \cap \bigcup_{(\ref{badball_tilt})} 2B\bigg)
\leq& \sum_{n=1}^N 
\H^{d-1}\bigg( \Gamma_n \cap \bigcup_{(\ref{badball_tilt})} 2B\bigg)
+ \H^{d-1}\bigg( \partial^\ast \Omega -\bigcup_{n\leq N} \Gamma_n\bigg).
\end{align*}
Note that for any ball $B$
\begin{align*}
 \H^{d-1}\left( \Gamma_n \cap 2B\right) 
\lesssim \left( 1+ \lip\; \Gamma_n\right)r^{d-1}
\end{align*}
and thus
\begin{align*}
 \H^{d-1}\bigg( \Gamma_n \cap \bigcup_{(\ref{badball_tilt})} 2B\bigg)
\leq \sum_{B: (\ref{badball_tilt})}  \H^{d-1}\left( \Gamma_n \cap 2B\right) 
\lesssim \left(1+\max_{n\leq N}\lip\; \Gamma_n \right) r^{d-1} 
\# \left\{ B \colon (\ref{badball_tilt})  \right\}.
\end{align*}
Using (\ref{number_of_bad_balls_tilt}), we have
\begin{align*}
 \sum_{B: (\ref{badball_tilt})} \int_{2B} \left|\nabla \chi\right|
\lesssim N\left(1+\max_{n\leq N}\lip\; \Gamma_n \right)
\gamma  \int \left|\nabla \chi\right|
+ \H^{d-1}\bigg( \partial^\ast \Omega -\bigcup_{n\leq N} \Gamma_n\bigg).
\end{align*}
Now, choose $N$ large enough such that
\begin{align*}
 \H^{d-1}\bigg( \partial^\ast \Omega -\bigcup_{n\leq N} \Gamma_n\bigg) \leq 
\varepsilon^2.
\end{align*}
Then, choose $\gamma>0$ small enough, such that
\begin{align*}
 N\left(1+\max_{n\leq N}\lip\; \Gamma_n \right)
\gamma  \int \left|\nabla \chi\right| \leq \varepsilon^2.
\end{align*}
\step{Step 2: Balls satisfying (\ref{badball_surf}).}
By De Giorgi's structure theorem (Theorem 4.4 in \cite{giusti1984minimal}), we may restrict to balls $B$ which in addition satisfy
$\partial^\ast \Omega \cap 2B \neq \emptyset$
and pick $x\in \partial^\ast \Omega \cap 2B$.
Note that since $B$ has radius $r$ we have
\begin{align*}
 B_{2r}(x) \subset 4B \subset B_{6r}(x).
\end{align*}
Therefore, if (\ref{badball_surf}) holds,
\begin{align*}
\int_{B_{2r}(x)} \left| \nabla \chi \right|
\leq \int_{4B} \left| \nabla \chi \right|
 \leq \frac12 \omega_{d-1} (2r)^{d-1}.
\end{align*}
For $x\in \partial^\ast \Omega$ we have
\begin{align*}
 \liminf_{r\to 0} \frac1{r^{d-1}} \int_{B_r(x)} \left| \nabla \chi \right|
\geq \omega_{d-1}
\end{align*}
and thus in particular
\begin{align*}
 \chara\bigg(\bigg\{ x\in \partial^\ast \Omega \colon 
\int_{B_r(x)} \left| \nabla \chi \right|\leq \frac12 \omega_{d-1} r^{d-1}
\bigg\}\bigg) \to 0
\end{align*}
pointwise as $r\to 0$. By De Giorgi's structure theorem (Theorem 4.4 in \cite{giusti1984minimal}), the finite overlap and 
Lebesgue's dominated convergence theorem, we thus have
\begin{align*}
 \sum_{B:  (\ref{badball_surf})}\,
\int_{2B} \left| \nabla\chi\right|
\lesssim \H^{d-1}\Big( \partial^\ast \Omega \cap 
\bigcup_{B:(\ref{badball_surf})}2 B\Big)
\to 0
\end{align*}
as $r\to0$.
\end{proof}

%
%
\begin{proof}[Proof of Lemma \ref{lem_goodballs}]
Let us first prove that for any $\chi$ satisfying (\ref{goodballs_premiss}), we have
\begin{align}\label{goodballs_2}
 \left( 1-\delta\right) \int \eta \left| \nabla \chi\right|
\leq \left|\int  \chi \,\nabla \eta\, dx\right| + \delta.
\end{align}
Indeed, we have
\begin{align*}
 \left| \int \eta\, \nu \left| \nabla \chi\right|\right|
\geq \left| \int \eta\, e_1\left| \nabla \chi\right|\right|
 - \left| \int \eta \left(\nu-e_1\right) \left| \nabla \chi\right|\right|
 = \int \eta \left| \nabla\chi\right|
 - \left| \int \eta \left(\nu-e_1\right) \left| \nabla \chi\right|\right|.
\end{align*}
By Young's inequality we have $\left| \nu-e_1\right| \leq \frac1\delta\left| \nu-e_1\right|^2 + \delta$, so that by (\ref{goodballs_premiss}) we can estimate
the last right-hand side term
\begin{align*}
 \left| \int \eta \left(\nu-e_1\right) \left| \nabla \chi\right|\right|
\leq  \int \eta \left|\nu-e_1\right| \left| \nabla \chi\right|
\overset{(\ref{goodballs_premiss})}{\leq} \delta + \delta \int \eta \left| \nabla \chi\right|.
\end{align*}
Therefore
\begin{align*}
 \left| \int \eta\, \nu \left| \nabla \chi\right|\right|
\geq \left( 1-\delta\right) \int \eta \left| \nabla \chi\right| - \delta,
\end{align*}
which is (\ref{goodballs_2}).\\
Now we give an indirect argument for the lemma.
Suppose there exists an $\varepsilon>0$ and a sequence $\{\chi_n\}_n$ such that
\begin{align}\label{goodballs_1}
 \int\eta \left|\nu_n-e_1 \right|^2 \left|\nabla\chi_n \right| \leq \frac1{n^2}
\end{align}
while for all half spaces $\chi^\ast$ in direction $e_1$,
\begin{align}\label{goodballs_contra}
 \int_B  \left|\nabla\chi_n \right|
\geq \varepsilon^2 + \int_B\left|\nabla\chi^\ast \right|, \quad
 \int_B  \left|\nabla\chi^\ast \right|
\geq \varepsilon^2 + \int_B\left|\nabla\chi_n \right|,
 \quad \text{or}\quad
 \int_B \left| \chi_n-\chi^\ast\right|dx\geq \varepsilon^2.
\end{align}
By (\ref{goodballs_1}), we can use (\ref{goodballs_2}) for $\chi_n$ and obtain:
\begin{align*}
 \int \eta \left| \nabla \chi_n\right| \leq \frac{1}{1-1/n} 
\left( \int \left| \nabla \eta\right|dx + \frac1n \right)\quad 
\text{stays bounded as } n\to \infty.
\end{align*}
Therefore, after passage to a subsequence and a diagonal argument to exhaust the open ball
$\{\eta>0\}$, we find $\chi$ such that
\begin{align}\label{goodballs_3}
 \chi_n \to \chi \quad \text{pointwise a.\ e.\ on } \{\eta >0\}.
\end{align}
By (\ref{goodballs_1}) we have
\begin{align*}
2\int \eta \left| \nabla \chi_n \right| - 2 \int \nabla \eta \cdot e_1\; \chi_n\, dx
= \int \eta \left| \nu_n - e_1\right|^2 \left|\nabla \chi_n\right|
\leq \frac1{n^2} \to 0.
\end{align*}
Since the first term on the left-hand side
 is lower semi-continuous and the second one is continuous,
we can pass to the limit in the above inequality and obtain
\begin{align*}
 \int \eta \left| \nu-e_1\right|^2\left|\nabla \chi\right|
= 2\int \eta \left| \nabla \chi \right| - 2 \int \nabla \eta \cdot e_1 \,\chi \,dx
\leq0.
\end{align*}
Hence
\begin{align*}
 \nu = e_1 \quad \left| \nabla\chi\right| \text{-a.\ e. in }\{\eta>0\}.
\end{align*}
A mollification argument shows that there exists a half space $\chi^\ast$ in
direction $e_1$ such that
\begin{align*}
 \chi=\chi^\ast \quad\text{a.\ e. in } \{\eta>0\}.
\end{align*}
Because of (\ref{goodballs_3}), this rules out
\begin{align*}
 \int_B \left| \chi_n - \chi^\ast\right| \geq \varepsilon^2
\end{align*}
on the one hand. 
On the other hand, by lower semi-continuity of the perimeter, also
\begin{align*}
  \int_B  \left|\nabla\chi^\ast \right|
\geq \varepsilon^2 + \int_B\left|\nabla\chi_n \right|
\end{align*}
is ruled out.
To obtain a contradiction also w.\ r.\ t.\ the first statement in (\ref{goodballs_contra}),
let $\tilde \eta \leq  \eta$ be a cut-off for $B$
in $(1+\delta)B$. Since (\ref{goodballs_2}) holds also for $\tilde \eta$ instead of $\eta$,
we have
\begin{align*}
 \varepsilon^2 + \int_B \left| \nabla \chi^\ast \right|
\overset{(\ref{goodballs_contra})}{\leq} &\int_B \left| \nabla \chi_n \right|
\leq \int \tilde \eta\left| \nabla \chi_n \right|
\overset{(\ref{goodballs_2})}{\leq} \frac{1}{1-1/n}
\left(  \left|\int \chi_n\, \nabla \tilde \eta  \,dx \right|+ 
\frac1n \right)\\
\overset{(\ref{goodballs_3})}{\to} & \left|\int \chi^\ast \,\nabla \tilde \eta \,dx \right|
= \left|\int \tilde \eta \,\nabla \chi^\ast \right|
\leq \int_{(1+\delta)B} \left|\nabla \chi^\ast \right|.
\end{align*}
Since $\chi^\ast$ is a half space and therefore has no mass on $\partial B$, we
have
\begin{align*}
 \int_{(1+\delta)B} \left|\nabla \chi^\ast \right|
 \to \int_B \left| \nabla \chi^\ast \right|,\quad \text{as } \delta\to 0,
\end{align*}
which is a contradiction.
\end{proof}
%
%
\begin{proof}[Proof of Lemma \ref{lem_onehalfspace}]
We give an indirect argument. 
Assume there exists a sequence of characteristic functions
 $\{\chi^n\}_n$ with $\sum_i \chi_i^n =1$ a.\ e., a number $\varepsilon>0$ such that we can
find approximate normals $\nu_i^{\ast n}\in \sphere$ with
 \begin{align*}
  \sum_{i=1}^\numphases 
  \frac12 \int \eta \left| \nu_i^n - \nu_i^{\ast n} \right|^2 \left| \nabla \chi_i^n\right|
    \leq \frac1{n^2}
 \end{align*}
while for all $\nu^\ast\in \sphere$, $n\in \N$ and any pair of indices $i\neq j$, we have
\begin{align}\label{Xi->X_contradiction}
 \sum_{k\notin \{i,j\}}\int_B \left| \nabla \chi^n_k\right|
  +\frac12 \int_B\left| \nu^n_{i} - \nu^\ast \right|^2 \left| \nabla \chi^n_{i}\right|
  +\frac12 \int_B \left| \nu^n_{j} + \nu^\ast \right|^2 \left| \nabla \chi^n_{j}\right|
  \geq \varepsilon^2.
\end{align}
Since $\sphere$ is compact, we can find vectors $\nu^\ast\in \sphere$, such that,
after passing to a subsequence if necessary, $\nu^{\ast n}_i \to \nu^\ast_i$ as $n\to\infty$.
Following the lines of the proof of Lemma \ref{lem_goodballs}, we find
\begin{align*}
 \int \eta \left| \nabla \chi^n_i\right| \leq \frac{1}{1-1/n} 
\left( \int \left| \nabla \eta\right|dx + \frac1n \right)\quad 
\text{stays bounded as } n\to \infty
\end{align*}
so that there exist $\chi_i\in\{0,1\}$ with
\begin{align}\label{onehalfspacepf1}
 \chi_i^n \to \chi_i \quad \text{pointwise a.\ e.\ on } \{\eta >0\}
\end{align}
and
\begin{align*}
 \frac12 \int \eta \left| \nu_i - \nu_i^{\ast} \right|^2 \left| \nabla \chi_i\right|
 \leq \liminf_{n\to \infty}
 \frac12 \int \eta \left| \nu_i^n - \nu_i^{\ast n} \right|^2 \left| \nabla \chi_i^n\right|=0.
\end{align*}
Therefore, $\nu_i=\nu_i^\ast$ $\left| \nabla \chi_i\right|$- a.\ e. and each $\chi_i=\chi_i^\ast$ is a
half space in direction $\nu_i^\ast$.
Continuing in our setting now, we note that the condition $\sum_i \chi_i^n=1$ carries
over to the limit: $\sum_i \chi_i^\ast=1$. 
Therefore there exists a pair of indices $i\neq j$ (w.\ l.\ o.\ g.\ $i=1$, $j=2$) such that for all $k\geq 3$
$\chi_k^\ast =0$ in $B$.
Then the other two half spaces are complementary, $\chi_2^\ast = \left( 1-\chi_1^\ast\right) $
and in particular $\nu_1^\ast=-\nu_2^\ast=: \nu^\ast$.
As in the proof of Lemma \ref{lem_goodballs}, we have
\begin{align*}
 \int_B \left| \nabla \chi_i^n\right| \to \int_B \left| \nabla \chi_i^\ast\right|.
\end{align*}
Together with (\ref{onehalfspacepf1}), we can take the limit $n\to \infty$ in
(\ref{Xi->X_contradiction}) and obtain
\begin{align*}
 \sum_{k\geq 3}\int_B \left| \nabla \chi^\ast_1\right|
  +\frac12 \int_B\left| \nu^\ast_{1} - \nu^\ast \right|^2 \left| \nabla \chi^\ast_{1}\right|
  +\frac12 \int_B \left| \nu^\ast_{2} + \nu^\ast \right|^2 \left| \nabla \chi^\ast_{2}\right|
  \geq \varepsilon^2,
\end{align*}
which is a contradiction since the left-hand side vanishes by construction.
\end{proof}

\printbibliography

\end{document}